\documentclass[reqno,english,12pt]{amsart}

\usepackage[margin=0.95in]{geometry}

\usepackage{amstext,amsthm,amsfonts,amssymb,euscript,mathrsfs,graphics,color,xcolor,amsmath,amssymb,latexsym}

\usepackage{amsthm}
\newtheorem{definition}{Definition}

\usepackage{hyperref}

\makeatletter

\newcommand{\lyxmathsym}[1]{\ifmmode\begingroup\def\b@ld{bold}
	\text{\ifx\math@version\b@ld\bfseries\fi#1}\endgroup\else#1\fi}

\numberwithin{equation}{section}
\numberwithin{figure}{section}

\theoremstyle{plain}
\newtheorem{thm}{\protect\theoremname}[section]

\theoremstyle{plain}
\newtheorem{cor}[thm]{\protect\corollaryname}

\theoremstyle{plain}
\newtheorem{rem}[thm]{\protect\remarkname}

\theoremstyle{plain}
\newtheorem{lem}[thm]{\protect\lemmaname}

\theoremstyle{plain}
\newtheorem{prop}[thm]{\protect\propositionname}

\makeatother

\def\wtF{\widetilde{{\mathcal{F}}}}
\def\whF{\widehat{{\mathcal{F}}}}
\def\jx{\langle x \rangle}
\def\jy{\langle y \rangle}
\def\jk{\langle k \rangle}

\def\wt{\widetilde}
\def\#{\sharp}

\def\R{\mathbb{R}}
\def\C{\mathbb{C}}

\def\K{\mathcal{K}}
\def\e{\epsilon}

\def\jt{\langle t \rangle}
\def\js{\langle s \rangle}

\def\pv{\mathrm{p.v.}}

\def\sign{\mathrm{sgn}}
\def\sgn{\mathrm{sgn}}

\def\what{\widehat}

\usepackage{babel}
\providecommand{\corollaryname}{Corollary}
\providecommand{\lemmaname}{Lemma}
\providecommand{\propositionname}{Proposition}
\providecommand{\remarkname}{Remark}
\providecommand{\theoremname}{Theorem}

\begin{document}

\title[Cubic NLS with a non-generic potential]{On the $1$d cubic NLS with a non-generic potential}

\author{Gong Chen}

\address{University of Kentucky, Department of Mathematics,
719 Patterson Office Tower,
Lexington, KY, 40506-0027, USA.}

\email{g.chen@uky.edu} 

\author{Fabio Pusateri}
\address{University of Toronto, Department of Mathematics, 40 St George Street,
Toronto, ON, M5S 2E4, Canada.}

\email{fabiop@math.toronto.edu}

\begin{abstract}

\small

We consider the $1d$ cubic nonlinear Schr\"odinger equation with an external potential $V$ 
that is non-generic.
Without making any parity assumption on the data,
but assuming that the zero energy resonance of the associated Schr\"odinger operator is either odd or even,
we prove global-in-time quantitative bounds and asymptotics for small solutions. 

First, we use a simple modification of the basis for the distorted Fourier transform (dFT)
to resolve the (possible) discontinuity at zero energy due to the presence of a resonance and the absence of symmetry
of the solution. 
We then use a refined analysis of the low frequency structure
of the (modified) 
nonlinear spectral distribution,
and employ smoothing estimates in the setting of non-generic potentials.

\end{abstract}

\maketitle

\setcounter{tocdepth}{1}

\begin{quote}
\tableofcontents 
\end{quote}

\date{\today}

\section{Introduction}
We consider 
the one dimensional cubic nonlinear Schr\"odinger equation (NLS)
\begin{align}\label{NLSV} 
i\partial_tu - \partial_{xx}u + V(x)u \pm \left|u\right|^{2}u = 0, \qquad  u(0,x)=u_{0}(x),
	\end{align}
for an unknown $u:(t,x) \in \R \times \R \mapsto \C$, with small initial data $u_0 \in L^2(\jx^2 dx) \cap H^1$, 
where $V=V(x)$ is a large external potential.
We assume that the potential is {\it non-generic}, i.e., there is bounded solution $\varphi=\varphi(x)$ 
of $H \varphi = 0$, where $H:=-\partial_x^2 + V(x)$ is the  Schr\"odinger operator associated to $V$.
We do not make any parity assumptions on the initial data,
but assume that the zero energy resonance $\varphi$  is either even or odd
(which is essentially equivalent to assuming that the potential $V$ is even). 
Our main result is the proof of global bounds and asymptotics for solution of \eqref{NLSV}.

It is well known that non-generic potentials pose many challenges in the study 
of nonlinear equations such as \eqref{NLSV}.
These difficulties include the lack of improved local decay 
and related weaker smoothing effects for the linear flow,
as well as the (possible) discontinuity at zero frequency of the distorted Fourier transform associated to 
$H$, which is an important tool in the analysis of \eqref{NLSV} and similar problems.


At the same time, contrary to what the denomination `non-generic' would seem to suggest,
these potentials appear naturally in many physical models. Focusing on one dimensional models, examples include the linearizations around
soliton solutions to the $1$d cubic NLS or the quadratic and cubic Klein-Gordon (KG) equation, see \cite{CGNT}. 
Other examples in the relativistic setting are the linearization around kink solutions to the $\phi^4$ 
model and the sine-Gordon equation.
In all these cases the zero energy resonance has a parity since the potential is even.
Also note that in $1$d the trivial potential $V=0$ is non-generic (with resonance $\varphi(x) \equiv 1$).

In Subsection \ref{ssecideas} we discuss in more details 
some of the difficulties in the treatment of non-generic potentials,
and how we address them in the context of \eqref{NLSV}.
Results similar to the ones in this paper can also be obtained for the Klein-Gordon equations
(see Remarks \ref{remKG3} and \ref{remKG2}),
and other nonlinear evolution equations sharing some basic features with \eqref{NLSV}.
It is natural to regard \eqref{NLSV} as a basic, yet important model to understand nonlinear equations
when the linear operator has a zero energy resonance, 
and outside of symmetry classes. 
Moreover, studying \eqref{NLSV} 
allows us to explore the applicability of methods based on the distorted Fourier transform (dFT)
when this has a discontinuity at zero frequency, i.e., in the case of an odd resonance.
In particular, our results show that the discontinuity of the dFT 
is not an obstruction to approaching questions about the asymptotic stability 
of solitons and kinks under non-symmetric perturbations
when the linearized problem has a non-generic even potential.



\subsection{Some background and literature}
Recall that sufficiently regular solutions of \eqref{NLSV} conserve the mass 
\[M(u) := \int\left|u\right|^{2}\,dx \]
and the total energy (Hamiltonian)
\begin{align}\label{Ham}
H\left(u\right) := \int\frac{1}{2}\left|\partial_xu\right|^{2}
  +\frac{1}{2}V\left|u\right|^{2}\pm\frac{1}{4}\left|u\right|^{4} \, dx.
\end{align}
The Cauchy problem for \eqref{NLSV} with $V=0$ - we will refer to this as the ``free'' or ``flat'' case -
is globally well-posed in $L^2$ and $H^1$, see for example Cazenave-Weissler \cite{CW}.

In what follows we only focus on the question of global-in-time bounds and asymptotics
for \eqref{NLSV} as $|t|\rightarrow \infty$.
The main feature of the cubic nonlinearity is its criticality with respect to scattering:
linear solutions of the Schr\"odinger equation decay at best like $|t|^{-1/2}$ in $L^\infty_x$, 
so that, when evaluating the nonlinearity on linear solutions, one see that $|u|^{2}u \sim |t|^{-1}u$;
the non-integrability of $|t|^{-1}$ 
results in a ``Coulomb''-type contribution of the nonlinear terms which produces 
modified scattering with an additional nonlinear phase correction compared to the linear behavior. 

\smallskip
\noindent
{\it The free case.}
%
%
In the case $V=0$ this problem is well understood. 
Using complete integrability, modified scattering was proven in the seminal work of Deift-Zhou \cite{DZ} 
without size restriction on the solutions;
see also \cite{DZ2}. 
Without making use of complete integrability, 
proofs of modified scattering for small solutions were given by Ozawa \cite{O},
Hayashi-Naumkin \cite{HN}, Lindblad-Soffer \cite{LS}, Kato-Pusateri \cite{KP} and Ifrim-Tataru \cite{IT}.

\smallskip
\noindent
{\it Generic potentials or symmetric solutions.}
Recently, the above results for small solutions have been extended to
\eqref{NLSV} with a generic potential (see Definition \ref{def:generic}) of sufficient regularity and decay 
in the works of Naumkin \cite{N}, Delort \cite{Del} and Germain-Pusateri-Rousset \cite{GPR};
The work \cite{Del} also considered the non-generic cases
of a so-called ``very exceptional'' potential
but under symmetry assumptions (even potential and odd data/solution).
The work of Masaki-Murphy-Segata \cite{MMS} treated the case of a delta potential. 
In our previous work \cite{NLSV} 
we treated general classes of potentials in a weighted $L^1$ space. 
We also analyzed the non-generic problem but with extra symmetry 
assumptions on the data/solution to ensure improved behavior at zero energy.
We also mention the work by Mart\'inez \cite{MM}
who used virial type arguments to prove decay estimates on compact regions 
for some classes of small energy solutions of $L^2$ subcritical NLS equations with potentials.


\smallskip
\noindent
{\it Non-generic potentials: difficulties and related works.}
In all the results mentioned above
one has to appeal to either genericity of the potential or parity of the solutions.
This is done to ensure a better behavior 
of the solution 
at zero frequency, i.e., the vanishing of its dFT at zero, $\wt{u}(t,0)=0$,
or, equivalently, improved local decay and smoothing properties.  
%
%
One crucial issue in the non-generic setting is exactly the non-vanishing of the transform at zero frequency.
An additional difficulty is that the dFT
may 
have a jump discontinuity at zero.
See Subsection \ref{sec:JostSpec} for more information.

Because of these issues, nonlinear problems 
with non-generic potentials in low dimensions and/or with low power nonlinearities
are not very well understood at the present moment.
Nevertheless, important 
recent works have been done in the context of $1d$ Klein-Gordon/relativistic theories.
We refer in particular to the works of Lindblad-L\"uhrmann-Schlag-Soffer \cite{LLSS} 
on KG with non-generic potentials and localized quadratic nonlinearities,
\cite{LLS,LLS1} on the case $V=0$, 
and the work of L\"uhrmann-Schlag \cite{LSch} on the asymptotic 
stability under odd perturbations of the Kink solution for the Sine-Gordon equation.

Loosely speaking, 
so far two successful ways 
have been proposed to handle zero energy resonances.
One approach 
consists in explicitly isolating the influence of the resonance and then 
carefully analyzing dispersive properties for the remaining contributions,
such as in \cite{LLSS}; see also  Krieger-Schlag \cite{KriSch} for wave equations in $3d$ and reference therein.
Another approach,
which goes under the name of ``super-symmetric'' factorization,
is to use factorization properties of the perturbed operator $-\partial_x^2 + V(x) + m^2$ ($m\geq 0$)
to conjugate it to the flat one.
L\"uhrmann-Schlag \cite{LSch} in particular were able to use this to prove a  asymptotic stability
result over the full line for 
the kink of the sine-Gordon equation. 
See also \cite{CGNT,KNS,KMMV}.

%


In this paper we adopt a different approach which does not rely on explicit formulas
and the super-symmetric factorization.
In the most interesting case of an odd resonance, we first use an elementary observation 
to handle the jump discontinuity of the dFT at zero frequency,
and obtain a modified dFT with no discontinuity. 
We then analyze the structure of the nonlinear spectral measure given by the spatial integral of 
the product of four (modified) generalized eigenfunctions, paying particular attention to its low frequency structure.
We then combine this with general versions of smoothing estimates, 
as well as low frequency improved local decay,
in the context of Schr\"odinger flows (and pseudo-differential variants of it) with non-generic potentials.
See Subsection \ref{ssecideas} for more details on the main ideas of the proof.



\subsection{Main result}\label{secmainrest}
In this section, we present our main results. We start with our assumptions on potentials. 
	
\smallskip
\noindent
{\it Assumptions on the potential $V$.}

\setlength{\leftmargini}{1.5em}
\begin{itemize}
\item We assume that $V$ is non-generic (see Definition \ref{def:generic})
and that the zero energy resonance $\varphi$ (see Definition \ref{def:zeroresonance}) is either even or odd;
		
\item We assume that $V$ has no discrete spectrum;
\item We assume that for $\gamma>\frac{5}{2}$
\begin{align}\label{VassumeWei}
\jx^\gamma V \in L^1(\R).
\end{align}
Note that, in particular, we do not assume any regularity on $V$.

\end{itemize}


\begin{thm}\label{mainthm}
Consider the nonlinear Schr\"odinger equation  \eqref{NLSV}  with a potential satisfying the assumptions above.
Then 
there exists $0<\epsilon_{0}\ll 1$ 
such that for all $\varepsilon\leq\epsilon_{0}$ and $u_0$ with
\begin{align}\label{datasmall}
\left\Vert u_{0}\right\Vert _{H^{1}}+\left\Vert xu_{0}\right\Vert _{L^{2}} \leq \varepsilon
\end{align}
the equation \eqref{NLSV} has a unique
global solution $u\in C(\R,H^1(\R))$, with $u(0,x)=u_{0}(x)$;
this solution satisfies the sharp decay rate
\begin{align}\label{main1fdecay}
\left\Vert u(t)\right\Vert_{L^\infty_x} \lesssim \varepsilon \left(1+\left|t\right|\right)^{-\frac{1}{2}}
\end{align}
and has a modified scattering behavior (see Remark \ref{remmodscatt} for more details).

Moreover, if we define the profile of the solution $u$ as
\begin{align}\label{main1prof}
f(t,x) := e^{-itH}u(t,x), \quad H := -\partial_{xx}+V,
\end{align}
then, 
for $\alpha \in (0,1/4)$ 
we have the global bounds 
\begin{align}\label{main1fbounds}
{\big\| (\mathcal{F} f) (t) \big\|}_{L_k^\infty} + (1+|t|)^{-\alpha} 
  {\big\| \partial_{k} (\mathcal{F} f) (t) \big\|}_{L_k^2} \lesssim \varepsilon,
\end{align}
where, in the case of an even, respectively odd, zero energy resonance,
$\mathcal{F}$ denotes the distorted Fourier transform $\wtF$ associated to $H$,
(see the definition \eqref{tildeF} with \eqref{matK} and \eqref{psipm}-\eqref{psipmlim})
respectively 
the `modified' distorted Fourier transform $\mathcal{F}^\sharp$
(see the definition \eqref{sharpF} with \eqref{eq:Ksharp}).
\end{thm}

\begin{rem}
Note that in \eqref{main1fbounds} we are measuring the $L^2_k$ norm of $\partial_k \wtF f$
in the case of an even zero energy resonance,
and of $\partial_k \mathcal{F}^\sharp f$ in the case of an odd resonance.
In both cases, the respective transforms are continuous (Lipschitz) everywhere including zero.
\end{rem}

\begin{rem}[Adding a coefficient]
Our analysis can be directly applied to the equation
\begin{align}
i\partial_tu - \partial_{xx}u + V(x)u \pm b_1(x) \left|u\right|^{2}u +b_2(x) u^3+b_3(x) |u|u^2+b_4(x) |u|^3  = 0
\end{align}
with $b_1(x)-1$, $b_2(x)$, $b_3(x)$ and $b_4(x)$ are
$C^1$ coefficients that converge to $0$ sufficiently fast 
(together with their derivative) as $|x|\rightarrow \infty.$
In fact, all the additional localized cubic terms can be estimated like 
the `regular' nonlinear terms treated in \S\ref{ssecNR2} and \S\ref{ssecRasy}.
\end{rem}

The following remark describes more precisely the asymptotics of the solutions 
constructed in Theorem \ref{mainthm}.
As in the statement above we let $\mathcal{F}$ denote the distorted Fourier transform $\wtF$
in the case of an even resonance, and the sharp transform $\mathcal{F}^\sharp$ in the case of an odd resonance.

\begin{rem}[About modified scattering]\label{remmodscatt}
As a consequence of Proposition \ref{propinfty} and, in particular of the asymptotic ODE \eqref{fODE},
we prove that the solutions in Theorem \ref{mainthm} have the following asymptotics:
there exists $W_{+\infty}\in L^{\infty}_k$ such that, for $t \geq 1$
\begin{align}\label{mainasy}
\left| (\mathcal{F}f)\left(t,k\right)\exp\left(\frac{i}{2}\int_{0}^{t} \left| (\mathcal{F}f) \left(s,k\right)\right|^{2}
	\frac{ds}{s+1}\right) - W_{+\infty}(k)\right| \lesssim \varepsilon \, t^{-\rho} 
\end{align}
for some $\rho\in(0,\alpha)$. 

Combining \eqref{mainasy} and the linear asymptotic formula
\begin{align}\label{linearasy}
u\left(t,x\right)=\frac{e^{i\frac{x^{2}}{4t}}}{\sqrt{-2it}}
  (\mathcal{F}f) \left(t,-\frac{x}{2t}\right) 
  + \mathcal{O} \big( {\big\| \partial_{k} (\mathcal{F} f) (t) \big\|}_{L_k^2}
  \, t^{-\frac{3}{4}} \big), \quad t\geq 1,
\end{align}
one obtains the following asymptotic formula for solutions of \eqref{NLSV} in physical space:
for some $\delta >0$,
\begin{align}\label{nonlinearasy}
u\left(t,x\right)=\frac{e^{i\frac{x^{2}}{4t}}}{\sqrt{-2it}}
  \exp\left(-\frac{i}{2}\left|W_{+\infty}\left(-\frac{x}{2t}\right)\right|^{2}\log t\right)
  W_{+\infty} \left(-\frac{x}{2t}\right)
  + \mathcal{O} \left( \varepsilon t^{-\frac{1}{2}-\delta} \right), \quad t\geq1.
\end{align}
For $t \rightarrow -\infty$ one can use time-reversal symmetry to obtain asymptotics.
The scattering matrix associated to $V$ will then enter 
formulas analogous to \eqref{mainasy} and \eqref{linearasy}; see Remark 1.2 in \cite{NLSV}. 
Note that these formulas are the same as the ones in the generic case.
\end{rem}

Let us now give a couple of remarks with examples of classes of potentials that we can treat.

\begin{rem}[Even potentials]\label{remeven}
Note that a Schr\"odinger operator with a non-generic potential that is even, $V(x)=V(-x)$,
has a zero energy resonance that is either even or odd.
In fact, suppose $\varphi$ is a zero resonance, i.e., a globally bounded solution to $(-\partial_{xx}+V)\varphi=0.$
Assume that $\varphi$ is normalized so that $\lim_{x\rightarrow\infty}\varphi(x)=1$ and let $\lim_{x\rightarrow-\infty}\varphi(x)=:a$. 
Then $\tau(x) := \varphi(-x)$ also solves $H\tau = 0$
with $\lim_{y\rightarrow\infty}\tau(y)=a$ and $\lim_{y\rightarrow-\infty}\tau(y)=1$. 
By 
uniqueness 
we have that $\tau=a\varphi$ and therefore
\[ \lim_{y\rightarrow-\infty} \tau(y)=1=a\lim_{y\rightarrow\-\infty}\varphi(y)=a^2. \]
So $a=\pm1$. When $a=1$, one has an even resonance and when $a=-1$, the resonance is odd.

Conversely, if $V$ is a non-generic potential and the zero energy resonance $\varphi$ 
is either even or odd,
then one can deduce that $V$ is even almost everywhere. 
Indeed, when $\varphi(x) \neq 0$ then $V(x) = \partial_{xx}\varphi(x) / \varphi(x) = V(-x)$.
It then suffices to observe that $\varphi$ cannot vanish on a set of positive measure, 
for otherwise the set of zeros of $\varphi$, which is compact, would contain a sequence $x_n \rightarrow x_0$,
and one sees that $\varphi(x_0)=0=\varphi'(x_0)$ ($\varphi \in C^1$); 
by uniqueness for the Cauchy problem for ODEs one gets $\varphi \equiv 0$, which is a contradiction.
\end{rem}

\begin{rem}[Some examples]
Even potentials with zero energy resonances
appear in the linearization around special solutions for many models.
For example, the hierarchy of Schr\"odinger operators with P\"oschl-Teller potentials
\begin{equation}\label{PTn}
H_n := -\partial_{xx}-n(n+1)\text{sech}^2(x),\qquad n\in\mathbb{N}_0,
\end{equation}
comes up in several classical asymptotic stability problems,
such as the linearization around kink solutions of the sine-Gordon equation ($H_1$) 
and the $\phi^4$ model ($H_2)$, and solitons of quadratic ($H_3$) and cubic ($H_2$)  Klein-Gordon equations. 
In all these cases the Schr\"odinger operator also has additional discrete spectrum in the negative half-line.
See for example the discussions in 
\cite{KS,KMM,LSch,KGV} and references therein. 


\end{rem}

\medskip
We conclude this subsection by explaining how our proofs 
can be naturally extended to the Klein-Gordon (KG) setting, and mention some implications.
First, an analogue of Theorem \ref{mainthm} can be proven for the cubic KG equation:

\begin{rem}[The KG equation: cubic case]\label{remKG3}
Consider, for a given mass $m>0$, the equation
\begin{align}\label{KGV}
\partial_t^2 \phi - \partial_x^2\phi + m^2 \phi + V(x)\phi = \phi^3, 
  \qquad (\phi,\phi_t)(0)=(\phi_0,\phi_1),
\end{align}
for an unknown $\phi: \R \times \R \mapsto \R$, a potential $V=V(x)$ 
that is non-generic, and satisfies the assumptions given at the beginning of
Subsection \ref{secmainrest}.
Then, if $V$ and the initial data are sufficiently regular\footnote{We
do not specify the exact regularity of $V$, and whether 
one may consider $V$ in a weighted $L^1$ space also in the KG case, as we do here in the NLS case.} 
and localized\footnote{For $V$ this means \eqref{VassumeWei}, for a possibly larger $\gamma$.},
Theorem \ref{mainthm} can be extended to \eqref{KGV} as follows:\footnote{Here 
we decided to use a topology for the data as in \cite{KGV,KGVSim},
but other (similar) choices are possible.
}

\smallskip
There exists $0 < \epsilon_0 \ll 1$ such that for all $\eta \leq \epsilon_0$ 
and initial data satisfying 
\begin{align}\label{datasmallKG}
{\big\| (\sqrt{H+m^2} \phi_0, \phi_1) \big\|}_{H^4} 
  + {\big\| \jx (\sqrt{H+m^2}\phi_0, \phi_1) \big\|}_{L^2} \leq \eta,
\end{align}
the equation \eqref{KGV} has a unique global solution $\phi\in C(\R,H^5(\R))$. 
This solution satisfies the sharp decay rate
\begin{align}\label{main1fdecaykg}
\left\Vert \phi(t)\right\Vert_{L^\infty_x} \lesssim \eta \left(1+\left|t\right|\right)^{-\frac{1}{2}}
\end{align}
and has a modified scattering behavior by a logarithmic correction.\footnote{See \cite{KGV} 
for details on the analogues of \eqref{mainasy}-\eqref{nonlinearasy} for the KG case.
Note that, although \cite{KGV} does not treat in full the non-generic case as we do here,
the asymptotic formulas are the same.}

Moreover, if we define the profile of the solution by
$f = e^{it\sqrt{H+m^2}}(\partial_t \phi - i \sqrt{H+m^2} \phi)$,
then, for some $0 < p_0 < \alpha < 1/4$, we have the global bounds
\begin{align}\label{main1fboundsKG}
{\big\| \jk^{3/2} (\mathcal{F} f) (t) \big\|}_{L_k^\infty} 
  + \jt^{-\alpha} {\big\| \jk \partial_{k} (\mathcal{F} f) (t) \big\|}_{L_k^2} 
  + \jt^{-p_0} {\big\| \jk^4 (\mathcal{F} f) (t) \big\|}_{L_k^2} \lesssim \eta.
\end{align}

\smallskip
About the proof: The above extension of Theorem \ref{mainthm} to \eqref{KGV} can be achieved by
noticing that all the ingredients that we use in our proof can either be directly
used in the KG case or adapted with minor modifications.
More precisely:

\setlength{\leftmargini}{1.5em}
\begin{itemize}

\item The structure of the nonlinear spectral distribution only depends on the Schr\"odinger operator
and not the evolution, so Theorem \ref{theomu} can be used as a black-box for any cubic nonlinearity
(with natural extensions in the case of nonlinearities of other homogeneity).

\item The commutation property for $\partial_k$ used in Lemma \ref{lem:algetri}
has a direct analogue in the context of the KG equation using $\jk\partial_k$;
see for example Germain-Pusateri \cite{KGV}. 

\item Analogues of all the necessary linear estimates from Section \ref{seclinest} 
are available in the case of the KG flow:

\noindent
(1) The smoothing estimates in Lemma \ref{lem:smoothingsim} also hold for the linear Klein-Gordon 
flow since these estimates are tied to the small-frequency behavior,
which is the same for KG and 
Schr\"odinger flows.  

\noindent
(2) The local decay estimates in Subsection \ref{sseclocdec} are also available for KG, for the same reason above.

\noindent
(3) Linear dispersive estimates, like those in Subsection \ref{ssecdisp}, are also available (and standard)
provided one adapts the norms on the right-hand sides of the inequalities in Corollaries \ref{cor:pointwisesingular}
and \ref{cor:regularLinfty} and Lemma \ref{lem:pointwiseH}
to take into account the different behavior of the KG flow for large frequencies.
This is done using the norms in \eqref{main1fboundsKG}.

\end{itemize}

\end{rem}

\begin{rem}[The KG equation: quadratic case]\label{remKG2}
Consider the quadratic model
\begin{align}\label{KGVquad}
\partial_t^2 \phi - \partial_x^2\phi + \phi + V(x)\phi = a(x)\phi^2 + b(x) \phi^3, 
  \qquad (\phi,\phi_t)(0)=(\phi_0,\phi_1),
\end{align}
for $\phi$, $V$ and $(\phi_0,\phi_1)$ as above, and a localized coefficient\footnote{Extensions
to non-localized coefficients are also possible, but require more algebraic work 
in carrying out normal form transformations, along the lines of \cite{KGV}.} $a=a(x)$.
 
One of the results in the work by Lindblad-L\"uhrmann-Soffer \cite{LLS1}
gives global bounds (including decay at the linear rate)  
and asymptotics for solutions of \eqref{KGVquad} with $V=0$ and under the 
`non-resonance' assumption $\widehat{a\varphi^2}(\pm \sqrt{3}) = 0$.

For more strongly localized data 
one of the results in Lindblad-L\"uhrmann-Schlag-Soffer \cite{LLSS} 
gives global bounds (for localized $L^2$-norms) 
and sharp linear decay for solutions of \eqref{KGVquad} with $b=0$ and under the 
`non-resonance' assumption $\big(\wtF (a\varphi^2)\big)(\pm\sqrt{3}) = 0$.

With the approach of our paper, 
one can extend the above results to include a non-trivial potential in the cited result from \cite{LLS1} 
or, equivalently, cubic terms in the cited results from \cite{LLSS},
provided the zero energy resonance is either even or odd (and the same non-resonance assumption 
$\big(\wtF (a\varphi^2)\big)(\pm\sqrt{3}) = 0$ holds).
\end{rem}

\begin{rem}[About the sine-Gordon equation]\label{remSG}
In important recent work, L\"uhrmann-Schlag \cite{LSch} consider the sine-Gordon equation
\begin{align}\label{SG} 
\partial_t^2 \phi - \partial_x^2\phi + \sin \phi = 0,
\end{align}
and proved asymptotic stability for small and localized odd perturbation of the kink $K(x) = 4\arctan(e^x)$.
The problem is essentially equivalent to proving global bounds and asymptotics for odd solutions
of the linearized equation
\begin{align}\label{SGlin} 
\partial_t^2 u + H_1 u + u = a(x) u^2 + u^3, 
\end{align}
where $H_1$ is defined in \eqref{PTn}.
Since the resonance of $H_1$ is odd, the assumption that $u$ is also odd 
does not eliminate the issues associated to it, but 
only avoids modulating by the symmetries of the equation.
The key `non-resonance' condition $\big(\wtF (a\varphi^2)\big)(\pm\sqrt{3}) = 0$ discussed in the previous remark,
holds for this model.

In \cite{LSch} the authors make use of a factorization property,
called the `super-symmetric' factorization,
of the linear operator of \eqref{SGlin}
that conjugates it to the flat one, while retaining the structure of the equation and the 
`non-resonance' condition.
A natural question then arises about the intrinsic necessity of this factorization,
and related implications, such as, how can one treat problems with resonances 
(and a discontinuous distorted Fourier transform) without a `super-symmetric' structure.

Our results show that it is actually possible 
to approach this type of problems 
without resorting to factorization properties that conjugate the Schr\"odinger operator to the free one.
\end{rem}


\smallskip
\subsection{The bootstrap argument and structure of the paper}
The proof of Theorem \ref{mainthm} is based on a bootstrap argument which involves the norms in \eqref{main1fbounds}.
This functional setting and the overall set-up for the proof is fairly standard, and common to most papers on this topics, such as
\cite{GPR,NLSV}, that use norms based on the distorted Fourier transform.
In the flat case analogous norms are used with the standard Fourier transform instead, 	see \cite{KP,HN}. 
In the context of the Klein-Gordon equation similar norms are employed as well, with small modifications
(e.g., using the norms in \eqref{main1fboundsKG})
see for example \cite{KGV,KGVSim}.
Essentially equivalent norms expressed in the physical space are used in other works on the subject, 
see for example \cite{Del,LS,LLS,LSch}.

There are two main steps in the proof of Theorem \ref{mainthm}:

\setlength{\leftmargini}{1.5em}
\begin{itemize}

\item[(i)] Prove a bound on the weighted-type norm:
\begin{align}\label{aprioriwei}
\| \partial_k \mathcal{F}f(t) \|_{L^2} \leq C_0\varepsilon \jt^\alpha,
\end{align}
for some $C_0>1$;

\smallskip
\item[(ii)] Prove a uniform bound on the Fourier-$L^\infty$ norm of the profile
\begin{align}\label{aprioriinfty}
\|\mathcal{F}f(t) \|_{L^\infty} \leq C_0\varepsilon.
\end{align}

\end{itemize}

Note that the $H^1$ norm can be bounded by $C\varepsilon$ 
using the conservation of the mass and the Hamiltonian and Sobolev's embedding.
Also, using the equivalence of norms in \eqref{kusharp} (and (ii) of Proposition \ref{proptildeF}) we have
\begin{align}\label{Sob}
{\| u(t) \|}_{H^1} \approx {\| f(t) \|}_{H^1} \approx {\| \jk \mathcal{F} f (t)\|}_{L^2} \lesssim \varepsilon.
\end{align}
These two facts then imply global existence in $H^1$ by standard arguments.
Through the linear estimate \eqref{eq:linearpoinwiseH}
one also obtains the pointwise decay \eqref{main1fdecay} directly from \eqref{aprioriwei} and \eqref{aprioriinfty}
for $|t| \geq 1$, and from \eqref{Sob} for $|t|\leq 1$.

The asymptotics described in Remark \ref{remmodscatt} are obtained 
in the course of the proof of the Fourier-$L^\infty$ bound \eqref{aprioriinfty},
that is based on the derivation of an asymptotic ODE for $(\mathcal{F}f)(\cdot,k)$.



\medskip
The first main task we need to accomplish is to show the 
following bootstrap proposition for the weighted norm:

\begin{prop}[A priori weighted estimates]\label{propaprwei}
Assume that $u\in C(\R,H^1(\R))$ is a solution of \eqref{NLSV} with initial data as in \eqref{datasmall},
satisfying, for some $T>0$,
\begin{align}\label{aprweias}
\sup_{t\in[0,T]} \big( {\|\mathcal{F}f(t) \|}_{L^\infty_k} 
  + \jt^{-\alpha} {\| \partial_k \mathcal{F}f(t) \|}_{L^2_k} \big) \leq 2C_0\varepsilon,
\end{align}
for some $C_0 \geq 1$.
Then, for all $t\in[0,T]$,
\begin{align}\label{aprweiconc}
{\| \partial_k \mathcal{F}f(t) \|}_{L^2_k} \leq C' \varepsilon + C\varepsilon^3 \jt^{\alpha},
\end{align}
for some\footnote{Here $C'$ is the implicit constant in the inequality ${\|\partial_k \mathcal{F}f(0)\|}_{L^2}
\lesssim {\| \jx f(0) \|}_{L^2}$; see \eqref{eq:weiF}.} $C,C'>0$.
In particular, for $\varepsilon\leq\epsilon_0$ small enough and $C_0$ large enough, we have 
\begin{align}\label{aprweifin}
{\| \partial_k \mathcal{F}f(t) \|}_{L^2_k} \leq C_0\varepsilon  \jt^{\alpha}
\end{align}
for all $t\in \R$.
\end{prop}

\eqref{aprweifin} follows from \eqref{aprweiconc}, the continuity of 
$t \mapsto \|\partial_k \mathcal{F}f(t) \|_{L^2}$, 
and a standard bootstrap argument. Proposition \ref{propaprwei} then follows from 
the proof of \eqref{aprweiconc} under the assumption \eqref{aprweias}.

\medskip
The other main task we need to accomplish is the bound \eqref{aprioriinfty}. 
This is the content of the following proposition.

\begin{prop}[A priori $L^\infty$ estimates]\label{propinfty}
Assume that $u\in C(\R,H^1(\R))$ is a solution of \eqref{NLSV} with initial data as in \eqref{datasmall}, 
satisfying, for some $T>0$ the a priori bounds \eqref{aprweias}.
Then, there exists $\rho>0$ such that
\begin{align}\label{fODE}
i\frac{d}{dt} \mathcal{F}f(t,k) = \frac{1}{2t}\big|\mathcal{F}f(t,k)\big|^2 \mathcal{F}f(t,k)
   + \mathcal{O}(\varepsilon^3 \jt^{-1-\rho})
\end{align}
for all $t\in[0,T]$.
In particular, 
\begin{align}\label{aprinftyfin}
{\|\mathcal{F}f(t) \|}_{L^\infty_k} \leq C_0\varepsilon 
\end{align}
for all $t\in \R$.
\end{prop}

We now show how the proof of the above proposition is reduced to showing the 
validity of the asymptotic ODE in \eqref{fODE}. 

\begin{proof}[Proof of \eqref{aprinftyfin} using \eqref{fODE}]
Define the modified profile 
\begin{align}\label{aprinftyw}
w(t,k) := \mathcal{F}f(t,k) \exp\Big(i \int_0^t \big|\mathcal{F}f(s,k)\big|^2 \, \frac{ds}{2(s+1)} \Big),
\qquad |\mathcal{F}f(t,k)| = |w(t,k)|.
\end{align}
Using \eqref{fODE} we see that $\partial_t w (t) = \mathcal{O}(\varepsilon^3 \jt^{-1-\rho})$
so that, for all $t_1 < t_2 \in [0,T]$, 
\begin{align}\label{aprinftyconc}
\big| w(t_1,k) - w(t_2,k) \big| \lesssim \varepsilon^3 \langle t_1 \rangle^{-\rho}.
\end{align}
In particular, with $t_1=0$ we see that, for all $t \in [0,T]$,
\begin{align}\label{aprinfty1}
\big| w(t,k) \big| \leq |\mathcal{F}f(0,k)| + C\varepsilon^3,
\end{align}
for some $C > 0$.
Using (i) in Proposition \ref{proptildeF} and the assumptions
on the data \eqref{datasmall}, we have, for $C_0$ large enough, 
$|\mathcal{F}f(0,k)| \leq (C_0/2) \|f(0)\|_{L^1_x} \leq (C_0/2)\varepsilon$.
Therefore, for $\varepsilon \leq \epsilon_0$ small enough, we see that, for all $t \in [0,T]$, 
\begin{align}\label{aprinfty2}
{\| \mathcal{F}f(t,\cdot) \|}_{L^\infty} = {\| w(t,\cdot) \|}_{L^\infty}
  \leq (C_0/2) \varepsilon + C\varepsilon^3 \leq C_0\varepsilon.
\end{align}
\eqref{aprinftyfin} then follows from this, the continuity of 
$t \mapsto \|\mathcal{F}f(t) \|_{L^\infty}$ and a bootstrap argument.
\end{proof}

We conclude this subsection by showing how the asymptotics \eqref{mainasy} 
(and, therefore, \eqref{nonlinearasy}) follow from the above propositions. 

\begin{proof}[Proof of \eqref{mainasy}]
From \eqref{aprinftyfin} and \eqref{aprweifin} we have that \eqref{fODE} holds for all $t\in\R$.
Then \eqref{aprinftyconc} also holds for all $0\leq t_1<t_2<\infty$.
In particular, we see that $w(t)$ is a Cauchy sequence in time with values in $L^\infty_k$.
We then let $W_{+\infty}(k) := \lim_{t\rightarrow \infty} w(t,k)$.
Since $\partial_t w (t) = \mathcal{O}(\varepsilon^3 \jt^{-1-\rho})$ we have
\begin{align*}
{\|  w(t,k) - W_{+\infty}(k) \|}_{L^\infty_k} \lesssim \varepsilon^3 \jt^{-\rho},
\end{align*}
and using \eqref{aprinftyw} we arrive at \eqref{mainasy}.
\end{proof}

\smallskip
\subsubsection*{Organization of the paper}
Section \ref{seclinscatt} presents basic results on linear scattering theory and introduces the 
distorted Fourier transform, the modified/sharp transform needed in the case of an odd resonance,
and our decomposition of the (modified) generalized eigenfunctions.

Section \ref{seclinest} contains all the necessary linear estimates, 
including smoothing estimates and local improved decay.

Section \ref{secmu} gives the main result on the structure of the nonlinear spectral measure.
In all the main Sections \ref{seclinest}-\ref{sec:estimateregular} we will work 
under the assumption the the zero energy resonance is odd, and therefore use the appositely defined `sharp' transform.
The case of an even resonance can be treated identically by replacing the sharp transform $\mathcal{F}^\sharp$ with the 
standard distorted transform $\wtF$. 

The nonlinear estimates needed to show \eqref{aprweiconc} 
are carried out in Sections \ref{sec:Cubic} and \ref{sec:estimateregular}
using the results from Sections \ref{seclinscatt}, \ref{seclinest} and \ref{secmu}.

Finally, the proof of \eqref{fODE}, hence of Proposition \ref{propinfty}, is presented in Section \ref{secinfty}.
Although the arguments there are fairly standard by now,
we give some details of the proof and in particular show how to modify (and simplify in some instances) 
the arguments in \cite{NLSV,GPR} to deal with the case of non-generic potentials.



\medskip
\subsection{Ideas of the proof}\label{ssecideas}
To deal with a non-generic potential in the setting of this paper
we use three main ingredients: 

\smallskip
\noindent
(1) a simple modification of the basis for the distorted Fourier transform  
to resolve the (possible) discontinuity at zero energy due to the presence of a resonance,

\smallskip
\noindent
(2) a detailed analysis of the nonlinear spectral measure
and, in particular, of its low-frequency structure, and 

\smallskip
\noindent
(3) general ``smoothing estimates'' for Schr\"odinger operators with non-generic potentials.

\smallskip
Our first step  is to use the Fourier transform adapted to $H=-\partial_{xx}+V$,
the so-called ``distorted Fourier Transform'' (dFT), to write Duhamel's formula in (distorted) frequency space.
For the sake of this brief introduction, it suffices
to admit for the moment the existence of ``generalized plane waves'' $\mathcal{K}(x,k)$ such that
one can define an $L^2$-unitary transformation $\wtF$ by
\begin{align}\label{tildeFint}
\wtF[f](k) := \widetilde{f}(k) := \int_\R \overline{\mathcal{K}(x,k)}f(x)\,dx,
\quad \mbox{with} \quad \wtF^{-1}\left[\phi\right](x) = \int_\R \mathcal{K}(x,k)\phi(k)\,dk.
\end{align}
See \eqref{matK}, \eqref{psipm} and \eqref{defT}
for the precise definition of $\mathcal{K}(x,k)$ and its relation
with the generalized eigenfunctions of $H$.
The distorted transform $\wtF$ diagonalizes the Schr\"odinger operator,
$-\partial_{xx}+V=\wtF^{-1}k^{2}\wtF$, and standard functional calculus follows.

It is important to remark that, in contrast with the generic case, in the non-generic setting
the distorted Fourier transform does not vanish at $k=0$ (zero energy);
this fact is essentially equivalent to the lack of local decay for the flow $e^{itH}$.
Moreover (unless the zero energy resonance is even) 
the distorted Fourier transform has a jump discontinuity at zero. 

As already explained, the proof of our main theorem is based on a bootstrap argument.
Given a solution $u$ of \eqref{NLSV} we let $f = e^{-itH} u$ be its linear profile, so that 
$\widetilde{f}(t,k) = e^{itk^2}\wt{u}(t,k)$.
A natural strategy, employed in several previous works, e.g. \cite{GPR,NLSV},
is to prove bounds on $\wt{f}$ to infer bounds on $u$,
using, for example, the basic linear estimate 
\begin{align}\label{intlinest}
{\| u(t,\cdot) \|}_{L^\infty_x} \lesssim \frac{1}{|t|^{1/2}} {\| \wt{f}(t) \|}_{L^\infty_k}
	+ \frac{1}{|t|^{3/4}} {\| \partial_k \wt{f}(t) \|}_{L^2_k},
\end{align}
which is the analogue of the linear estimate for the case $V=0$
(where one can replace ${\| \partial_k \wt{f}(t) \|}_{L^2}$ by a standard weighted norm 
${\| xf(t) \|}_{L^2}$).
However, the possible discontinuity of the distorted Fourier transform
brings in the first 
obstacle since one cannot compute $\partial_k \wt{f}$ at $k=0$ when the zero energy resonance is odd.
When the resonance is even, instead, $\wt{f}$ is Lipschitz at zero and this issue is absent.

In order to resolve the above issue in the case of an odd zero energy resonance, 
we introduce a \emph{modified}  version of the distorted Fourier base, 
$\mathcal{K}^{\#}(x,k) = \mathrm{sign}(k) \mathcal{K}(x,k)$ and its associated  transform $\mathcal{F}^{\#}$. 
We then observe that this modified Fourier basis is Lipschitz at $k=0$, 
and the same linear estimate \eqref{intlinest} holds when we replace all $\wt{f}$ by $f^{\#}:=\mathcal{F}^{\#}[f]$.  
This gives us the basic setting to carry out our analysis. 
To obtain the sharp pointwise decay of $|t|^{-1/2}$ it then suffices to control 
$f^{\#}(t,k)$ uniformly in $k$ and $t$ and the $L^2_k$-norm of $\partial_k {f}^{\#}(t)$ with a small growth in $t$. 
Both of these bounds can be achieved studying the equation 
in the modified distorted Fourier space. 
A bound on the weighted-type norm ${\| \partial_k f^\sharp\|}_{L^2}$ is the main estimate we need to prove
while, as mentioned before, the (modified distorted) Fourier-$L^\infty$ bound 
can be proved with relatively small modifications of the arguments from \cite{GPR,NLSV},
and some additional care in the handling of small frequencies.

In the modified distorted Fourier space, Duhamel's formula associated to \eqref{NLSV} becomes
\begin{align}\label{introD}
\begin{split}
f^{\#}(t,k) & = f^{\#}(0,k) \pm i \int_{0}^{t} \mathcal{N}_{\mu^\sharp}[f,f,f](s,k) \, ds,
\\ 
\mathcal{N}_{\mu^\sharp}[f,f,f](s,k) & := \!\! \iiint e^{is(-k^2+\ell^2-m^2+n^2)}
f^{\#}(s,\ell)\overline{f^{\#}(s,m)}f^{\#}(s,n) \, \mu^\sharp(k,\ell,m,n) \,dndmd\ell,
\end{split}
\end{align}
where we have defined
the {\it modified nonlinear spectral distribution}
\begin{align}\label{intromu}
\mu^{\#}(k,\ell,m,n):=\int\overline{\mathcal{K}^{\#}(x,k)}\mathcal{K}^{\#}(x,\ell)
  \overline{\mathcal{K}^{\#}(x,m)}\mathcal{K}(x,n)\,dx.
\end{align}

To understand the structure of $\mu^\sharp$ we first use a decomposition of $\K^\sharp$
into two parts, $\K^{\#}_S$ and $\K^{\#}_R$, with the following properties:
$\K_S^\sharp$ consists of two pieces, one carries the main contribution from the zero energy resonance
and does not vanish at $k=0$,
while the other piece is a linear combination of exponentials $e^{\pm i xk}$ whose coefficients 
vanish at $k=0$; 
$\K_R^\sharp$ is the part that arise from the interaction with the potential,
it has strong localization in $x$ and is regular in $k$.
See \eqref{Ksharpdecomp} and \eqref{KsharpS}-\eqref{KR} for the precise formulas.  

According to this decomposition 
we split of $\mu$ into three pieces: 
\begin{align}\label{intromudec}
\mu(k,\ell,m,n) = \sqrt{2\pi}\delta_0(k-\ell+m-n) + \mu^{\#}_L(k,\ell,m,n) + \mu^{\#}_R(k,\ell,m,n).
\end{align}
We call $\mu^{\#}_L$ the ``low-frequency improved'' part of $\mu^{\#}$,
since it vanishes if one of $k,\ell,m,n$ is $0$,
and we call $\mu^{\#}_R$ the``regular'' part of $\mu^{\#}$ since it is a 
function rather than a distribution.   
We remark that this decomposition is finer than the one we proposed for the generic case in \cite{NLSV}, 
and it is needed since we can not use local improved decay here. 

According to \eqref{intromudec}
we split the nonlinear term $\mathcal{N}_{\mu^\sharp}$ in \eqref{introD} into three corresponding pieces: 
$\mathcal{N}_0 := \mathcal{N}_{\sqrt{2\pi}\delta_0}$, 
$\mathcal{N}_L := \mathcal{N}_{\mu^{\#}_L}$ and  $\mathcal{N}_R := \mathcal{N}_{\mu^{\#}_R}$.
Clearly $\mathcal{N}_{0}$ can be handled as in the free case $V=0$,
while new ingredients are needed to estimate the other two nonlinear terms.

To deal with $\mathcal{N}_L$ we first apply a ``commutation identity''
with $\partial_k$ discovered in \cite{NLSV}. 
An explicit calculation shows that the (flat) transform of the multilinear commutator
between $\partial_k$ and $\mathcal{N}_L$ 
is a localized term which, a priori, may be of the form $a(x) s |u|^2 u$, for some localized function $a$. 
We then face the following obstacle: 
in the non-generic setting, the best 
decay rate one can obtain for local-in-space norms of $u(t)$ is $t^{-1/2}$.
In contrast, in the generic case (or in the presence of symmetries), local decay 
has an improved rate of $t^{-1}$, under the assumption that $xf$ is bounded in $L^2$, for example;
see the estimates \eqref{improvedlocaldecay}.
This weak local decay in the non-generic setting 
is far from sufficient to close the estimate on $\partial_k \mathcal{N}_L$ 
by a direct time-integration.

We then rely on some vanishing properties of $\mu^\sharp_L$ 
which corresponds to low frequency improvements for $\mathcal{N}_L$ and its multilinear commutator with $\partial_k$. 
%
%
We leverage this vanishing using two distinct types of estimates:
smoothing estimates obtained by exploiting the time oscillation of the Schr\"odinger flow (see Lemma \ref{lem:smoothingsim})
and enhanced local decay for low frequency improved Schrödinger flows (see Lemma \ref{lem:lowlocaldecay}).
Our main smoothing estimate can be stated as follows:
assuming $\mathcal{Q}(x,k)$ is a bounded function, and letting $\phi=\phi(k)$ be a  
function such that $|\phi(k)|\lesssim \sqrt{|k|}$, we have
\begin{align}\label{eq:smoothingintro}
\left\Vert \int\left[\int e^{-ik^{2}s} \phi(k) \overline{\mathcal{Q}}(y,k)F(s,y)\,dy\right]
	  \,ds\right\Vert _{L_{k}^{2}}\lesssim & \left\Vert F\right\Vert _{L_{y}^{1}L_{s}^{2}}.
\end{align}
This estimate can be proven using the Fourier transform in time,
and may be regarded as the dual version of the classical Kato-$\frac{1}{2}$ smoothing estimate, 
see Remark \ref{rem:classicalsmoothing}. 
The key point 
is the $L^2_t$ norm on the right-hand side of \eqref{eq:smoothingintro}, 
which is sufficient to make up for the lack of local decay.

Finally, the regular part $\mathcal{N}_R$ is essentially
a nonlinear term of the form  $\whF\big( e^{isH} a(x)|u(s)|^2u(s)\big)$, for some localized coefficient $a$.
It is not hard to see that applying $\partial_k$ to $\mathcal{N}_R$ essentially amounts to multiplying it by 
a factor of $sk$. For $|k|$ small, this factor directly allows us to employ the 
smoothing estimate \eqref{eq:smoothingintro} to obtain the desired bound. 
For $|k|\gtrsim 1$, \eqref{eq:smoothingintro} cannot be applied directly
but we can use an integration by parts argument to absorb the growing factor of $k$ 
followed by an application of \eqref{eq:smoothingintro} (in the form given by Corollary \ref{corsmoothing}),
and  improved local decay for the differentiated Schr\"odinger flow (see Lemma \ref{lem:localderivative}).

\smallskip
\subsection*{Notation}
We use the notation $\langle x\rangle := (1+|x|^{2})^{\frac{1}{2}}$.

\noindent
As usual \textquotedblleft $A:=B\lyxmathsym{\textquotedblright}$
	or $\lyxmathsym{\textquotedblleft}B=:A\lyxmathsym{\textquotedblright}$
	is the definition of $A$ by means of the expression $B$. 

\noindent
For positive quantities $a$ and $b$, we write $a\lesssim b$ for
	$a\leq Cb$ where $C$ is a universal constant,
	and $a\simeq b$ when $a\lesssim b$ and $b\lesssim a$. 
	
\noindent
We let $\mathbf{1}_A$ denote the characteristic function of the set $A$,
	and let $\mathbf{1}_+ := \mathbf{1}_{[0,\infty)}$, $\mathbf{1}_- := \mathbf{1}_{(-\infty,0)}$.
	
\noindent
We denote $u_{t}:=\frac{\partial}{\partial_{t}}u$, $u_{xx}:=\frac{\partial^{2}}{\partial x^{2}}u$.

\noindent
The (regular) Fourier transform is defined as
	\begin{equation}\label{eq:FT}
	\hat{h}\left(\xi\right) = \whF \left[h\right]\left(\xi\right)
	=\frac{1}{\sqrt{2\pi}} \int e^{-ix\xi}h(x)\,dx,
	\qquad \whF^{-1} [h](x)
	=\frac{1}{\sqrt{2\pi}} \int e^{ix\xi}h(\xi)\,d\xi. 
	\end{equation}

\noindent
We use the standard notation $L^p$ for Lebesgue spaces, and $H^s$ for Sobolev spaces of order $s$.
We will sometimes specify the domain and the variable, e.g. $L^p_x(\R)$ or $L^2_s([0,t])$,
but often omit these when there is no risk of confusion.

Given $p,q \in[1,\infty]$, the mixed norms for a space-time function $F(t,x)$ are given by
$$	\big\| F \big\|_{L_{x}^{p}L_t^q}=	\big\| \big\| F\big\|_{L^q_t} \big\|_{L_{x}^{p}}\quad\text{and}\quad\big\| F \big\|_{L_{t}^{q}L_x^p}=	\big\| \big\| F\big\|_{L^p_x} \big\|_{L_{t}^{q}}.$$


	\medskip
{\bf Acknowledgments.}
	F.P. was supported in part by a start-up grant from the University of Toronto 
	and the NSERC grant RGPIN-2018-06487. We would like to thank Jonas L\"uhrmann for useful comments and suggestions.

	\medskip
	\section{Linear theory, distorted and `sharp' Fourier transforms}\label{seclinscatt}

	In this section, we first collect some facts about Jost functions in \S\ref{sec:JostSpec},
	present some elements of linear scattering theory, 
	and introduce the (standard) distorted Fourier transform (dFT) in \S\ref{ssecDFT0}.
	In \S\ref{subsec:GeneEigenZero} and \S\ref{subsec:DecK} we analyze the generalized eigenfunctions
	of the Schr\"odinger operator and introduce a modification of the distorted Fourier transform
	that resolves the singularity at the zero frequency in the case of non-generic potentials;
	we call this the `sharp' transform and denote it by $\mathcal{F}^\#$). 
	The introduction of this variant of the dFT is only needed in the case of odd zero energy resonances.
	For even resonances the standard dFT is continuous, so in the rest of our paper we 
	will mostly focus on the harder case of an odd resonance.

	\medskip
	\subsection{Jost functions and spectral theory}\label{sec:JostSpec}
	In this subsection we recall some basic properties of Jost functions.
	The facts provided here hold for all potentials regardless of genericity;
	see Deift-Trubowitz \cite{DT}, Delort \cite{Del}, 
	Germain-Pusateri-Rousset \cite{GPR}, Lindblad-L\"uhrmann-Schlag-Soffer \cite{LLSS} and references therein.
	
	The Jost functions $\psi_+(x,k)$ and $\psi_-(x,k)$
	are defined as solutions to 
	\begin{align}\label{psipm}
	H\psi_{\pm}(x,k)=\left(-\partial_{xx}+V\right)\psi_{\pm}(x,k)=k^{2}\psi_{\pm}(x,k)
	\end{align}
	such that
	\begin{align}\label{psipmlim}
	\lim_{x\rightarrow\infty}\left|e^{-ikx}\psi_+(x,k)-1\right|=0,
	\qquad  \lim_{x\rightarrow-\infty}\left|e^{ikx}\psi_-(x,k)-1\right|=0. 
	\end{align}
	We let 
	\begin{align}\label{mpm}
	m_{\pm}(x,k)=e^{\mp ikx}\psi_{\pm}(x,k).
	\end{align}
	Then for fixed $x$, $m_{\pm}$ is analytic in $k$ for $\Im k>0$
	and continuous up to $\Im k\geq0$.
	
	
	We define
	\begin{align}\label{defW_+-}
	\mathcal{W}_{+}^{s}(x)=\int_x^\infty \jy^{s}\left|V(y)\right|\,dy,
	\qquad \mathcal{W}_{-}^{s}(x)=\int_{-\infty}^{x}\jy^{s}\left|V(y)\right|\,dy.
	\end{align}
	Note that if $V$ decays fast enough, then $\mathcal{W}_{\pm}^{s}(x)$
	also decays as $x\rightarrow\pm\infty$ respectively.

	%
	%
	
	\begin{lem}\label{lem:Mestimates}
		For every $s\ge0$, we have the estimates: 
		\begin{align}\label{Mestimates1}
		\begin{split}
		& \left|\partial_{k}^{s}\left(m_{\pm}(x,k)-1\right)\right|\lesssim\frac{1}{\jk}\mathcal{W}_{\pm}^{s+1}(x),\qquad \pm x\geq-1,
		\\
		& \left|\partial_{k}^{s}\left(m_{\pm}(x,k)-1\right)\right|\lesssim\frac{1}{\jk}\jx^{s+1}, \qquad \pm x\leq1.
		\end{split}
		\end{align}
		Also, 
		\begin{align}\label{Mestimates1'}
		\begin{split}
		& \left|\partial_{k} \left(m_{\pm}(x,k)-1\right)\right|
		\lesssim\frac{1}{|k|}\mathcal{W}_{\pm}^{1}(x),\qquad \pm x\geq-1.
		\end{split}
		\end{align}
		Moreover
		\begin{align}\label{Mestimates2}
		\begin{split}
		& \left|\partial_{k}^{s} \partial_x m_{\pm}(x,k) \right| \lesssim \mathcal{W}_{\pm}^{s}(x), \qquad \pm x\geq-1,
		\\
		& \left|\partial_{k}^{s} \partial_x m_{\pm}(x,k) 
		\right|\lesssim\jx^{s}, \qquad \pm x\leq1.
		\end{split}
		\end{align}
	\end{lem}

	\begin{proof}
		The proofs of these estimates follow from analyzing the Volterra equation satisfied by $m_\pm$,
		that is,
		\begin{align}\label{minteq}
		m_\pm(x,k) = 1\pm\int_x^{\pm\infty} D_k(\pm(y-x))V(y)m_\pm(y,k)\,dy,
		\qquad D_k(x) = \frac{e^{2ik x}-1}{2ik};
		\end{align} 
		as in Deift-Trubowitz \cite{DT}, Weder \cite{We1} or \cite[Appendix A]{GPR}. 
		The estimates \eqref{Mestimates1'} follow from Lemma 2.1 in \cite{We2}.
	\end{proof}

	Denote $T(k)$ and $R_{\pm}(k)$ the transmission
	and reflection coefficients associated to the potential $V$ respectively.
	With these coefficients, one can write
	\begin{align}\label{eq:f_+-1}
	T(k)\psi_{+}(x,k) & = R_{-}(k) \psi_{-}(x,k) + \psi_{-}\left(x,-k\right),
	\\
	\label{eq:f_--1}
	T(k)\psi_{-}(x,k) & = R_{+}(k) \psi_{+}(x,k) + \psi_{+}\left(x,-k\right).
	\end{align}
	Using these formulae, one has
	\begin{align*}
	T(k)\psi_{+}(x,k) & \sim R_{-}(k) e^{-ikx} + e^{ikx}\ \ \ x\rightarrow-\infty,
	\\
	T(k)\psi_{-}(x,k)& \sim R_{+}(k) e^{ikx} + e^{-ikx}\ \ \ x\rightarrow\infty.
	\end{align*}
	Moreover, these coefficients are given explicitly by
	\begin{align}\label{defT}
	\frac{1}{T(k)}=1-\frac{1}{2ik}\int V(x)m_{\pm}(x,k)\,dx, 
	\end{align}
	\begin{align}\label{defR}
	\frac{R_{\pm}(k)}{T(k)}=\frac{1}{2ik}\int e^{\mp2ikx}V(x)m_{\mp}(x,k)\,dx, 
	\end{align}
	and satisfy
	\begin{align}\label{TRid}
	\begin{split}
	& T\left(-k\right)=T(k),\ \ \ R_{\pm}\left(-k\right)=\overline{R_{\pm}(k),}
	\\
	& \left|R_{\pm}(k)\right|^{2}+\left|T(k)\right|^{2}=1,\ \ T(k)\overline{R_{-}(k)}+\overline{T(k)}R_{+}(k)=0.
	\end{split}
	\end{align}

	\begin{definition}[Generic and non-generic potentials]\label{def:generic}
		$V$ is said to be a ``generic'' potential if
		\begin{equation}\label{eq:genericcond}
		\int V(x)m_{\pm}\left(x,0\right)\,dx \neq 0,
	\end{equation}
		and it is ``non-generic'' or ``exceptional'' otherwise.

	\end{definition}

Given \eqref{eq:f_+-1} and \eqref{eq:f_--1}, one can compute the  Wronskian as
\begin{equation}
    T(k)W(\psi_+(\cdot,k),\psi_-(\cdot,k))=-2ik.
\end{equation}
In the non-generic setting, from \eqref{defT} and \eqref{eq:genericcond},
$T(0)\neq0$, whence, $W(\psi_+(\cdot,0),\psi_-(\cdot,0))=0$. 
So from basic ODE theory we have that $\psi_+(x,0) = \varphi(x)$ 
is a globally bounded solution 
to the equation $Hf=0$ which is unique up to a multiplicative constant.

\begin{definition}[Zero-energy resonance]\label{def:zeroresonance}
Assume $V$ is a non-generic potential in the sense of Definition \ref{def:generic}.  
Then there exists a nonzero solution of $H\varphi=0$ with  
$\varphi\in L^{\infty}(\mathbb{R}),\,\,\varphi\neq0$ 
normalized so that $\varphi(x)=\psi_+(x,0)$ approaches $1$ as $x\rightarrow \infty$ 
and a nonzero constant as $x\rightarrow-\infty$. 
Such $\varphi(x)$ is called a zero-energy resonance.
\end{definition}

	Given $T(k)$ and $R_{\pm}(k)$ as above, we
	can define the scattering matrix associated to the potential by
	\begin{equation}
	S(k):=\left(\begin{array}{cc}
	T(k) & R_{+}(k)
	\\
	R_{-}(k) & T(k)
	\end{array}\right),\ \ S^{-1}(k):=\left(\begin{array}{cc}
	\overline{T(k)} & \overline{R_{-}(k)}
	\\
	\overline{R_{+}(k)} & \overline{T(k)}
	\end{array}\right).\label{eq:scatMatrix}
	\end{equation}
	
	We have the following lemma on the coefficients.

	\begin{lem}\label{estiTR}
		Assuming that $\jx^{2}V\in L^1$, we have the uniform estimates for $k\in\mathbb{R}$:
		\[ \jk \big( \left|\partial_{k}T(k)\right| + \left|\partial_{k}R_{\pm}(k)\right| \big) \lesssim 1. \]
	\end{lem}

	\subsection{Some bounds on Pseudo-differential operators}

	We state a few bounds on Pseudo-Differential Operators (PDOs) 
	that are specifically tailored to the Jost functions $m_\pm(x,k)$.
	The proofs of these results can be found in Lemmas 2.4-2.9 in \cite{NLSV}.

	\begin{lem}\label{lem:m-1}
		Suppose $\jx^{\gamma}V\in L^1$.
		Then, for $\gamma>3/2$, we have
		\begin{align}\label{eq:pseeasy}
		\left\Vert \mathbf{1}_{\{\pm x\geq-1\}}\int e^{ik x}m_\pm(x,k)g(k)\,dk
		\right\Vert_{L^2_x}\lesssim{\| g \|}_{L^2},
		\\
		\label{eq:pseeasy-1}
		\left\Vert \int e^{ik x} m_\pm(x,k)
		\mathbf{1}_{\{\pm x\geq-1\}}h(x)\,dx\right\Vert_{L_{k}^{2}}
		\lesssim\left\Vert h\right\Vert_{L^2},
		\end{align}
		and, for $\gamma>\beta+3/2$,
		\begin{align}\label{eq:mPDO+}
		\left\Vert \mathbf{1}_{\{\pm x\geq-1\}}\jx^{\beta}
		\int e^{ik x}\left(m_\pm(x,k)-1\right)g(k)\,dk\right\Vert_{L^2_x}
		\lesssim \left\Vert g\right\Vert_{L^{2}},
		\\
		\label{eq:mPDO+-1}
		\left\Vert \int e^{ik x}\left(m_\pm(x,k)-1\right)
		\mathbf{1}_{\{\pm x\geq-1\}}h(x)\,dx\right\Vert_{L_{k}^{2}}
		\lesssim\left\Vert \jx^{-\beta}h\right\Vert_{L^2_x}.
		\end{align}
		%
		Moreover, with $\gamma>\max(\beta/2+3/4,\beta)$, we have
		\begin{align}\label{eq:psehard}
		& \left\Vert \jx^{\beta} \mathbf{1}_{\{\pm x\geq-1\}}
		\int e^{ik x}\partial_xm_\pm(x,k)g(k)\,dk\right\Vert_{L^2_x}\lesssim{\| g \|}_{L^2},
		\\
		\label{eq:mPDOhard+2}
		& \left\Vert \int e^{ik x}\partial_xm_\pm (x,k)
		\mathbf{1}_{\{\pm x\geq-1\}}h(x)\,dx\right\Vert_{L_{k}^{2}}
		\lesssim\left\Vert \jx^{-\beta}h\right\Vert_{L^2}.
		\end{align}
		%
		Furthermore, if $\gamma > \beta + 5/2$, then
		\begin{align}
		\label{eq:pdomk}
		& \left\Vert \mathbf{1}_{\{\pm x\geq-1\}} \jx^\beta 
		\int e^{ik x}\partial_{k}m_\pm(x,k)g(k)\,dk\right\Vert_{L^2_x}\lesssim{\| g \|}_{L^2},
		\\
		\label{eq:pdomkd}
		& \left\Vert \int e^{ik x}\partial_{k}m_\pm(x,k)
		\mathbf{1}_{\{\pm x\geq-1\}}h(x)\,dx\right\Vert_{L^2_k}
		\lesssim\left\Vert \jx^{-\beta}h\right\Vert_{L^2_x}.
		\end{align}
		Finally if $\gamma > \max(\beta/2 + 3/4, \beta) + 1 $, then
		\begin{align}
		\label{eq:pdomxk}
		& \left\Vert \mathbf{1}_{\{\pm x\geq-1\}} \jx^\beta 
		\int e^{ik x} \partial_x\partial_k m_\pm(x,k)g(k)\,dk\right\Vert_{L^2_x}\lesssim{\| g \|}_{L^2},
		\\
		\label{eq:pdomxkd}
		& \left\Vert \int e^{ik x} \partial_x\partial_k m_\pm(x,k)
		\mathbf{1}_{\{\pm x\geq-1\}}h(x)\,dx\right\Vert_{L^2_k}
		\lesssim\left\Vert \jx^{-\beta}h\right\Vert_{L^2_x}.
		\end{align}

	\end{lem}

	\medskip
	\subsection{Distorted Fourier Transform (dFT)}\label{ssecDFT0}
	We recall some properties of the distorted Fourier transform
	with respect to the perturbed Schr\"odinger operator.
	First, recall that the standard Fourier transform is defined, for $\phi\in L^2$, as in \eqref{eq:FT}.
	Given the Jost functions $\psi_{\pm}$ from \eqref{psipm}, we set
	\begin{align}\label{matK}
	\mathcal{K}(x,k):=\frac{1}{\sqrt{2\pi}}
	\begin{cases}
	T(k) \psi_+(x,k) & k \geq 0
	\\
	T(-k) \psi_-(x,-k) & k < 0,
	\end{cases}
	\end{align}
	and define the ``distorted Fourier Transform'' (dFT) for $f\in \mathcal{S}$ by
	\begin{align}\label{tildeF}
	\wtF\left[\phi\right](k) = \widetilde{f}(k) := \int_\R \overline{\mathcal{K}(x,k)} f(x)\,dx.
	\end{align}
	
	The following proposition summarizes some basic properties of the dFT.
	
\begin{prop}\label{proptildeF}
In our setting, one has
\begin{align}\label{tildeFL2}
{\big\| \wtF\left[f\right] \big\|}_{L^{2}} = \left\Vert f\right\Vert _{L^{2}},\,\,\forall f\in L^{2}
\end{align}
and
\begin{align}\label{tildeF-1}
\wtF^{-1}\left[\phi\right](x) = \int_\R \mathcal{K}(x,k)\phi(k)\,dk.
\end{align}
Also, if $D:=\sqrt{-\partial_{xx}+V}$, 
$m(D)=\wtF^{-1}m(k)\wtF$,
so that, in particular,
$\left(-\partial_{xx}+V\right)=\wtF^{-1}k^{2}\wtF$. We also have the following properties:

\setlength{\leftmargini}{2em}
\begin{itemize}

\item[(i)] If $\phi\in L^1$, then $\widetilde{\phi}$ is 
a bounded function, ${\| \wt{\phi} \|}_{L^\infty} \lesssim {\| \phi \|}_{L^1}$,
which is continuous everywhere except possibly at $k=0$.

\smallskip
\item[(ii)] One has
$\left\Vert \jk \widetilde{u}\right\Vert _{L^{2}} \approx  \| u \|_{H^{1}}$,
where the implicit constants depend on  ${\| V  \|}_{L^1}$.
\end{itemize}

\end{prop}

For a proof see, for example, 
Chapter 5 in \cite{Yaf}, Section 6 in \cite{Agmon}, 
Section 3 in \cite{KGV} or Section 2 in \cite{LLSS} and references therein.
	
	\medskip
	Here are some facts about non-generic potentials (with some comparison to the case of generic potentials),
	and some consequences concerning the discontinuity of the distorted Fourier transform at zero frequency.
	
	\begin{lem}[Low energy scattering]\label{lemngTR}
		If $V$ is a non-generic potential, let
		\begin{align}\label{a}
		a := \psi_+(-\infty,0) \in \R \setminus \{ 0 \}. 
		\end{align}
		Then, 
		\begin{align}\label{LESTR0}
		T(0) = \frac{2a}{1+a^2}, \qquad R_+(0) = \frac{1 - a^2}{1+ a^2}, \qquad \mbox{and} \qquad R_-(0) = \frac{a^2-1}{1+ a^2}.
		\end{align}
	\end{lem}
	
	For a proof see for example Proposition 3.4 in \cite{KGV} and references therein.
	
	\smallskip
	\begin{rem}[Transmission and reflection for symmetric zero energy resonance]\label{remTRnongen}
		In the setting of this paper the Schr\"odinger operator $H$ has a resonance, $\psi_+(x,0) = m_+(x,0)$, with 
		odd or even parity.
		For an odd, respectively even, resonance, we have that $a$ in \eqref{a} is $-1$, respectively $1$,
		and therefore $T(0)=-1$, respectively $T(0)=1$. In both cases we have $R_+(0) = R_-(0) = 0$.
	\end{rem}

	\smallskip
	\begin{rem}[Transmission and reflection for generic potentials]\label{remTRgen}
		For generic potentials the associated transmission and
		reflection coefficients have the following Taylor expansions near $0$:
		Assuming that $\jx^2 V\in L^1$, then
		\begin{align}\label{LEST}
		\begin{split}
		& T(k) = \alpha k+\mathcal{O}(k^2), \quad \alpha\neq0, \quad \text{as}\ k\rightarrow0,
		\\
		& 1+R_{\pm}(k) = \alpha_{\pm}k+\mathcal{O}(k^2), \quad \text{as} \quad k\rightarrow0.
		\end{split}
		\end{align}
		For a detailed proof, see page 144 in Deift-Trubowitz \cite{DT}.
		
		In particular, in the generic case, the distorted Fourier transform vanishes at frequency zero (and is continuous).
		This gives improved local decay estimates, such as, see for example \cite{KS,Sch,GSch,NLSV},
		\begin{align}\label{improvedlocaldecay}
		{\| \jx^{-1} e^{itH} f \|}_{L^\infty} \lesssim \jt^{-3/2} {\| \jx f \|}_{L^1},
		\qquad {\| \jx^{-1} e^{itH} f \|}_{L^\infty} \lesssim \jt^{-1} {\| \jx f \|}_{L^2},
		\end{align}
		which simplify the study of nonlinear equations with generic potentials.
		However these estimates fail in the non-generic case and 
		one cannot have more than $\jt^{-1/2}$ decay even locally in space.
	\end{rem}

	\smallskip
	\begin{rem}[Discontinuity of the distorted Fourier transform]\label{remdisc}
		As mentioned before, the Fourier transform is continuous for a generic potential
		(and vanishes at zero) but a non-generic potential may give rise to a discontinuous Fourier transform
		at zero.
		
		\setlength{\leftmargini}{1.5em}
		\begin{itemize}
			
			\item If $V$ is generic, from \eqref{matK} and Remark \ref{remTRgen},
			we see that $\mathcal{K}(x,0) = 0$ and thus $\widetilde{f}(0) = 0$ if $f \in L^1$. 
			Furthermore, assuming better integrability properties at $\infty$ for $f$, one has
			\begin{align}\label{wtf0gen}
			\begin{split}
			& \mbox{if $k > 0$}, \qquad \widetilde{f}(k) = - 
			\frac{\alpha k }{\sqrt{2\pi}} \int f(x) \psi_+(x,0)\,dx + O(k^2)
			\\
			& \mbox{if $k < 0$}, \qquad \widetilde{f}(k) =
			\frac{ \alpha k }{\sqrt{2\pi}} \int f(x) \psi_-(x,0)\,dx + O(k^2),
			\end{split}
			\end{align}
			where $\alpha$ was defined in \eqref{LEST}.
			Thus, $\widetilde{f}$ is typically continuous, but not continuously differentiable at zero.
			
			\smallskip
			\item If $V$ is exceptional, then
			\begin{align*}
			\sqrt{2\pi} \, \K(x,0+) = \frac{2a}{1+a^2} \psi_+(x,0), \qquad \mbox{and} \qquad \K(x,0-) = \frac{1}{a} \K(x,0+), 
			\end{align*}
			where $a$ is defined in \eqref{a}. Therefore, if $f \in L^1$,
			\begin{align}\label{wtf0exc}
			\begin{split}
			\widetilde{f}(0+) = \frac{2a}{1+a^2} \frac{1}{\sqrt{2\pi}} \int f(x) \psi_+(x,0) \,dx 
			\qquad \mbox{and} \qquad \widetilde{f}(0-) = \frac{1}{a} \widetilde{f}(0+).
			\end{split}
			\end{align}
			As a consequence, $\widetilde{f}$ is continuous if $a=1$, but might not be otherwise.
		\end{itemize}
		
	\end{rem}

	

	\medskip
	
	
	
	\subsection{Generalized eigenfunctions and the zero frequency}\label{subsec:GeneEigenZero}
	We need a suitable decomposition of the generalized eigenfunctions
	in order to have a good understanding of the nonlinear spectral distribution
	(see \eqref{intromu} and \eqref{eq:nonmea}).

	First, from the definition of $\mathcal{K}$ in \eqref{matK} and $m_\pm$ in \eqref{mpm} we have 	
	\begin{align}
	\begin{split}\label{matKk>00}
	\mbox{for $k\geq 0$,} \qquad \sqrt{2\pi}\mathcal{K}(x,k) & = T(k)m_+(x,k)e^{ixk},
	\end{split}
	\\
	\begin{split}\label{matKk<00}
	\mbox{for $k<0$,} \qquad \sqrt{2\pi}\mathcal{K}(x,k) & = T(-k)m_-(x,-k)e^{ixk}.
	\end{split}
	\end{align}
	
	Note that the expression \eqref{matKk>00} is bounded for $x>0$ but unbounded when $x\rightarrow -\infty$
	(and viceversa for  \eqref{matKk<00}).
	We resolve this in the standard way by letting $\Phi$ be a smooth, 
	non-negative function, which is one in a neighborhood of $0$, vanishes outside
	of $[-2,2]$, and such that $\int\Phi\,dx=1$. 
	We then define $\chi_+$ and $\chi_-$ by
	\begin{align}\label{chipm}
	\chi_+(x) = \int_{-\infty}^{x}\Phi(y)\,dy,
	\qquad \text{and} \qquad  \chi_+(x)+\chi_-(x)=1.
	\end{align}
	
	Using $\chi_{\pm}(x)$, the definition of $\mathcal{K}(x,k)$ in \eqref{matK},
	and the identities \eqref{eq:f_+-1}-\eqref{eq:f_--1} we can write
	\begin{align*}
	\begin{split}
	\mbox{for $k\geq 0$,} \qquad \sqrt{2\pi}\mathcal{K}(x,k) & = \chi_+(x)T(k)\psi_+(x,k)
	+\chi_-(x)\left[\psi_-(x,-k) + R_-(k) \psi_-(x,k) \right],
	\end{split}
	\\
	\begin{split}
	\mbox{for $k<0$,} \qquad \sqrt{2\pi}\mathcal{K}(x,k) & =\chi_-(x)T(-k) \psi_-(x,-k)
	+\chi_+(x)\left[ \psi_+(x,k) + R_+(-k)\psi_+(x,-k) \right].
	\end{split}
	\end{align*}
	Then, 
	with the definition of $m_{\pm}(x,k)$ in \eqref{mpm}, we can write:
	\begin{align}
	\begin{split}\label{matKk>0}
	\mbox{for $k\geq 0$,} \qquad \sqrt{2\pi}\mathcal{K}(x,k) & =\chi_+(x)T(k)m_+(x,k)e^{ixk}
	\\
	& +\chi_-(x)\left[m_-(x,-k)e^{ikx}+R_-(k)m_-(x,k)e^{-ikx}\right],
	\end{split}
	\\
	\begin{split}\label{matKk<0}
	\mbox{for $k<0$,} \qquad \sqrt{2\pi}\mathcal{K}(x,k) & =\chi_-(x)T(-k)m_-(x,-k)e^{ixk}
	\\
	& +\chi_+(x)\left[m_+(x,k)e^{ikx}+R_+(-k)m_+(x,-k)e^{-ikx}\right].
	\end{split}
	\end{align}

	\smallskip
	\begin{rem}[Zero frequency behavior for an odd/even resonance]\label{remK0}
		Note that in our setting of a non-generic potential, $\mathcal{K}(x,\cdot)$ may be discontinuous
		at $k=0$. More precisely, we have
		\begin{align}
		\begin{split}\label{lim0+}
		\lim_{k \rightarrow 0^+} \sqrt{2\pi}\mathcal{K}(x,k) & =\chi_+(x)T(0)m_+(x,0)
		+ \chi_-(x) [m_-(x,0) + R_-(0)m_-(x,0)],
		\end{split}
		\end{align}
		and
		\begin{align}
		\begin{split}\label{lim0-}
		\lim_{k \rightarrow 0^-}  \sqrt{2\pi}\mathcal{K}(x,k) & =\chi_-(x)T(0)m_-(x,0)
		+ \chi_+(x) [m_+(x,0) + R_+(0)m_+(x,0) ].
		\end{split}
		\end{align}
		
		We see that if $T(0) = 1$, and therefore $R_\pm(0) = 0$, we have that $\mathcal{K}$ is continuous at $k=0$,
		but if $T(0) \neq 1$, this may not be the case.
		
		When the zero energy resonance is odd we have from Lemma \ref{lemngTR} that $T(0)=-1$ (and $R_{\pm}(0)=0$)
		and therefore
		\begin{align}\label{limodd}
		\begin{split}
		\lim_{k \rightarrow 0^+} \mathcal{K}(x,k) = - \lim_{k \rightarrow 0^-} \mathcal{K}(x,k).
		\end{split}
		\end{align}
		In particular, $\text{sgn}(k)\mathcal{K}(x,k)$ is continuous in the case of an odd zero energy resonance.

		Continuity of $\text{sgn}(k)\mathcal{K}(x,k)$ at $k=0$ 
		can also be seen directly from \eqref{matKk>00}-\eqref{matKk<00} and the fact that
		\begin{align}\label{mpmrel}
		m_+(x,0) = -m_-(x,0). 
		\end{align}
		This latter holds true because oddness implies $m_+(x,0) = -m_+(-x,0)$,
		and therefore $\psi(x) = m_+(-x,0)$ satisfies $H\psi=0$ with $\psi(-\infty) = 1$
		so that, by uniqueness, $\psi(x) = m_-(x,0)$. 
	\end{rem}

	\smallskip
	\subsection{Decomposition of generalized eigenfunctions with an odd resonance}\label{subsec:DecK}
	We concentrate first on the main case of an odd resonance.
	Motivated by \eqref{limodd} we define
	\begin{equation}\label{eq:Ksharp}
	\mathcal{K}^{\#}(x,k) = \text{sgn}(k)\mathcal{K}(x,k).
	\end{equation}
	
From \eqref{matKk>0}-\eqref{matKk<0} we can write
	\begin{align}\label{Ksharpk>0}
	\begin{split}
	\mbox{for \ensuremath{k\geq 0},}\qquad\sqrt{2\pi}\mathcal{K}^{\#}(x,k) & =\chi_{+}(x)T(k)m_{+}(x,k)e^{ixk}
	\\
	& +\chi_{-}(x)\left[m_{-}(x,-k)e^{ikx}+R_{-}(k)m_{-}(x,k)e^{-ikx}\right],
	\end{split}
	\\
	\label{Ksharpk<0}
	\begin{split}\mbox{for \ensuremath{k<0},}\qquad\sqrt{2\pi}\mathcal{K}^{\#}(x,k) & =-\chi_{-}(x)T(-k)m_{-}(x,-k)e^{ixk}
	\\
	& -\chi_{+}(x)\left[m_{+}(x,k)e^{ikx}+R_{+}(-k)m_{+}(x,-k)e^{-ikx}\right].
	\end{split}
	\end{align}

	Let us then analyze the structure of $\mathcal{K}^\#$ more closely. 
	For notational convenience we let $\chi_0(x) \equiv 1$.
	We then use the formulas above to write
	$\mathcal{K}^\#$ as the sum of a ``singular'' and ``regular'' part
	\begin{align}\label{Ksharpdecomp}
	\begin{split}
	& \sqrt{2\pi} \mathcal{K}^{\#}(x,k) =  \mathcal{K}^{\#}_S(x,k) + \mathcal{K}^{\#}_R(x,k),
	\end{split}
	\end{align}
	with the following definitions:
	
	\setlength{\leftmargini}{1.5em}
	\begin{itemize}
		
		\item The singular part is given by
		\begin{align}\label{KsharpS}
		\begin{split}
		& \mathcal{K}^{\#}_S(x,k) := \chi_0(x)\mathcal{K}^{\#}_0(x,k) + \chi_+(x) \mathcal{K}^{\#}_+(x,k) 
		+ \chi_-(x) \mathcal{K}^{\#}_-(x,k)
		\end{split}
		\end{align}
		where 
		\begin{align}\label{K0}
		\mathcal{K}_{0}^{\#}(x,k) & = T(0) m_{+}(x,0)e^{ikx},
		\end{align}
		and
		\begin{align}\label{K+}
		\begin{split}
		\mathcal{K}_+^{\#}(x,k) := (T(k)-T(0)) \mathbf{1}_+(k) e^{ikx} - R_{+}(-k) \mathbf{1}_-(k) e^{-ikx}
		\\
		=: a_+^+(k) e^{ikx} + a_+^-(k) e^{-ikx} ,
		\end{split}
		\\
		\label{K-}
		\begin{split}
		\mathcal{K}_-^{\#}(x,k) := R_{-}(k)  \mathbf{1}_+(k) e^{-ikx} - (T(-k)-T(0) ) \mathbf{1}_-(k) e^{ikx}
		\\
		=: a_-^+(k) e^{ikx} + a_-^-(k) e^{-ikx}.
		\end{split}
		\end{align}
		%
		
		\noindent
		Note that even though we know that $T(0) = -1$ (see \eqref{LESTR0}), 
		we still leave it as $T(0)$ for convenience.
		
		\smallskip
		\item the regular part is given as
		\begin{align}\label{KR}
		\mathcal{K}_{R}^{\#}(x,k):=\begin{cases}
		\chi_{+}(x) \left[(T(k)-T(0))(m_{+}(x,0)-1) + T(k) \left(m_{+}(x,k)-m_{+}(x,0) \right) \right] e^{ikx}
		\\
		\qquad +\chi_{-}(x) \big[\left(m_{-}(x,-k)-m_{-}(x,0)\right)e^{ixk}
		\\
		\qquad +R_{-}(k)\left(m_{-}(x,k)-1\right)e^{-ixk} \big], \qquad \qquad  k\geq0,
		\\
		\\
		-\chi_{-}(x)\left[(T(-k)-T(0))(m_{-}(x,0)-1)e^{ikx} + T(-k)\left(m_{-}(x,-k)-m_{-}(x,0) \right) \right]e^{ikx}
		\\
		\qquad -\chi_{+}(x) \big[ \left(m_{+}(x,k)-m_{+}(x,0)\right)e^{ikx}
		\\
		\qquad +R_{+}(-k)\left(m_{+}(x,-k)-1\right)e^{-ixk} \big], \qquad \qquad k<0.
		\end{cases}
		\end{align}
		
	\end{itemize}


	Let us 
	record here the fact that $m_+(x,0)$ converges rapidly to $\pm 1$ as $x \rightarrow \pm\infty$,
	as one can see from \eqref{Mestimates1} and \eqref{Mestimates2} (with $s=0$) and using the oddness of $m_+$.
	
	\begin{lem}
		For $\epsilon \in \{-1,1\}$ we have, for all $r=1,2,\dots$
		\begin{align}\label{m+0}
		& \chi_{\epsilon}(x) \big| \partial_x^\alpha \big( m_+(x,0)^r - \epsilon^r \big) \big| 
		\lesssim \jx^{-s} \mathcal{W}_{\pm}^{s+1}(x), \qquad \alpha = 0,1,
		\end{align}
		for all $s \geq 0$.
		In particular, under our assumption \eqref{VassumeWei}
		we have
		\begin{align}\label{m+0'}
		\chi_{\epsilon}(x) \big| \partial_x^\alpha \big( m_+(x,0)^r - \epsilon^r \big) \big| 
		\lesssim \jx^{-\gamma+1}, \qquad r=1,2,3,4, \quad \alpha = 0,1.
		\end{align}
	\end{lem}
	


	

	\medskip
	We now gather some properties of the $\#$-generalized eigenfunctions
	that can be seen from the explicit formulas and using 
	the estimates for the Jost functions in Lemma \ref{lem:Mestimates}.
	
	\begin{lem}\label{lem:Kregular}
		With the definitions \eqref{Ksharpdecomp}-\eqref{KR} we have
		
		\setlength{\leftmargini}{1.5em}
		\begin{itemize}
			
			\item $k \mapsto \mathcal{K}_{\pm}^{\#}(x,k)$ is Lipschitz continuous 
			with $\mathcal{K}_{\pm}^{\#}(x,0) = 0$.
			More precisely, $\mathcal{K}_{\pm}^{\#}(x,k) = a_\pm^+(k) e^{ikx} + a_\pm^-(k) e^{-ikx}$
			with Lipschitz coefficients $a_{\epsilon_1}^{\epsilon_2}$, $\epsilon_1,\epsilon_2 \in \{+,-\}$,
			with $a_{\epsilon_1}^{\epsilon_2}(0)=0$. 
			
			\smallskip
			\item $k \mapsto \mathcal{K}_R^{\#}(x,k)$ is Lipschitz continuous with $\mathcal{K}_{R}^\#(x,0)=0$. 
			

\end{itemize}

\end{lem}

\begin{proof}
These statements are a direct verification, using that $T$ and $R_\pm$ are smooth
and $R_\pm(0)=0$ and the estimates \eqref{Mestimates1}.
\end{proof}

	\medskip
	\subsection{The `sharp' transform}
	In analogy with \eqref{tildeF}, for $f \in \mathcal{S}$, we define the following `sharp' transform:
	\begin{align}\label{sharpF}
	\mathcal{F}^\# [f](k) = f^\#(k) = \int \overline{\mathcal{K}^\#(x,k)} f(x)\,dx.
	\end{align}
	Here is the analogue of Proposition \ref{proptildeF}:
	
	\begin{prop}\label{propsharpF}
		We have 
		\begin{align}\label{FsharpL2}
		{\big\| f^\# \big\|}_{L^2} = {\| f \|}_{L^2}, \quad \forall f\in L^{2} 
		\end{align}
		and
		\begin{align}\label{Fsharp-1}
		\big(\mathcal{F}^\#\big)^{-1}[\phi](x) = \int \mathcal{K}^\#(x,k) \phi(k)\,dk.
		\end{align}
		Also, if $D:=\sqrt{-\partial_{xx}+V}$, for any symbol $m$ we have 
		\begin{align}\label{Fsharpm}
		m(D) = \big(\mathcal{F}^\#\big)^{-1} m(k) \mathcal{F}^\#,
		\end{align}
		so that in particular 
		$\left(-\partial_{xx}+V\right) 
		= \big(\mathcal{F}^\#\big)^{-1} k^{2} \mathcal{F}^\#$.
		Moreover,
		
		\setlength{\leftmargini}{1.5em}
		\begin{itemize}
			
			\item If $\phi\in L^1$, then $\phi^\#$ is a Lipschitz continuous bounded function
			that decays to zero at infinity.
			
			\smallskip
			\item There exists $C>0$, depending on ${\| V \|}_{L^1}$,
			such that one has
			\begin{equation}\label{kusharp}
			\frac{1}{C} {\| u \|}_{H^1} \leq {\big\| \jk u^\# \big\|}_{L^2} \leq C 
			{\| u \|}_{H^1}.
			\end{equation}
			
			\smallskip
			\item We have
			\begin{align}\label{eq:weiF}
			{\| \partial_{k} f^\# \big\|}_{L^2} \lesssim {\| \jx f \|}_{L^2}.
			\end{align}
			Moreover, for any $\beta \in [0, \gamma - 5/2]$,
			\begin{align}\label{eq:weiFR}
			{\Big\| \partial_{k} \int \mathcal{K}^\#_R(x,k) f(x) \, dx \Big\|}_{L^2}
			\lesssim {\big\| \jx^{-\beta} f \big\|}_{L^2}.
			\end{align}
			
		\end{itemize}
		
	\end{prop}
	
	\begin{proof}
		\eqref{FsharpL2} follows directly from \eqref{tildeFL2} and the definition \eqref{sharpF}
		that is, $(\mathcal{F}^\# f)(k) = \sgn(k) (\wtF f)(k)$.
		To verify \eqref{Fsharp-1} we write
		\begin{align*}
		\big( (\mathcal{F}^\#)^{-1} \circ \mathcal{F}^\# g \big) (x) & =
		\int_{\R_k} \mathcal{K}^\#(x,k) \Big( \int_{\R_y} \overline{\mathcal{K}^\#(y,k)} g(y)\,dy \Big) dk
		\\
		& = \int_{\R_k} \mathcal{K}(x,k) \Big( \int_{\R_y} \overline{\mathcal{K}(y,k)} g(y)\,dy \Big) dk = g(x)
		\end{align*}
		having used \eqref{tildeF-1}. Similarly 
		\begin{align*}
		\big(\mathcal{F}^\# \circ (\mathcal{F}^\#)^{-1} g \big) (k) & =
		\int_{\R_y} \overline{\mathcal{K}^\#(y,k)} 
		\Big( \int_{\R_\ell} \mathcal{K}^\#(y,\ell) g(\ell)\,d\ell \Big) dy
		\\
		& = \int_{\R_y} \sgn(k) \overline{\mathcal{K}(y,k)} 
		\Big( \int_{\R_\ell} \mathcal{K}(y,\ell) \sgn(\ell) g(\ell)\,d\ell \Big) dy
		\\
		& =  \sgn(k) \wtF \big( \wtF^{-1}(\sgn(\cdot) g)  \big) = g(k).
		\end{align*}
		
		The intertwining identity \eqref{Fsharpm} follows along the same lines of the first calculation above using the analogue for $\wtF$.
		
		The fact that $\mathcal{F}^\#$ maps $L^1$ to bounded functions that decay at infinity
		follows from the Riemann-Lebesgue lemma.
		Continuity follows from dominated convergence, since $k \mapsto \K^\#(x,k)$ is continuous.

For \eqref{kusharp} one can first use \eqref{FsharpL2} and \eqref{Fsharpm} to see that
\begin{align*}
{\| u \|}_{H^1}^2 & = - \int u u_{xx} \, dx + \int u^2 = 
  \int u Hu \, dx - \int u Vu \, dx + \int u^2 \, dx
\\
& = 
{\big\| |k| u^\sharp \big\|}_{L^2}^2 - \int u Vu \, dx 
+ {\| u^\sharp \|}_{L^2}^2.
\end{align*}
Then, using \eqref{Fsharp-1} and $|\mathcal{K}^\sharp| \lesssim 1$
one can estimate ${\| u \|}_{L^\infty} \lesssim {\| u^\sharp \|}_{L^1} \lesssim {\| \jk u^\sharp \|}_{L^2}$
and
\begin{align*}
\Big| \int u Vu \, dx \Big| \lesssim {\| V \|}_{L^1}  {\| \jk u^\sharp \|}_{L^2}^2
\end{align*}
to obtain the first inequality in \eqref{kusharp};
the second inequality follows similarly using ${\| u \|}_{L^\infty} \lesssim {\| u \|}_{H^1}$.

		For \eqref{eq:weiF} we first use \eqref{Ksharpdecomp}-\eqref{KsharpS} to write 
		\begin{align}\label{dkfsharp}
		\begin{split}
		\partial_k f^\# (k) & = \int_\R \partial_k \overline{\mathcal{K}^\#_0(x,k)} f(x) \, dx
		+ \int_\R \chi_+(x) \partial_k \overline{\mathcal{K}^\#_+(x,k)} f(x) \, dx
		\\
		& + \int_\R \chi_+(x) \partial_k \overline{\mathcal{K}^\#_-(x,k)} f(x) \, dx
		+ \int_\R \partial_k \overline{\mathcal{K}^\#_R(x,k)} f(x) \, dx.
		\end{split}
		\end{align}
		For the first term on the right-hand side above we see from \eqref{K0} that
		\begin{align*}
		\int_\R \partial_k \overline{\mathcal{K}^\#_0(x,k)}  f(x) \, dx
		= T(0) \int_\R m_+(x,0) (-ix) e^{ixk} f(x) \, dx
		\end{align*}
		so that taking $L^2$-norms and using (flat) Plancharel we get a bound by ${\| x f \|}_{L^2}$.
		For the second and third term on the right-hand side of \eqref{dkfsharp} 
		we recall the definitions \eqref{K+}-\eqref{K-}
		and use again Plancharel together with the fact that that $|\partial_k a_{\epsilon_1}^{\epsilon_2}(k)|\lesssim 1$,
		$\epsilon_1,\epsilon_2 \in \{+,-\}$.
		
		We are left with estimating the last term  on the right-hand side of \eqref{dkfsharp},
		for which we prove directly the stronger estimate \eqref{eq:weiFR}.
		We start from the formula \eqref{KR} to write
		\begin{align}\label{dkKR}
		\begin{split}
		& \int_\R \partial_k \overline{\mathcal{K}^\#_R(x,k)} f(x) \, dx
		\\
		& = \int_\R \chi_+(x) \, \partial_k
		\Big\{ \big[ \mathbf{1}_+(k) (T(k)-T(0))(m_{+}(x,0)-1) + T(k) \mathbf{1}_+(k) (m_{+}(x,k)-m_{+}(x,0)) 
		\\
		& - \mathbf{1}_-(k) (m_{+}(x,k)-m_{+}(x,0)) \big] e^{ikx} 
		- \mathbf{1}_- (k) R_{+}(-k) (m_{+}(x,-k)-1) e^{-ixk} \Big\} f(x) \, dx
		\\
		& + \int_\R \chi_-(x) \, \partial_k \Big\{ \big[ \mathbf{1}_+(k)(m_{-}(x,-k)-m_{-}(x,0))e^{ixk}
		+ \mathbf{1}_+(k)  R_{-}(k)(m_{-}(x,k)-1 )e^{-ixk}
		\\
		& - \mathbf{1}_-(k) (T(-k)-T(0)) (m_{-}(x,0)-1) - \mathbf{1}_-(k) T(-k)(m_{-}(x,-k)-m_{-}(x,0)) \big] e^{ixk}
		\Big\} f(x) \, dx.
		\end{split}
		\end{align}
		Note that the symbol in curly brackets is Lipschitz in $k$.
		Since the contributions from all four lines on the right-hand side of \eqref{dkKR} can be treated in the same way, 
		we only look at the first one.
		We can write it as
		\begin{align}
		\nonumber
		& \int_\R \chi_+(x) \, 
		\big\{ \partial_k a_+^+(k) (m_{+}(x,0)-1) 
		\\ 
		\label{dkKR1}
		& \qquad + \partial_k T(k) \mathbf{1}_+(k) 
		(m_{+}(x,k)-m_{+}(x,0)) \big\} e^{ixk} f(x) \, dx
		\\
		\label{dkKR2}
		+ & \int_\R \chi_+(x) \, T(k) \partial_k \big[\mathbf{1}_+(k) (m_{+}(x,k)-m_{+}(x,0)) \big] \, e^{ixk} 
		f(x) \, dx
		\\
		\nonumber
		+ & \int_\R \chi_+(x) \,\big\{ a_+^+(k) (m_{+}(x,0)-1) 
		\\
		\label{dkKR3}
		& \qquad + T(k) \mathbf{1}_+(k) (m_{+}(x,k)-m_{+}(x,0)) \big\}
		(ix) e^{ixk} f(x) \, dx.
		\end{align}
		For \eqref{dkKR1}, using $|\partial_k a_+^+(k)|, |\partial_kT(k)| \leq 1$, we have
		\begin{align*}
		\big| \eqref{dkKR1} \big| & \lesssim \Big| \int_\R \chi_+(x) \,(m_{+}(x,0)-1) e^{ixk} f(x) \, dx \Big|
		\\
		& + \Big| \int_\R \chi_+(x) (m_{+}(x,k)- 1 
		) e^{ixk} f(x) \, dx \Big|.
		\end{align*}
		We can then use Plancharel with \eqref{m+0'} for the first line and \eqref{eq:mPDO+-1} ($\gamma > 3/2 + \beta$)
		for the second line, to estimate the $L^2$ norm of \eqref{dkKR1} by ${\| \jx^{-\beta}f \|}_{L^2}$ as desired.
		
		For \eqref{dkKR2} 
			we first notice that
			$\partial_k [ \mathbf{1}_+(k) (m_{+}(x,k)-m_{+}(x,0)) ] = \mathbf{1}_+(k) \partial_k m_{+}(x,k)$,
			and then use  $|T(k)| \lesssim 1$ and \eqref{eq:pdomkd} ($\gamma > 5/2 + \beta$)
			to estimate
			\begin{align*}
			{\big\| \eqref{dkKR2} \big\|}_{L^2} & \lesssim 
			{\Big\| \int_\R \chi_+(x)  \partial_k m_+(x,k) e^{ixk} f(x) \, dx \Big\|}_{L^2}
			\lesssim {\| \jx^{-\beta} f \|}_{L^2}.
			\end{align*}
		
		
		Finally for \eqref{dkKR3} we can estimate
		\begin{align*}
		\big| \eqref{dkKR3} \big| & \lesssim \Big| \int_\R \chi_+(x) \, (m_{+}(x,0)-1) e^{ixk} (ix) f(x) \, dx \Big|
		\\
		& + \Big| \int_\R \chi_+(x) (m_{+}(x,k)-m_{+}(x,0)) e^{ixk} (ix) f(x) \, dx \Big|;
		\end{align*}
		using again Plancharel with \eqref{m+0'} for the first line,
		and the $L^2$-bound for PDOs \eqref{eq:mPDO+-1}, we obtain a bound by ${\| \jx^{-\beta} f\|}_{L^2}$ as desired.
		This concludes the proof of \eqref{eq:weiFR} and of the proposition.
	\end{proof}

	\medskip
	\begin{definition}[Singular and Regular `Projections']
		According to the decomposition \eqref{Ksharpdecomp}-\eqref{KsharpS}, 
		given $\phi\in \mathcal{S}$ we can write
		\begin{align}\label{decompphi}
		\begin{split}
		& \phi = \phi_S + \phi_R, \qquad \phi_S = \phi_0 + \phi_+ + \phi_-,
		\\
		& \phi_\ast(x) :=\frac{1}{\sqrt{2\pi}} \chi_\ast (x)  \int \mathcal{K}_\ast^\#(x,k) \, 
		(\mathcal{F}^\# \phi) (k) \, dk, \qquad \ast \in \{0,+,-\}
		\\
		& \phi_R(x) :=\frac{1}{\sqrt{2\pi}}  \int \mathcal{K}_R^\#(x,k) \, 
		(\mathcal{F}^\# \phi) (k) \, dk.
		\end{split}
		\end{align}
		We can then extend these definitions to $L^2$ analogously to Lemma \ref{proptildeF}.
		Note that the identities \eqref{decompphi} are the same with $\wtF$ instead of $\mathcal{F}^\#$.
	\end{definition}
	
		
	\medskip
	\subsection{The case of an even resonance}\label{sseceven}
	To conclude this section, let us comment on the case of an even resonance.
	If the zero energy resonance is even the distorted Fourier transform is already continuous,
	see Remark \ref{remdisc},
	so we do not need to use the `sharp transform'. 
	In what follows we will only work in the case of an odd resonance, with the understanding
	that the case of an even resonance follows with minor modifications,
	and we will only briefly comments about this in a few instances.


	\bigskip
	\section{Linear estimates: decay and smoothing}\label{seclinest}
	In this section we provide all the main linear estimate that we will need in the rest of the paper.
	We start by collecting the standard linear decay estimates, 
	which we phrase in terms of the sharp transform, Lemma \ref{lem:pointwiseH}, 
	Corollaries \ref{cor:pointwisesingular} and \ref{cor:regularLinfty}.  
	Then we establish smoothing estimates for non-generic potentials, see Lemma \ref{lem:smoothingsim}.  
	Finally, we show an estimate of ``improved local decay" type 
	specific to the low frequency-part of (pseudo) linear solutions with a coefficient vanishing at $0$, Lemma \ref{lem:lowlocaldecay},
	and a version of $L^2$-local decay for differentiated flows, Lemma \ref{lem:localderivative}.
	
	Together with the structure of the nonlinear spectral distribution (NSD), 
	these linear estimates will be the main ingredients that allow us to to close our nonlinear bounds 
	in Sections \ref{sec:Cubic} and \ref{sec:estimateregular}.

	\medskip
	\subsection{Dispersive decay}\label{ssecdisp}
	We start with a basic stationary phase-type lemma (see, for example, \cite{GPR}): 
	
	\begin{lem}
		\label{lemstat}
		Consider a function $a(x,k)$ defined on $B \times \mathbb{R}_{+}$, for some $B \subset \R$, and such that
		\begin{equation}
		\label{hypest}
		| a(x, k)| +|k| |\partial_{k}a(x,k)| \lesssim 1, \quad \forall x \in B, \, \forall k \in \mathbb{R}_{+}
		\end{equation}
		and for every  $X \in \mathbb{R}$, consider the oscillatory  integral
		$$
		I(t,X,x) = \int_{0}^{\infty} e^{it(k-X)^2} a(x,k)  h(k) \, dk, \quad t>0, \, x \in B.
		$$
		Then, we have the  estimate
		\begin{equation}
		\label{decaylem} | I(t,X,x)|  \lesssim \frac{1}{\sqrt t} \| h  \|_{L^\infty} + \frac{1}{t^{\frac{3}{4}}} \| \partial_k h \|_{L^2}
		\end{equation}
		which is uniform for $X\in \mathbb{R}$, $t >0$ and $x\in B$.
	\end{lem}
	
	Given the decomposition \eqref{Ksharpdecomp}, 
	we can directly obtain  pointwise decay estimates associated to each piece of the decomposition. 
	We start with the singular part.
	
	\begin{cor}\label{cor:pointwisesingular}
		Using the notation in \eqref{KsharpS}, we have
		\begin{align}\label{eq:singulardecaysep}
		\left\Vert\chi_\ast(x)\int\mathcal{K}_{\ast}^{\#}(x,k)e^{ik^{2}t}h(k)\,dk\right\Vert_{L^{\infty}_x}
		\lesssim\frac{1}{\sqrt{t}}{\big\| h \big\|}_{L^\infty_k}
		+\frac{1}{t^{\frac{3}{4}}}{\big\| \partial_k h \big\|} _{L^2_k}
		\end{align}
		for $\ast\in\{0,+,-\}$. 
		In particular, we get
		\begin{align}\label{eq:singulardecay}
		\left\Vert  \int\mathcal{K}^{\#}_S(x,k)e^{ik^{2}t}h(k)\,dk\right\Vert_{L^{\infty}_x}
		\lesssim\frac{1}{\sqrt{t}}{\big\| h \big\|}_{L^\infty_k}
		+\frac{1}{t^{\frac{3}{4}}}{\big\| \partial_k h \big\|} _{L^2_k}.
		\end{align}
	\end{cor}

	\begin{proof}
		We only consider the pieces with $k\geq0$ since $k<0$ is similar.  
		We define
		\begin{equation}
		J_{\ast}:=\chi_{\ast}(x)\int_{k\geq0} \mathcal{K}_{\ast}^{\#}(x,k)e^{ik^{2}t}h(k)\,dk,
		\end{equation}
		and consider the case $\ast=0,+$. The analysis of $J_-$ is the same. 
		We write $
		J_{\ast}= e^{ -i {\frac{x^2}{ 4t}}} I_{\ast} (t,X,x)$, for $X=  -x/(2 t)$
		with
		$$I_{\ast}(t,X,x)= \int_{0}^{+\infty} e^{it (k-X)^2}  \mathcal{K}_{\ast}^{\#}(x,k)e^{-ikx} h(k)\, dk.$$ 
		Thanks to Lemma \ref{estiTR},  
		given the explicit formulae from \eqref{K0} and \eqref{K+},
		we can use Lemma \ref{lemstat} with 
		$a(x,k) =T(0)m_+(x,0)$ and $a(x,k) =\chi_+(x)(T(k)-T(0))$ for $\ast=0$ and $\ast=+$ respectively. 
		These give
		$$ |J_{\ast}| \lesssim   \frac{1}{\sqrt t} \| h \|_{L^\infty} 
		+ \frac{1}{t^{\frac{3}{4}}} \| \partial_k h \|_{L^2}$$
		as desired.
	\end{proof}

	Recall the definition of the regular part from \eqref{KR}.
	We have the following $L^\infty$ decay for the regular part of  the Schr\"odinger flow:
	
	\begin{cor}\label{cor:regularLinfty}
		If $\jx^\gamma V(x)\in L^1$ with $\gamma\geq\beta+1$, then
		\begin{align}\label{regularinfty}
		\left\Vert \jx^\beta \int\mathcal{K}^{\#}_R(x,k)e^{ik^{2}t}h(k)\,dk\right\Vert_{L^{\infty}_x}
		\lesssim\frac{1}{\sqrt{t}}{\big\| h \big\|}_{L^\infty_k}
		+\frac{1}{t^{\frac{3}{4}}}{\big\| \partial_k h \big\|} _{L^2_k}.
		\end{align}
	\end{cor}
	
	\begin{proof}
		We only  analyze the case $k\geq0$ since the case $k<0$ is similar.
		From the definition of the regular part \eqref{KR}
		\begin{align}\label{locdecpr12ng0}
		\begin{split}
		& \int_{k\geq0} \mathcal{K}^{\#}_R(x,k)e^{ik^2 t}h(k)\,dk = J_+(x) + J_-(x),
		\\
		& J_+(x) := 
		\chi_+(x)\int_{k\geq0}e^{ik^2t}  \big[ (T(k)-T(0))(m_{+}(x,0)-1) 
		\\ & \qquad \qquad \qquad + T(k) \left(m_{+}(x,k)-m_{+}(x,0) \right)  \big] e^{ikx} h(k)\,dk,
		\\
		& J_-(x) := \chi_-(x)\int_{k\geq0} e^{ik^2t} 
		\big[\left(m_{-}(x,-k)-m_{-}(x,0)\right)e^{ixk}
		\\  & \qquad \qquad \qquad+ R_{-}(k)\left(m_{-}(x,k)-1\right)e^{-ixk} \big] h(k)\,dk.
		\end{split}
		\end{align}
		From Lemma \ref{lem:Mestimates} and Lemma \ref{estiTR} we know that 
		\begin{align}
		& \label{m+}
		\jx^\beta|m_\pm(x,k) - 1 | \lesssim \jx^\beta\mathcal{W}^1_\pm, \quad x \geq \mp 1,
		\\
		& \label{m+2}
		\jx^\beta| \partial_{k}m_\pm(x,k)| 
		\lesssim \frac{1}{|k|}\jx^\beta\mathcal{W}^{1}_\pm, \quad x \geq  \mp 1,
		\\
		& \label{Tk2}  | \partial_{k} T(k) | + | \partial_{k}R_{+} (k)|  \lesssim \frac{1}{\jk}.
		\end{align}
		We then write $\jx^\beta J_{+}= e^{ -i {\frac{x^2}{ 4t}}} I_{+} (t,X,x)$ with $X= - x/(2 t)$
		and
		\begin{align*}
		I_{+}(t,X,x) = \int_{0}^{+\infty} e^{it (k-X)^2}  \jx^\beta \big[(T(k)-T(0))(m_{+}(x,0)-1)
		\\ + T(k) \left(m_{+}(x,k)-m_{+}(x,0) \right) \big] h(k)\, dk.
		\end{align*}
		Thanks to \eqref{m+}-\eqref{Tk2}, we can use Lemma \ref{lemstat} with 
		$$a(x,k) = \jx^\beta\left[(T(k)-T(0))(m_{+}(x,0)-1)+ T(k) \left(m_{+}(x,k)-m_{+}(x,0) \right)  \right].$$
		This yields
		$$ |J_{+}| \lesssim   \frac{1}{\sqrt t} \| h \|_{L^\infty} 
		+ \frac{1}{t^{\frac{3}{4}}} \| \partial_k h \|_{L^2}$$
		as desired. The term $J_-$ can be estimated in the same way.
	\end{proof}

	\smallskip
	Putting together the results from Corollary \ref{cor:pointwisesingular}
	and Corollary \ref{cor:regularLinfty} we can obtain the following pointwise decay estimate:
	
	\begin{lem}\label{lem:pointwiseH}
		Suppose $\jx^\gamma V(x)\in L^1$ with $\gamma\geq1$ and assume the Schr\"odinger operator $H=-\partial_{xx}+V$ 
		has an odd resonance. 
		Then the perturbed Schr\"odinger flow has the following dispersive estimate:
		for $t\geq0$
		\begin{equation}\label{eq:linearpoinwiseH}
		{\big\| e^{itH}h \big\|}_{L^\infty_x} = \left\Vert\int \mathcal{K}^{\#}(x,k) 
		e^{ik^2t} h^{\#}(k)\,dk \right\Vert_{L^\infty_x} \lesssim \frac{1}{\sqrt{t}} {\big\| h^\# \big\|}_{L^\infty_k}
		+\frac{1}{t^{\frac{3}{4}}}{\big\| \partial_k h^\# \big\|} _{L^2_k}.
		\end{equation}
	\end{lem}

	\begin{proof}
		From Proposition \ref{propsharpF} we have
		\begin{align}
		e^{itH}h 
		= \int \mathcal{K}^{\#}(x,k) \mathcal{F}^{\#}(e^{itH} h) \,dk
		& = \int \mathcal{K}^{\#}(x,k) e^{itk^2} \wt{h}^{\#}(k)\,dk.
		\end{align}
		By the decomposition \eqref{Ksharpdecomp}, 
		the desired decay estimate is a direct consequence of 
		Corollary \ref{cor:pointwisesingular} and Corollary \ref{cor:regularLinfty} with $\beta=0$.
	\end{proof}

	\subsection{Smoothing estimates}\label{ssecsmooth}
	In this subsection we prove some smoothing-type estimates for Schr\"odinger flows with non-generic potentials
	which will be an important tool for the nonlinear estimates of Sections \ref{sec:Cubic} and \ref{sec:estimateregular}.
	These estimates resemble the Kato-$\frac{1}{2}$ smoothing estimates; see also Mizumachi \cite{Miz} for 
	smoothing estimate for the Schr\"odinger problem and Krieger-Nakanishi-Schlag for 
	smoothing estimates in the Klein-Gordon setting \cite{KNS}.
	Here we provide self-contained proofs 
	using the Fourier transform in time, and allowing the additional presence of fairly general classes of symbols. 

	\begin{lem}\label{lem:smoothingsim}
		Given a function $\mathcal{Q}: \R_x\times \R_k \mapsto \C$ such that, for some $\beta \in \R$,
		\begin{equation}\label{eq:Qbound}
		\sup_{x\in\mathbb{R},\,k\in\mathbb{R}} \big| \left\langle x\right\rangle^{\beta}\mathcal{Q}(x,k) \big| < \infty,
		\end{equation}
		and $\phi: \R \mapsto \C$ such that $|\phi(k)|\lesssim \sqrt{|k|}$,
		then, for all $t\geq0$,
		\begin{equation}\label{eq:moregeneralsmoothing}
		\left\Vert \left \langle x\right\rangle^{\beta} \int_\R \mathcal{Q}(x,k) \phi(k) e^{ik^{2}s}h(k)\,dk\right\Vert _{L_{x}^{\infty}L_{s}^{2}([0,t])}
		\lesssim \left\Vert h\right\Vert _{L^{2}}.
		\end{equation}
		
		Under the same assumptions, the following inhomogeneous estimate holds:
		\begin{align}\label{eq:generalimhomsmoothing}
		\left\Vert \int_0^t \left[ \int_\R e^{-ik^{2}s} \overline{\phi}(k) \overline{\mathcal{Q}}(y,k)F(s,y)\,dy \right] \,ds\right\Vert_{L_{k}^{2}}
		\lesssim & \left\Vert \left\langle x\right\rangle^{-\beta}F\right\Vert_{L_{x}^{1}L_{s}^{2}([0,t])}.
		\end{align}
		
	\end{lem}
	
	\begin{proof}
		Without loss of generality we may assume $\beta=0$ and restrict the integral in \eqref{eq:moregeneralsmoothing} to $k\in[0,\infty)$.
		Making a change of variable $k^{2}=\lambda$, $2kdk=d\lambda$, one has
		\begin{align*}
		I(s,x) := \int_0^\infty \mathcal{Q}(x,k)\phi(k)e^{ik^{2}s}h(k)\,dk
		& = \int_0^{\infty}\mathcal{Q}(x,\sqrt{\lambda})\phi(\sqrt{\lambda})
		e^{is\lambda}h(\sqrt{\lambda})\frac{1}{2\sqrt{\lambda}}\,d\lambda.
		\end{align*}
		Taking the $L^2_s$ norm, and then applying Plancherel's theorem in $s$ we obtain
		\begin{align*}
		{\| I(\cdot,x) \|}_{L^2_s([0,t])}^2 
		& = \left\| \int_\R \mathcal{Q}(x,\sqrt{\lambda}) \, \mathbf{1}_{[0,\infty)}(\lambda)
		\phi(\sqrt{\lambda})e^{is\lambda} h(\sqrt{\lambda})\frac{1}{2\sqrt{\lambda}}
		\,d\lambda \right\|_{L^2_s(\R)}^2
		\\
		& \lesssim \int_0^{\infty} \left|  \mathcal{Q}(x,\sqrt{\lambda}) \frac{\phi(\sqrt{\lambda})}{\sqrt{\lambda}} 
		h(\sqrt{\lambda}) \right|^{2}\,d\lambda 
		\\
		& \lesssim \int_0^{\infty} \big| h(\sqrt{\lambda}) \big|^2 \frac{1}{\sqrt{\lambda}} \,d\lambda 
		\lesssim \int_{0}^{\infty} | h (k) |^{2} \,dk
		\lesssim \left\Vert h\right\Vert _{L^{2}}^2 .
		\end{align*}

		To obtain the inhomogeneous estimate \eqref{eq:generalimhomsmoothing}
		we test the expression on the left-hand side against an arbitrary function $h \in L^2$:
		\begin{align}
		\left \langle  h,\int_0^t \left[ \int_\R e^{-ik^{2}s} \overline{\phi}(k) 
		\overline{\mathcal{Q}}(y,k)F(s,y)\,dy\right] ds \right \rangle 
		= & \int_0^t \int_\R \Big( \int_\R e^{ik^{2}s}\phi(k)\mathcal{Q}(y,k) h(k)\,dk \Big)  \overline{F}(s,y)\,dyds
		\nonumber. 
		\end{align}
		Applying H\"older's inequality followed by the homogeneous estimate \eqref{eq:moregeneralsmoothing} we get
		\begin{align}
		& \left| \int_0^t \int_\R \Big( \int e^{ik^{2}s}\phi(k)\mathcal{Q}(y,k) h(k)\,dk \Big) \overline{F}(s,y) \,dyds\right|
		\nonumber
		\\
		& \lesssim \left\Vert \left\langle x\right\rangle ^{\beta}\int\mathcal{Q}(x,k)\phi(k)e^{ik^{2}s}h(k)\,dk\right\Vert _{L_{x}^{\infty} L_s^2([0,t])} 
		\big\| \jx^{-\beta}F \big\|_{L_{x}^{1}L_s^2([0,t])}
		\nonumber
		\\
		& \lesssim \left\Vert h\right\Vert _{L^{2}} \big\| \jx^{-\beta}F \big\|_{L_{x}^{1}L_s^2([0,t])}
		\nonumber
		\end{align}
		which implies \eqref{eq:generalimhomsmoothing} by duality.
	\end{proof}

	\smallskip	
	\begin{rem}\label{rem:classicalsmoothing}
		Assuming that the potential $V$ is generic
		and taking $\beta=-1$, $\mathcal{Q}(x,k)=\mathcal{K}(x,k)/\sqrt{|k|}$, $\phi(k)=\sqrt{|k|}$, and $h(k)=\widetilde{f}(k)$ in \eqref{eq:moregeneralsmoothing},
		gives the smoothing estimates for the perturbed Schr\"odinger flow
		$$\left\Vert \left\langle x\right\rangle ^{-1}e^{itH}f\right\Vert _{L_{x}^{\infty}L_{s}^{2}}\lesssim\left\Vert f\right\Vert _{L^{2}};$$
		see Mizumachi \cite{Miz} where the smoothing estimate is proved with a slightly stronger weight.
		
		More classically, 
		with $\beta=0$,  $\mathcal{Q}(x,k)=e^{ikx}$,  $\phi(k)=\sqrt{k}$, and $h(k)=\widehat{f}(k)$
		in 
		\eqref{eq:moregeneralsmoothing} 
		gives the classical Kato-$\frac{1}{2}$ smoothing estimate
		$$\big\| |\partial_x|^{1/2} e^{it\partial_{xx}}f \big\|_{L_{x}^{\infty}L_{s}^{2}}\lesssim\left\Vert f\right\Vert _{L^{2}}.$$
	\end{rem}

	For our later applications, we record the following special cases of the smoothing estimates of Lemma \ref{lem:smoothingsim}
	where we restrict the frequency integral to small or large frequencies.
	
	\begin{cor}\label{corsmoothing}
		Given $\mathcal{Q}:\R_x\times \R_k \mapsto \C$ such that for some $\beta\in\mathbb{R}$
		\begin{equation}\label{eq:Qbound'}
		\sup_{x\in\mathbb{R},\,k\in\mathbb{R}}\left|\left\langle x\right\rangle ^{\beta}\mathcal{Q}(x,k)\right|
		<\infty,
		\end{equation}
		then, for all $t>0$,
		\begin{equation}
		\left\Vert \left\langle x\right\rangle ^{\beta}\int\mathbf{1}_{\{|k|\lesssim1\}}\,k\,\mathcal{Q}(x,k) e^{ik^{2}s}h(k)\,dk
		\right\Vert_{L_{x}^{\infty}L_{s}^{2}([0,t])}
		\lesssim\left\Vert h\right\Vert _{L^{2}}.\label{eq:smoothingQ}
		\end{equation}
		By duality, the following inhomogeneous estimate holds:
		\begin{align}\label{eq:smoothingQim}
		\left\Vert \int_0^t \left[\mathbf{1}_{\{|k|\lesssim1\}} \int e^{-ik^{2}s} \,k\, \overline{\mathcal{Q}}(y,k) F(s,y)\,dy\right]\,ds\right\Vert _{L_{k}^{2}}
		& \lesssim {\big\| \left\langle x\right\rangle^{-\beta}F \big\|}_{L_{x}^{1}L_{s}^{2}([0,t])}.
		\end{align}
		
		Under the same assumptions above
		\begin{equation}
		\left\Vert \left\langle x\right\rangle ^{\beta}\int\mathbf{1}_{\{ |k|\gtrsim1 \}}
		\mathcal{Q}(x,k)e^{ik^{2}s}h(k)\,dk\right\Vert _{L_{x}^{\infty}L_{s}^{2}([0,t])}
		\lesssim\left\Vert h\right\Vert _{L^{2}}.\label{eq:smoothingQh}
		\end{equation}
		By duality, the following inhomogeneous estimate holds:
		\begin{align}\label{eq:smoothingQimh}
		\left\Vert \int_0^t \left[\mathbf{1}_{\{ |k|\gtrsim1 \}} \int e^{-ik^{2}s}\overline{\mathcal{Q}}(y,k)F(s,y)\,dy\right]\,ds\right\Vert _{L_{k}^{2}}
		\lesssim & {\big\| \left\langle x\right\rangle^{-\beta}F \big\|}_{L_{x}^{1}L_{s}^{2}([0,t])}.
		\end{align}
	\end{cor}
	
	\begin{proof}
		These results follow directly from Lemma \ref{lem:smoothingsim} by choosing $\phi(k)=\mathbf{1}_{\{|k|\lesssim 1\}} k$  
		and $\phi(k)=\mathbf{1}_{\{|k|\gtrsim 1\}}$.
	\end{proof}


	\medskip
	\subsection{Local decay}\label{sseclocdec}
	We have the following improved local decay estimate for flows with improved low-frequency behavior:
	
	\begin{lem}\label{lem:lowlocaldecay}
		Given a function $\mathcal{Q}(x,k)$ such that for some $\beta\in\mathbb{R}$
		\begin{equation}\label{eq:QboundP}
		\sup_{x\in\mathbb{R},\,k\in\mathbb{R}}\left|\left\langle x\right\rangle^{\beta}
		\partial_{k}^{j}\mathcal{Q}(x,k)\right|<\infty, \qquad j=0,1,
		\end{equation}
		and a smooth function $\phi$ such that $|\phi(k)|\lesssim |k|$ for $|k|\leq1$
			and $|\phi(k)|$ is bounded for $|k|\geq1$,
		then one has
		\[
		\left\Vert \left\langle x\right\rangle^{\beta}\int\mathbf{1}_{\{\pm k\geq0\}}\phi(k)
		\mathcal{Q}(x,k)e^{ik^{2}t}h(k)\,dk\right\Vert_{L_{x}^{\infty}}
		\lesssim \frac{1}{\left\langle t\right\rangle }\left\Vert h\right\Vert _{H^{1}}.
		\]
	\end{lem}
	
	\begin{proof}
		This follows from a standard integration by parts argument as in \cite[Lemma 3.6]{NLSV}.
		For completeness we give some details.
		Let us consider the case of $k\geq0$
		and integrate by parts in $k$ 
		to obtain the following:
		\begin{align*}
		2it \int\mathbf{1}_{\{k\geq0\}}\phi(k)\mathcal{Q}(x,k)e^{ik^{2}t}h(k)\,dk 
		=\int\mathbf{1}_{\{k\geq0\}}\phi(k) k^{-1} \mathcal{Q}(x,k)h(k)\,\partial_k e^{ik^{2}t} \, dk
		\\
		= \phi'(0) h(0) \mathcal{Q}(x,0)
		-\int\mathbf{1}_{\{k\geq0\}}e^{ik^{2}t}\partial_{k}\big( \phi(k) k^{-1} \mathcal{Q}(x,k)h(k) \big)\,dk.
		\end{align*}
		By Sobolev's embedding and \eqref{eq:QboundP}, 
		$|\phi'(0)\langle x \rangle^{\beta}h(0)\mathcal{Q}(x,0)|\lesssim\left\Vert h\right\Vert _{H^{1}}$.
		For the second  term on the right-hand side above, one has
		\begin{align*}
		& \left\langle x\right\rangle^{\beta}\left|\int \mathbf{1}_{\{k\geq0\}} e^{ik^{2}t}
		\partial_{k}\left(\phi(k)/k\mathcal{Q}(x,k)h(k)\right)\,dk\right|
		\\
		& \lesssim\left\langle x\right\rangle^{\beta}
		\int \mathbf{1}_{\{k\geq0\}} \left|
		\partial_{k} \big(\phi(k)/k \mathcal{Q}(x,k) h(k) \big)\right|\,dk
		\lesssim \left\Vert h\right\Vert _{H^{1}},
		\end{align*}
		where in the last line above we applied the Cauchy-Schwarz inequality
		in $k$ under the assumption \eqref{eq:QboundP}. 
	\end{proof}

Finally, we give a local decay estimate for the derivative of the Schr\"odinger flow, once again allowing the presence 
of some general symbols.

	\begin{lem}\label{lem:localderivative}
		Consider $\mathfrak{Q}(x,k)$ such that for some $\beta\in\mathbb{R}$
		\begin{equation}\label{eq:QboundP2}
	\sup_{x\in\mathbb{R},\,k\in\mathbb{R}}\left| \jx^\beta \partial_x^j \partial_k^\rho \mathfrak{Q}(x,k)\right|\lesssim 1,
	\qquad j,\rho=0,1,
		\end{equation}
		and assume that $\jx^\beta \partial_x^j \partial_k^\rho \mathfrak{Q}(x,k)$ 
		are symbols of $L^2$-bounded PDOs.
		Then one has
		\begin{equation}\label{eq:QdecayL2}
		\left\Vert \left\langle x\right\rangle ^{\beta-1}\partial_{x}
		\int\mathbf{1}_{\{\pm k\geq0\}}e^{ikx}\mathfrak{Q}(x,k)e^{ik^{2}t}h(k)\,dk\right\Vert_{L_{x}^{2}}
		\lesssim \frac{1}{\jt^{1/2}} \left\Vert h\right\Vert _{H^1}.
		\end{equation}
	\end{lem}
	
	\begin{proof}
		It suffices to look at the case $k>0$. We calculate:
		\begin{align}
		\nonumber
		& 
		\partial_{x}\int\mathbf{1}_{\{k\geq0\}}e^{ikx}\mathfrak{Q}(x,k)e^{ik^{2}t}h(k)\,dk
		\\
		\label{locald1}
		& = 
		\int\mathbf{1}_{\{k\geq0\}}ike^{ikx}\mathfrak{Q}(x,k)e^{ik^{2}t}h(k)\,dk
		+ 
		\int\mathbf{1}_{\{k\geq0\}}e^{ikx}\partial_{x}\mathfrak{Q}(x,k)e^{ik^{2}t}h(k)\,dk.
		\end{align}
		For the first term on the right-hand side above, 
		we integrate by parts in $k$ using $e^{ik^2t} = (2itk)^{-1}\partial_k e^{ik^2t}$:
		\begin{align*}
		\int\mathbf{1}_{\{k\geq0\}}ike^{ikx}\mathfrak{Q}(x,k)e^{ik^{2}t}h(k)\,dk
		& =
		\frac{1}{2t} \mathfrak{Q}(x,0)h(0) 
		-\frac{1}{2t}\int\mathbf{1}_{\{k\geq0\}}e^{ikx}\mathfrak{Q}(x,k)e^{ik^{2}t}\partial_{k}h(k)\,dk
		\\
		& -\frac{1}{2t}\int\mathbf{1}_{\{k\geq0\}}e^{ikx}\partial_{k}\mathfrak{Q}(x,k)e^{ik^{2}t}h(k)\,dk
		\\
		& -\frac{1}{2t}\int\mathbf{1}_{\{k\geq0\}}e^{ikx}ix\mathfrak{Q}(x,k)e^{ik^{2}t}h(k)\,dk.
		\end{align*}
		Using $\jx^\beta\mathfrak{Q}(x,k)$ and $\jx^\beta\partial_k\mathfrak{Q}(x,k)$ as symbols,
		by our assumptions, it follows that the three integrated terms above satisfy
		\begin{align}
		\label{eq:psebound1}
		& \left\Vert \left\langle x\right\rangle ^{\beta}\int\mathbf{1}_{\{k\geq0\}}e^{ikx}
		\mathfrak{Q}(x,k)e^{ik^{2}t}\partial_{k}h(k)\,dk\right\Vert _{L_{x}^{2}}
		\lesssim\left\Vert \partial_{k}h\right\Vert _{L^{2}},
		\\
		\label{eq:psebound2}
		& \left\Vert \left\langle x\right\rangle^{\beta-1}
		\int\mathbf{1}_{\{k\geq0\}}e^{ikx}ix\mathfrak{Q}(x,k)e^{ik^{2}t}h(k)\,dk\right\Vert_{L_{x}^{2}}
		\lesssim\left\Vert h\right\Vert_{L^{2}},
		\\
		\label{eq:psebound3}
		& \left\Vert \left\langle x\right\rangle ^{\beta} \int \mathbf{1}_{\{k\geq0\}}e^{ikx}
		\partial_{k}\mathfrak{Q}(x,k)e^{ik^{2}t}h(k)\,dk\right\Vert _{L_{x}^{2}}\lesssim\left\Vert h\right\Vert_{L^{2}}.
		\end{align}
		For the boundary term, by Sobolev's embedding we have
		\begin{equation}\label{eq:psebound4}
		\left\Vert \left\langle x\right\rangle ^{\beta-1}\mathfrak{Q}(x,0)h(0)\right\Vert_{L^{2}}
		\lesssim\left\Vert h\right\Vert _{H^{1}}.    
		\end{equation}
		
		Finally, it remains to analyze the last integral in \eqref{locald1}, which we split as
		\begin{align}
		& \int\mathbf{1}_{\{k\geq0\}}e^{ikx}\partial_{x}\mathfrak{Q}(x,k)e^{ik^{2}t}h(k)\,dk = I_{1}(t,x)+I_{2}(t,x),
		\\
		& I_1(t,x) := \int\mathbf{1}_{\{0\leq k\leq \frac{1}{ \sqrt{t}}\}}e^{ikx}\partial_{x}\mathfrak{Q}(x,k)e^{ik^{2}t}h(k)\,dk,
		\\
		& I_2(t,x) := \int\mathbf{1}_{\{ k\geq \frac{1}{\sqrt{t}}\}}e^{ikx}\partial_{x}\mathfrak{Q}(x,k)e^{ik^{2}t}h(k)\,dk.
		\end{align}
		For $I_1$, by the pointwise boundedness of $\jx^\beta \partial_x \mathfrak{Q}(x,k)$ 
		and Sobolev's embedding one has
		\begin{equation}\label{eq:psebound5-1}
		\left\Vert \left\langle x\right\rangle ^{\beta-1} I_1(t,\cdot)\right\Vert _{L_{x}^{2}}
		\lesssim \left\Vert \mathbf{1}_{\{0\leq k\leq \frac{1}{ \sqrt{t}}\}} h(k)\right\Vert_{L^1_k}
		\lesssim \frac{1}{\sqrt{t}}\left\Vert h \right\Vert_{H^1}.
		\end{equation}
		Next, for $I_2$, we perform integration by parts in $k$ as done before, and obtain
		\begin{align*}
		I_2(t,x) & = \frac{1}{2it} \sqrt{t} e^{i\frac{1}{\sqrt{t}}x}e^i \mathfrak{Q}\big(x,\frac{1}{\sqrt{t}}\big)
		h\big(\frac{1}{\sqrt{t}}\big)
		\\
		& - \frac{1}{2it} \int_{k\geq \frac{1}{\sqrt{t}}}\frac{1}{k}e^{ikx}
		\partial_x\mathfrak{Q}(x,k)e^{ik^{2}t}\partial_{k}h(k)\,dk
		\\
		& - \frac{1}{2it} \int_{k\geq \frac{1}{\sqrt{t}}}\frac{1}{k}e^{ikx}\partial_{k}\partial_x\mathfrak{Q}(x,k)e^{ik^{2}t}h(k)\,dk
		\\
		& - \frac{1}{2it} \int_{k\geq\frac{1}{\sqrt{t}}}\frac{1}{k}e^{ikx}ix\partial_x\mathfrak{Q}(x,k)e^{ik^{2}t}h(k)\,dk. 
		\end{align*}
		Then we apply similar arguments to \eqref{eq:psebound1}-
		\eqref{eq:psebound4}. 
		Using $\jx^\beta \partial_x\mathfrak{Q}(x,k)$ and $\jx^{\beta-1} \partial_x\partial_k\mathfrak{Q}(x,k)$ 
		as symbols, one has
		\begin{align}
		\label{eq:psebound5-2}
		& \left\Vert \left\langle x\right\rangle ^{\beta}\int_{k\geq \frac{1}{\sqrt{t}}}\frac{1}{k}
		e^{ikx}\partial_x\mathfrak{Q}(x,k)e^{ik^{2}t}\partial_{k}h(k)\,dk\right\Vert _{L_{x}^{2}}
		\lesssim \sqrt{t}\left\Vert \partial_{k}h\right\Vert _{L^{2}},
		\\
		\label{eq:psebound5-3} 
		& \left\Vert \left\langle x\right\rangle ^{\beta-1}\int_{k\geq\frac{1}{\sqrt{t}}}\frac{1}{k}
		e^{ikx} \partial_k 
		\partial_x\mathfrak{Q}(x,k)e^{ik^{2}t}h(k)\,dk\right\Vert _{L_{x}^{2}}
		\lesssim \sqrt{t}\left\Vert h\right\Vert_{L^{2}},
		\\
		\label{eq:psebound5-4}
		& \left\Vert \jx^{\beta-1}
		\int_{k\geq\frac{1}{\sqrt{t}}}\frac{1}{k}
		e^{ikx}ix\partial_x\mathfrak{Q}(x,k)e^{ik^{2}t}h(k)\,dk\right\Vert_{L_{x}^{2}}
		\lesssim \sqrt{t}\left\Vert h \right\Vert _{L^{2}}.
		\end{align}
		By Sobolev's embedding, we also estimate the boundary term
		\begin{equation}\label{eq:psebound5-5}
		\left\Vert \left\langle x\right\rangle ^{\beta-1} \sqrt{t} e^{i\frac{1}{\sqrt{t}}x}e^i 
		\mathfrak{Q}(x,\frac{1}{\sqrt{t}})h(\frac{1}{\sqrt{t}})\right\Vert_{L^{2}}
		\lesssim \sqrt{t} \left\Vert h \right\Vert_{H^{1}}.    
		\end{equation}
		Putting together \eqref{eq:psebound5-2}
		-\eqref{eq:psebound5-5}, it follows that
		\begin{equation}\label{eq:psebound5-6}
		\left\Vert \left\langle x\right\rangle ^{\beta-1} I_2(t,\cdot)\right\Vert_{L_{x}^{2}}
		\lesssim \frac{1}{\sqrt{t}}\left\Vert h\right\Vert_{H^1}.
		\end{equation}
		The desired result follows from \eqref{eq:psebound5-1}, \eqref{eq:psebound5-6},
		\eqref{eq:psebound1}-
		\eqref{eq:psebound4}.
	\end{proof}

	\smallskip
As an immediate corollary of Lemma \ref{lem:localderivative}, together
		with the pointwise bounds on $m_\pm$ in Lemma \ref{lem:Mestimates} and the 
		PDO bounds in Lemma \ref{lem:m-1}, we obtain the following estimates:
		using the notation \eqref{decompphi}, with $u=e^{-itH} f$ ($u^\# = e^{itk^2} f^\#$), we have  
		\begin{align}\label{dxuM}
		{\big\| \jx^{-1}\partial_x u_{M}(t) \big\|}_{L^2_x} \lesssim \jt^{-1/2} {\| f^\#(t) \|}_{H^1} 
		, \qquad M \in \{S,R\}.
		\end{align}


	\def\eps{\epsilon}
	\def\epss{\eps_1,\eps_2,\eps_3,\eps_4}
	\def\pv{\mathrm{p.v.}}
	\def\kk{k_1,k_2,k_3,k_4}
	
	\bigskip
	\section{The nonlinear spectral distribution for non-generic potentials}\label{secmu}
	In this section we give our main result on the decomposition of the 
	nonlinear spectral distribution
	\begin{equation}\label{musharpdef}
	\mu^{\#}(k,\ell,m,n):=\int\overline{\mathcal{K}^{\#}(x,k)}\mathcal{K}^{\#}(x,\ell)
	\overline{\mathcal{K}^{\#}(x,m)}\mathcal{K}^{\#}(x,n)\,dx.
	\end{equation}

	\medskip
	\subsection{Definitions and notation conventions}
	We introduce some definitions and convenient notation.
	First, we define the set of `good' linear coefficients, which is given by six scalar functions as follows:
	\begin{align}\label{defS}
	\begin{split}
	A := \big\{ [T(\cdot)-T(0)]\mathbf{1}_+(\cdot), \,\, -R_+(-\cdot)\mathbf{1}_-(\cdot), \,\, R_-(\cdot)\mathbf{1}_+(\cdot),
-[T(-\cdot)-T(0)]\mathbf{1}_-(\cdot), \,\, 1, \,\, -1 \big\}.
	\end{split}
	\end{align}
Notice that the set \eqref{defS} contains the coefficients $a^\epsilon_\iota$,
	$\epsilon,\iota \in \{+,-\}$ defined in \eqref{K+}-\eqref{K-},
	and that these are all Lipschitz continuous (hence the denomination `good').
	
	In order to have more compact notation we use the following convention:
	we denote signs $\{+,-\}$ by $\eps_j$ and let
	\begin{align}\label{convsum}
	\sum_\ast = \sum_{\epss \in \{+,-\}} \, .
	\end{align}
	
	Given four coefficients in $A$ we denote a fourfold tensor product with alternate conjugation as follows:
	\begin{align}\label{convprod}
	\prod_\ast 
	a_j(k_j) := \overline{a_1(k_1)} a_2(k_2) \overline{a_3(k_3)} a_4(k_4).
	\end{align}
	We will also use a similar notation for the product of four functions that also depend on the (same) variable $x$:
	\begin{align}\label{convprod'}
	\prod_\ast 
	b_j(x,k_j) := \overline{b_1(x,k_1)} b_2(x,k_2) \overline{b_3(x,k_3)} b_4(x,k_4).
	\end{align}
	See for example \eqref{muR1}.
	
	For ${\bf \eps} = (\epss)$ and ${\bf k} = (\kk)$ we let
	\begin{align}\label{dot}
	{\bf \eps} \cdot {\bf k} = \eps_1k_1 + \eps_2k_2 + \eps_3k_3 + \eps_4k_4
	\end{align}
	be the standard dot product.
	
	We define $\mathcal{C}$ to be the set of functions on $\R^4$ that are a tensor product of $4$ elements in $A$:
	\begin{align}\label{defSet}
	\mathcal{C} := \big\{ f: {\bf k}:=(k_1,k_2,k_3,k_4) \in \R^4 \rightarrow \C, \, \, 
	\mbox{s.t.} \,  f({\bf k}) = \prod_\ast a_j(k_j), \,\, a_j \in A \big\},
	\end{align}
	$\mathcal{C} =  \bar{A}\otimes  A \otimes  \bar{A}\otimes A.$ 
	We will also consider the above set minus the constant functions $1$ and $-1$:
	\begin{align}\label{defSet'}
	\mathcal{C}_L = \mathcal{C} \smallsetminus \{1,-1\}.
	\end{align}
	In particular, any function $f: {\bf k} \in \R^4 \rightarrow \C$ that belongs to $\mathcal{C}_L$
	is such that $f({\bf k}) = 0$ if $k_1\cdot k_2\cdot k_3\cdot k_4 = 0$.

	Given $\chi_+$ and $\chi_-$ as before, by standard computations we have, for $r=1,2,3,4$, 
	\begin{align}\label{chi+-}
	\begin{split}
	\int e^{-i\xi x}\chi_{+}^{r}(x)\,dx = 
	\sqrt{\frac{\pi}{2}}\delta_{0}(\xi)+\pv\frac{\what{\zeta}_{r}(\xi)}{i\xi}+\what{\varpi}_{r}(\xi),
	\\
	\int e^{-i\xi x}\chi_{-}^{r}(x)\,dx =
	\sqrt{\frac{\pi}{2}}\delta_{0}(\xi)-\pv\frac{\what{\zeta}_{r}(\xi)}{i\xi}+\what{\varpi}_{r}(\xi),
	\end{split}
	\end{align}
	where $\zeta_{r}$ is an even $C_{c}^{\infty}$ function with integral 
	$1$ and $\varpi_r$ denotes a (generic) $C_{c}^{\infty}$ function.
	
	\begin{thm}\label{theomu}
		Let $\mu^\#$ be defined as in \eqref{musharpdef}. Then, we can decompose into the sum of a delta function,
		a singular part with improved low frequency behavior (subscript $L$)
		and a regular part (subscript $R$) as follows:
		let ${\bf k}=(k_1,k_2,k_3,k_4) \in \R^4$, we have
		\begin{align}\label{mudecomp}
		(2\pi)^2 \, \mu^\#({\bf k}) = \sqrt{2\pi}\, \delta_0(k_1-k_2+k_3-k_4) + \mu^\#_L({\bf k}) + \mu^\#_R({\bf k}),
		\end{align}
		where
		
		\setlength{\leftmargini}{1.5em}
		\begin{itemize}
			
			\medskip
			\item Denoting ${\bf \eps} = (\epss)$ we have that  $\mu^\#_L({\bf k})$ is a finite linear combination (over $\R$) of terms of the form
			\begin{align}\label{muL}
			\begin{split}
			a_{\epss}({\bf k}) \Big[ \sqrt{\frac{\pi}{2}} \delta({\bf \eps} \cdot {\bf k} ) 
			\pm \pv \frac{\phi_{\epss}({\bf \eps} \cdot {\bf k} )}{i{\bf \eps} \cdot {\bf k} } \Big] &,
			\quad \, \mbox{with} \quad a_{\epss} \in \mathcal{C}_L,
			\end{split}
			\end{align}
			and (even) $\phi_{\epss} \in \mathcal{S}$ with integral $1$ 
			(where $\mathcal{S}$ is the Schwartz class).
			
			\medskip
			\item We have
			\begin{align}\label{muR0}
			\begin{split}
			\mu^\#_R({\bf k}) =  \mu^\#_{R,1}({\bf k}) + \mu^\#_{R,2}({\bf k}),  
			\end{split}
			\end{align}
			where
			\begin{align}\label{muR1}
			\begin{split}
			\mu^\#_{R,1}({\bf k}) = \sum_{(A_1,A_2,A_3,A_4) \in \mathcal{X}_R} \int_{\R} \prod_\ast \mathcal{K}^\#_{A_j}(k_j,x) \, dx,
			\,\quad \mathcal{X}_R := \{ S,R \}^4 \smallsetminus \{(S,S,S,S)\},
			\end{split}
			\end{align}
			and $\mu^\#_{R,2}$ is a linear combination of terms of the form
			\begin{align}\label{muR2}
			\begin{split}
			& a_{\epss}({\bf k}) \, \varphi_{\epss}({\bf \eps} \cdot {\bf k}),
			\\
			& \mbox{with} \quad a_{\epss} \in \mathcal{C},
			\quad \mbox{and} \quad \big|\partial_x^\alpha \whF^{-1}(\varphi_{\epss}) \big| \lesssim \jx^{-\gamma+1},\quad \alpha=0,1.
			\end{split}
			\end{align}
			
		\end{itemize}
		
	\end{thm}

	\medskip
	Here are a few words to explain the statement and the logic behind Theorem \ref{theomu}.
	
	\setlength{\leftmargini}{2em}
	\begin{itemize}
		
		\smallskip
		\item The singular part of the distribution is made of $\delta$ contributions which will be easy to handle,
		as these correspond to the flat $V=0$ case.
		
		\smallskip
		\item The $\pv$ parts in \eqref{muL} are potentially dangerous, but they come with at least one vanishing coefficient,
		since $a_{\epss} \in \mathcal{C}_L$; see \eqref{defSet'}. 
		For this reason, we refer to $\mu^\#_L$ as an ``improved low frequency'' singular distribution.
		
		\smallskip
		\item The regular part $\mu^\#_R$ is split into a local part $\mu^\#_{R,2}$ which is just a linear combination
		of Schwartz functions of ${\bf \eps} \cdot {\bf k}$, and a pseudodifferential part $\mu^\#_{R,1}$
		which involves at least one localized $\mathcal{K}^\#_R(x,k_j)$ function.
		
		
	\end{itemize}
	
	
	%
	%
	%
	%
	%
	%

	\medskip
	\subsection{Proof of Theorem \ref{theomu}} 
	\label{subsec:DecM}
	
	We proceed in a few steps.
	Recall the definitions \eqref{Ksharpdecomp} and \eqref{K0}-\eqref{KR} and Lemma \ref{lem:Kregular}.
	Note that $\mathcal{K}_{\pm}^{\#}(x,k)$ vanish at $k=0$, and are a linear combination of regular exponentials with 
	coefficients in the set $A \smallsetminus \{1,-1\}$.
	
	\medskip
	\noindent
	{\it Step 1: Removing $\mu^\#_{R,1}$ and the delta function}.
	Starting from the definition of $\mu^\#$ we first remove $\mu^\#_{R,1}$ 
	as defined in \eqref{muR1}, and are left with
	\begin{align}\label{prmuS1}
	\begin{split}
	(2\pi)^2 \mu^\#({\bf k}) - \mu^\#_{R,1}({\bf k}) & = \int_{\R} \prod_\ast \mathcal{K}^\#_S(k_j,x) \, dx
	\\
	& = \sum_{(\iota_1,\iota_2,\iota_3,\iota_4) \in \{0,+,-\}^4} \int_{\R} \prod_\ast \chi_{\iota_j}(x)\mathcal{K}^\#_{\iota_j}(k_j,x) \, dx.
	\end{split}
	\end{align}
	To prove the proposition it will suffice to show that the terms in \eqref{prmuS1} are 
	either of the form \eqref{muL} or of the form \eqref{muR2}.
	
	We first look at the leading order term in the last sum in \eqref{prmuS1}, 
	that is the one with $(\iota_1,\iota_2,\iota_3,\iota_4)=(0,0,0,0)$;
	recalling the convention \eqref{convprod} and the convention $\chi_0(x) \equiv 1$, this leading order term
	is given by (recall $T(0)=-1$)
	\begin{align}\label{prmuS0}
	\begin{split}
	\int_{\R} \prod_\ast (-m_+(x,0)) e^{ixk_j}\, dx 
	& = \int_{\R} m_+^4(x,0) e^{ix(-k_1+k_2-k_3+k_4)} \, dx 
	\\
	& = \sqrt{2\pi} \, \delta(k_1-k_2+k_3-k_4) 
	\\
	& + \int_{\R} \big( m_+^4(x,0) -  1\big) e^{ix(-k_1+k_2-k_3+k_4)} \, dx .
	\end{split}
	\end{align}
	The second to last line gives us the delta contribution in \eqref{mudecomp}.
The last line in \eqref{prmuS0} is $\sqrt{2\pi} \whF(g)(k_1-k_2+k_3-k_4)$
		where $g(x) = m_+^4(x,0) - 1$; 
		this latter is regular and decaying according to \eqref{m+0'},
		and therefore this term is accounted for in the sum on the right-hand side of \eqref{muR2}.
	
	Next, we look at the contributions from \eqref{prmuS1} that give rise 
	to the improved low frequency singular distribution $\mu^\#_L$.
	
	\medskip
	\noindent
	{\it Step 2: The singular part $\mu^\#_L$}.
	This part of the distribution arises from the contributions to \eqref{prmuS1} 
	that have no terms with both $+$ and $-$ in the sum,
	that is, the term
	\begin{align}
	\label{prmu+a}
	& \sum_{(\iota_1,\iota_2,\iota_3,\iota_4) \in \{0,+\}^4 \smallsetminus \{0\}^4} 
	\int_{\R} \prod_\ast \chi_{\iota_j}(x)\mathcal{K}^\#_{\iota_j}(k_j,x) \, dx
	\\
	\label{prmu+b}
	& + \sum_{(\iota_1,\iota_2,\iota_3,\iota_4) \in \{0,-\}^4 \smallsetminus \{0\}^4} 
	\int_{\R} \prod_\ast \chi_{\iota_j}(x)\mathcal{K}^\#_{\iota_j}(k_j,x) \, dx.
	\end{align}
	
	We look at the first sum above \eqref{prmu+a}. 
	By symmetry in the $k_j$ variables (up to irrelevant conjugations),
	it suffices to analyze four types of terms in the sum in \eqref{prmu+a}, 
	depending on how many indexes $\iota_j = 0$ appear; these terms are (recall the convention $\chi_0\equiv1$)
	\begin{subequations}\label{prmu+1}
		\begin{align}
		\label{prmu+11}
		& \int_{\R} \chi_+(x)  \overline{\mathcal{K}^\#_0(k_1,x)} \mathcal{K}^\#_0(k_2,x) 
		\overline{\mathcal{K}^\#_0(k_3,x)} \mathcal{K}^\#_+(k_4,x) \, dx,
		\\
		\label{prmu+12}
		& \int_{\R} \chi_+^2(x)  \overline{\mathcal{K}^\#_0(k_1,x)} \mathcal{K}^\#_0(k_2,x) 
		\overline{\mathcal{K}^\#_+(k_3,x)} \mathcal{K}^\#_+(k_4,x) \, dx,
		\\
		\label{prmu+13}
		& \int_{\R} \chi_+^3(x) \overline{\mathcal{K}^\#_0(k_1,x)} \mathcal{K}^\#_+(k_2,x) 
		\overline{\mathcal{K}^\#_+(k_3,x)} \mathcal{K}^\#_+(k_4,x) \, dx,
		\\
		\label{prmu+14}
		& \int_{\R} \chi_+^4(x) \prod_\ast \mathcal{K}^\#_{+}(k_j,x) \, dx.
		\end{align}
	\end{subequations}
	
	Recalling the definitions \eqref{K0} and \eqref{K+} we have 
	\begin{align}
	\nonumber
	\eqref{prmu+11} & = \sum_{\eps_4 \in\{+,-\}}\int_{\R}\big( -m_+^3(x,0) \big) e^{ix(-k_1+k_2-k_3)} 
	\chi_+(x)  a_+^{\eps_4}(k_4) e^{\eps_4 i xk_4} \, dx
	\\
	\label{prmu+11a}
	& = -\sum_{\eps_4\in\{+,-\}}  a_+^{\eps_4}(k_4) \int_{\R} \chi_+(x) e^{ix(-k_1+k_2-k_3 + \eps_4 k_4)} \, dx
	\\
	\label{prmu+11b}
	& - \sum_{\eps_4\in\{+,-\}} a_+^{\eps_4}(k) \int_{\R} \chi_+(x)  
	\big( m_+^3(x,0) - 1\big) e^{ix(-k_1+k_2-k_3 + \eps_4 k_4)} \, dx.
	\end{align}
	Using the first formula in \eqref{chi+-} we see that \eqref{prmu+11a} is of the form \eqref{muL},
	since the coefficient $a^\epsilon_\iota(k_4)$ belongs to $\mathcal{C}_L$
	(recall the definitions \eqref{defSet}-\eqref{defSet'} and \eqref{defS}).
	For the terms in \eqref{prmu+11b}, denoting 
	$$\rho(K) := \int_{\R} \chi_+(x) \big( m_+^3(x,0) - 1\big) e^{-ixK} \, dx, $$
	and using \eqref{m+0'}, we have that
	\begin{align}\label{accmuR2}
	\big|\partial_x^\alpha  \whF^{-1}_{K\mapsto x} \rho \big| 
	= \sqrt{2\pi} \big| \partial_x^\alpha \big[ \chi_+(x) \big( m_+^3(x,0) - 1\big) \big] \big|
	\lesssim \jx^{-\gamma+1},
	\end{align}
	so that \eqref{prmu+11b} are acceptable regular remainders that we can include in \eqref{muR2}.
	
	For the term in \eqref{prmu+12} we can write similarly
	\begin{align}
	\nonumber
	\eqref{prmu+12} 
	& = \sum_{ \eps_3,\eps_4 \in\{+,-\} } \int_{\R} m_+^2(x,0) e^{ix(-k_1+k_2)} 
	\, \chi_+(x)  \overline{a_+^{\eps_3}(k_3) e^{\eps_3 i k_3x}} \, \chi_+(x)  a_+^{\eps_4}(k_4) e^{\eps_4 i k_4x}\, dx
	\\
	\label{prmu+12a}
	& = \sum_{ \eps_3,\eps_4 \in\{+,-\} } \overline{a_+^{\eps_3}(k_3)}  a_+^{\eps_4}(k_4) \int_{\R} \chi_+^2(x) 
	e^{ix(-k_1+k_2-\eps_3k_3 + \eps_4 k_4)} \, dx
	\\
	\label{prmu+12b}
	& = \sum_{ \eps_3,\eps_4 \in\{+,-\} }  \overline{a_+^{\eps_3}(k_3)}  a_+^{\eps_4}(k_4) \int_{\R} \chi_+^2(x) \big(m_+^2(x,0) - 1\big) 
	e^{ix(-k_1+k_2-\eps_3k_3 + \eps_4 k_4)} \, dx .
	\end{align}
	Using \eqref{chi+-} we see that \eqref{prmu+12a} is of the form \eqref{muL}.
	Using \eqref{m+0'} we see, similarly to \eqref{accmuR2}, that the term \eqref{prmu+12b} is of the form \eqref{muR2}. 
	The term \eqref{prmu+13} can be analyzed in the same exact way,
	using \eqref{m+0}, and contributes to \eqref{muL} and the linear combination of terms as in \eqref{muR2}.
	Using again \eqref{K+} and \eqref{chi+-} we see that the term \eqref{prmu+14} 
	also contributes to \eqref{muL} and \eqref{muR2}.
	This completes the analysis of \eqref{prmu+a}
	
	
	The sum \eqref{prmu+b} can be analyzed in a similar way, 
	using \eqref{K0}, \eqref{K-}, the second formula in \eqref{chi+-},
	and \eqref{m+0'} with $\epsilon = -$.
	

	\medskip
	\noindent
	{\it Step 3: The regular part}.
	We are left with the terms in \eqref{prmuS1} where the sum is taken over quadruples 
	$(\iota_1,\iota_2,\iota_3,\iota_4)$
	that contain at least one $+$ and one $-$ sign, which we denote by 
	\begin{align}
	\mathcal{I} := \big\{ (\iota_1,\iota_2,\iota_3,\iota_4) \in \{+,-,0\}^4 \, \mbox{s.t.} \, \exists \, a,b \in \{1,2,3,4\} \,\, \mbox{with} 
	\,\, \iota_a = +, \iota_b = - \big\}.
	\end{align}
	Up to permuting variables (and conjugating), we can reduce matters to the terms where the $+$ and $-$ indexes correspond to the first two 
	generalized eigenfunctions, 
	that is, the sum
	\begin{align}\label{prmuR0}
	\begin{split}
	& \sum_{(\iota_3,\iota_4) \in \{+,-,0\} } \int_{\R} \overline{\chi_{+}(x)\mathcal{K}^\#_{+}(k_1,x)}  
	\, \chi_{-}(x)\mathcal{K}^\#_{-}(k_2,x) 
	\, \overline{\chi_{\iota_3}(x)\mathcal{K}^\#_{\iota_3}(k_3,x)}  \, \chi_{\iota_4}(x)\mathcal{K}^\#_{\iota_4}(k_4,x) \, dx.
	\end{split}
	\end{align}
	Using \eqref{K+} and \eqref{K-}  we can write
	\begin{align}\label{prmuR1}
	\begin{split}
	\eqref{prmuR0} & = \sum_{ \substack{\iota_3,\iota_4 \in \{+,-,0\} \\ \eps_1,\eps_2 \in \{+,-\} } }   \overline{a_+^{\eps_1}(k_1)}  a_-^{\eps_2}(k_2)
	\\ 
	& \times \int_{\R} (\chi_+\chi_-)(x) \, 
	e^{ix(-\eps_1k_1+\eps_2k_2)} 
	\overline{\chi_{\iota_3}(x)\mathcal{K}^\#_{\iota_3}(k_3,x)}  
	\, \chi_{\iota_4}(x)\mathcal{K}^\#_{\iota_4}(k_4,x) \, dx
	\end{split}
	\end{align}
	and see that since $\chi_+\chi_-$ is compactly supported,
	the expression above is of the form \eqref{muR2} in view of the definitions 
	\eqref{K0}, \eqref{K+} and \eqref{K-}, and the definition \eqref{defSet}. 
	
	This concludes the proof of the theorem. $\hfill \Box$



\bigskip
\section{Nonlinear estimates I: set-up and the singular part}\label{sec:Cubic}

In this and the next section we perform the main weighted estimates on solutions of \eqref{NLSV},
and show the a priori bound \eqref{aprweiconc}, which then gives Proposition \ref{propaprwei}.
As already pointed out, it suffices to concentrate on the case of an odd resonance,
in which case we denote $\mathcal{F} = \mathcal{F}^\sharp$.
Taking the modified/sharp distorted Fourier transform and using Duhamel's formula, 
in terms of the profile we can write
\begin{align}
\begin{split}
& f^{\#}(t,k) = f^{\#}(0,k)
\\ 
& \pm i\int_{0}^{t}
\iiint e^{is(-k^2+\ell^2-m^2+n^2)}f^{\#}(s,\ell)\overline{f^{\#}(s,m)}f^{\#}(s,n)\mu^{\#}(k,\ell,m,n)\,dndmd\ell ds
\label{eq:profile}
\end{split}
\end{align}
where, we recall,
\begin{equation}\label{eq:nonmea}
\mu^{\#}(k,\ell,m,n) := \int\overline{\mathcal{K}^{\#}(x,k)}\mathcal{K}^{\#}(x,\ell)\overline{\mathcal{K}^{\#}(x,m)}\mathcal{K}^{\#}(x,n)\,dx
\end{equation}
is the the modified nonlinear spectral distribution defined in \eqref{musharpdef} 
and analyzed in Theorem \ref{theomu}.
	
\medskip
Our main purpose in this and the next section is to 
show the following:

\begin{prop}\label{pro:weightmain}
For $0\leq t\leq T$, one has that, for some $C>0$,
\begin{align}\label{weightmainconc}
{\big\| \partial_{k}f^{\#}(t) \big\|}_{L_{k}^{2}} \leq {\big\| \partial_{k}f^{\#}(0) \big\|}_{L_k^2} 
  + C \jt^{\alpha} \varepsilon^3 
\end{align}
\end{prop}

\medskip
\subsection{Decomposition of the nonlinearity}\label{subsec:DecMpointwise}
	According to the decomposition  \eqref{mudecomp} of $\mu^\#$ from Theorem \ref{theomu}, and 
	Duhamel's formula \eqref{eq:profile}, we can write
	\begin{align}\label{eq:expandf}
	\begin{split}
	f^{\#}(t,k) & = f^{\#}(0,k) \, \mp \, 
	i\int_0^{t} \big(\mathcal{N}_{0} + \mathcal{N}_{L}
	+ \mathcal{N}_{R,1} + \mathcal{N}_{R,2} \big)\,ds,
	\end{split}
	\end{align}
	where
	\begin{align}\label{Nast}
	\begin{split}
	\mathcal{N}_{\ast}(s,k)=\frac{1}{(2\pi)^{2}} \iiint e^{is(-k^2+\ell^2-m^2+n^2)}
	f^{\#}(s,\ell)\overline{f^{\#}(s,m)}
	f^{\#}(s,n) &
	\\ \times \mu^{\#}_{\ast}(k,\ell,m,n)\,dndmd\ell, & \quad \ast\in\{0; L; R,\!1; R,\!2\}
	\end{split}
	\end{align}
	according to the definitions in Theorem \ref{theomu} and the notation
	\begin{align}	
	\mu^{\#}_{0}(k,\ell,m,n) = \sqrt{2\pi}\, \delta_0(k-\ell+m-n).
	\end{align}	
	To prove the main Proposition \ref{pro:weightmain} it then suffices to show the following
	estimates:
		\begin{align}
		\label{weightmain0}
		& {\Big\| \int_0^t \partial_k \mathcal{N}_0(s) \, ds \Big\|}_{L^2_k} \lesssim 
		\jt^{\alpha} \varepsilon^3, 
		\\
		\label{weightmain1}
		& {\Big\| \int_0^t \partial_k \mathcal{N}_L(s) \, ds  \Big\|}_{L^2_k} \lesssim \jt^{\alpha} \varepsilon^3, 
		\\
		\label{weightmain2}
		& {\Big\| \int_0^t \partial_k \mathcal{N}_{R,1}(s) \, ds  \Big\|}_{L^2_k} 
		+ {\Big\| \int_0^t \partial_k \mathcal{N}_{R,2}(s) \, ds \Big\|}_{L^2_k} 
		\lesssim \jt^{\alpha} \varepsilon^3. 
		\end{align}

The proof of \eqref{weightmain1} is given in Subsection \ref{ssecsing} below,
and the proof of \eqref{weightmain2} in Section \ref{sec:estimateregular}.

\medskip
\subsection{Estimate of $\mathcal{N}_0$}
We skip the details of the estimate for this term since it is straightforward,
and essentially a special case of \eqref{weightmain1}.
In particular, one can apply Lemma \ref{lem:inverseFD},
and estimate ${\|\partial_k \mathcal{N}_0\|}_{L^2}$ 
by the same terms in \eqref{N+3}-\eqref{N+4} (taking $a_{\eps_1}=1$);
since $yb(y) = 0$ for $b = \delta$, one immediately obtains a bound by the right-hand side of \eqref{N+5},
which suffices.

	\medskip
	\subsection{Estimate of $\mathcal{N}_L$}
	\label{ssecsing}
	Recall that from \eqref{muL} a generic term in $\mu^{\#}_{L}$ is of the form
	\begin{align}\label{muLg}
	\begin{split}
	\mu^{\#}_{L,g}({\bf k}) = a_{\epss}({\bf k}) \Big[ \sqrt{\frac{\pi}{2}} \delta({\bf \eps} \cdot {\bf k} ) 
	\pm \pv \frac{\phi_{\epss}({\bf \eps} \cdot {\bf k} )}{i{\bf \eps} \cdot {\bf k} } \Big] &,
	\quad \, \mbox{with} \quad a_{\epss} \in \mathcal{C}_L.
	\end{split}
	\end{align}
	To estimate $\mathcal{N}_L$, it suffices to estimate
	a generic term $\mathcal{N}_{L,g}$, where we are again using the notation \eqref{Nast}.

	First of all, we recall a lemma for the commutation of $\partial_k$
	and trilinear expressions like the ones we need to estimate:
	
	\begin{lem}\label{lem:algetri}
		For $f_{j}\in\mathcal{S}$, define the trilinear form 
		\begin{align}\label{eq:tril}
		\begin{split}
		\mathcal{T}_{\mathrm{b},\e_1,\e_2,\e_3,\e_4}(f_1,f_2,f_3)(k)=
		\iiint e^{it(-k^2+\ell^2-m^2+n^2)}f_{1}(\ell) \overline{f_{2}(m)} f_{3}(n)
		\\ \times \mathrm{b}(\e_1k+\e_2\ell+\e_3m+\e_4n)\,d\ell dmdn 
		\end{split} 
		\end{align}
		where $\epsilon_{j}\in\left\{ +,- \right\}$, and $\mathrm{b}$ is a distribution. 
		Then one has
		\begin{align}\label{eq:kdiffT}
		\begin{split}
		& \e_1 \partial_{k}\mathcal{T}_{\mathrm{b},\e_1,\e_2,\e_3,\e_4}(f_1,f_2,f_3)(k) 
		\\
		& \qquad = -\e_2\mathcal{T}_{\mathrm{b},\e_1,\e_2,\e_3,\e_4}
		\left(\partial_{\ell}f_{1},f_{2},f_{3}\right)(k)
		+\e_3\mathcal{T}_{\mathrm{b},\e_1,\e_2,\e_3,\e_4}\left(f_{1},\partial_{m}f_{2},f_{3}\right)(k)
		\\
		& \qquad -\e_4\mathcal{T}_{\mathrm{b},\e_1,\e_2,\e_3,\e_4}\left(f_{1},f_{2},\partial_{n}f_{3}\right)(k)
		-2it \, \mathcal{T}_{\mathrm{yb},\e_1,\e_2,\e_3,\e_4}(f_1,f_2,f_3)(k).
		\end{split}
		\end{align}
	\end{lem}
	
	\begin{proof}
		See Chen-Pusateri \cite[Lemma  4.3]{NLSV}.
	\end{proof}
	
	\smallskip
	We will also use the following:
	
	\begin{lem}\label{lem:inverseFD}
		Given $f_{j}\in\mathcal{S}$ and $\epsilon_{j}\in \{ +,-\}$, $j=1,2,3$, then
		\begin{align}\label{eq:inverseFD}
		\begin{split}
		& \widehat{\mathcal{F}}^{-1} \big[e^{itk^2} {\mathcal{T}}_{\mathrm{b},\e_1,\e_2,\e_3,\e_4}(f_1,f_2,f_3)\big]
		= 2\pi\e_1 u_{1}(t,-\e_2x)\overline{u_{2}(t,\e_3x)}u_{3}(t,-\e_4x)\widehat{\mathcal{F}}^{-1}\left[\mathrm{b}\right],
		\\
		& \qquad \mbox{where} \qquad u_{j}:=e^{-it\partial_{xx}}\widehat{\mathcal{F}}^{-1}(f_j).
		\end{split}
		\end{align}
	\end{lem}

	\begin{proof}
		The proof is an explicit computation which follows from the convolution identities for
		the flat Fourier transform, and is left to the reader.
	\end{proof}

	\smallskip
	\begin{proof}[Proof of \eqref{weightmain1}]
		According to our notation \eqref{muLg}, recalling also \eqref{defSet}-\eqref{defSet'},
		we write the generic nonlinear term $\mathcal{N}_{L,g}$ as
		\begin{align}\label{N+0}
		\begin{split}
		(2\pi)^{2} \mathcal{N}_{L,g}(s,k) & = \overline{a_{\e_1}(k)} \,
		\iiint e^{is(-k^2+\ell^2-m^2+n^2)} \, a_{\e_2}(\ell)f^{\#}(s,\ell) 
		\\ 
		& \times \, \overline{a_{\e_3}(m)f^{\#}(s,m)} \,
		a_{\e_4}(n)f^{\#}(s,n)
		\, \mathrm{b}(\epsilon_1k + \epsilon_2\ell + \epsilon_3m + \epsilon_4n) 
		\,dndmd\ell,
		\end{split}
		\end{align}
		with
		\begin{equation}\label{eq:bchoice}
		\mathrm{b}(y) 
		:= \sqrt{\frac{\pi}{2}} \delta(y) 
		\pm \pv \frac{\phi_{\epss}(y) 
		}{i y
		}.
		\end{equation}
		We then let
		\begin{align}\label{N+fj}
		f_j(s,\cdot) := a_{\e_{j+1}}(\cdot) f^{\#}(s, 
			\cdot) \qquad j=1,2,3,
		\end{align}
		and using the notation \eqref{eq:tril} we write
		\begin{align}\label{N+1}
		\begin{split}
		(2\pi)^{2} \mathcal{N}_{L,g}(s,k)  = \overline{a_{\e_1}(k)} \,
		\mathcal{T}_{\mathrm{b},\e_1,\e_2,\e_3,\e_4}(f_1,f_2,f_3)(
		k).
		\end{split}
		\end{align}
		
		Using Lemma \ref{estiTR} and the fact that $a_{\e_1}$ is a Lipschitz  function, we see that
		\begin{align}\label{N+2}
		\begin{split}
		\left\Vert \int_0^t \partial_k  \mathcal{N}_{L,g}(s,k) \,ds \right\Vert_{L^2_k}
		& \lesssim \left\Vert \int_0^t \mathcal{T}_{\mathrm{b},\e_1,\e_2,\e_3,\e_4}(f_1,f_2,f_3)(k)\,ds\right\Vert_{L^2_k}
		\\
		& + \left\Vert \int_0^t \overline{a_{\e_1}(k)}  \partial_k
		\mathcal\mathcal{T}_{\mathrm{b},\e_1,\e_2,\e_3,\e_4}(f_1,f_2,f_3)(k) \,ds\right\Vert_{L^2_k}.
		\end{split}
		\end{align}
		The first term in the right-hand side above is easily estimated using Lemma \ref{lem:inverseFD}: 
		we let
		\begin{align}\label{N+uj}
		u_{j} := e^{-it\partial_{xx}}\widehat{\mathcal{F}}^{-1}(f_j)
		= e^{-it\partial_{xx}} \widehat{\mathcal{F}}^{-1}
		\big( a_{\e_{j+1}}f^{\#}
		\big), \qquad j=1,2,3,
		\end{align}
		use the  standard pointwise decay estimate for the free flow 
		(see for example \cite[Lemma 1.7]{KP} and \cite[Lemma 2.2]{HN}) 
		and the bootstrap assumptions to deduce
		\begin{align}\label{linestboot}
		\begin{split}
		{\| u_j(s) \|}_{L^\infty_x} \lesssim \frac{1}{\sqrt{\langle s \rangle}} 
		\big( {\|a_{\epsilon_j+1} f^\# \|}_{L^\infty} + {\| \jk a_{\epsilon_j+1} f^\# \|}_{L^2}  \big) 
		+ \frac{1}{\langle s \rangle^{3/4}} {\| \partial_k  a_{\e_{j+1}}f^{\#} \|}_{L^2}
		\lesssim \frac{\varepsilon}{\sqrt{\langle s \rangle}},
		\end{split}
		\end{align}
		and, using $|\whF^{-1}[b]| \lesssim 1$, we estimate
		\begin{align*}
		{\big\| \mathcal{T}_{\mathrm{b},\e_1,\e_2,\e_3,\e_4}(f_1,f_2,f_3) \big\|}_{L^2} 
		& \lesssim 
		{\| u_1(s,-\e_2 \cdot) \, u_2(s,-\e_3\cdot)\, u_3(s,-\e_4\cdot) \,\whF^{-1}\left[\mathrm{b}\right] \|}_{L^2}
		\\ 
		& \lesssim {\| u_1(s) \|}_{L^2} {\| u_2(s) \|}_{L^\infty} {\| u_3(s) \|}_{L^\infty}
		\\
		& \lesssim \left\langle s\right\rangle^{-1} \varepsilon^3. 
		\end{align*}
		Integrating over $[0,t]$  
		we obtain a bound consistent with the desired \eqref{weightmainconc}.
		
		For the second term on the right-hand side of \eqref{N+2} we first use Lemma \ref{lem:algetri} to obtain
		\begin{align}
		\nonumber 
		& \left\Vert \int_0^t \overline{a_{\e_1}(k)} 
		\partial_k\mathcal{T}_{\mathrm{b},\e_1,\e_2,\e_3,\e_4}(f_1,f_2,f_3)\,ds\right\Vert_{L^2}
		\\ 
		\label{N+3} 
		& \lesssim \left\Vert\int_0^t \mathcal{T}_{\mathrm{b},\e_1,\e_2,\e_3,\e_4}(\partial_kf_1,f_2,f_3)\,ds\right\Vert_{L^2}
		+ \left\Vert\int_0^t \mathcal{T}_{\mathrm{b},\e_1,\e_2,\e_3,\e_4}(f_1,\partial_kf_2,f_3)\,ds\right\Vert_{L^2}
		\\ 
		\label{N+4}
		& + \left\Vert\int_0^t \mathcal{T}_{\mathrm{b},\e_1,\e_2,\e_3,\e_4}(f_1,f_2,\partial_kf_3)\,ds\right\Vert_{L^2}
		+ \left\Vert\int_0^t \overline{a_{\e_1}(k)} \, s \, 
		\mathcal{T}_{\mathrm{yb},\e_1,\e_2,\e_3,\e_4}(f_1,f_2,f_3)\,ds\right\Vert_{L^2}.
		\end{align}
		The first three terms in the right-hand side above are similar and can be estimated using 
		again Lemma \ref{lem:inverseFD},
		\eqref{N+fj} and \eqref{N+uj}, and \eqref{linestboot}:
		\begin{align}\label{N+5}
		\begin{split}
		{\|\mathcal{T}_{\mathrm{b},\e_1,\e_2,\e_3,\e_4}(\partial_kf_1,f_2,f_3)\|}_{L^2}
		\lesssim {\| \partial_kf_1(s) \|}_{L^2} {\| u_2(s) \|}_{L^\infty} {\| u_3(s) \|}_{L^\infty}
		\lesssim  \varepsilon \langle s \rangle^\alpha 
		\Big( \frac{\varepsilon }{\sqrt{\left\langle s\right\rangle}} 
		\Big)^2;
		\end{split}
		\end{align}
		upon time integration this bound is again consistent with \eqref{weightmain1}.  
		Note that in the estimate of $\partial_k f_1$, 
		we used again the Lipschitz continuity of $a_{\e_2}$.
		
		To estimate last term in \eqref{N+4} we first notice that, by \eqref{eq:bchoice}, one has
		\begin{align}\label{ybphi}
		y\mathrm{b}(y) = \mp i\phi_{\epss}(y) 
		\end{align}
		which is a Schwartz function; however, for later purposes, we 
		will only make use of a finite amount of decay for its transform.
		Moreover, since $a_{\epss} \in \mathcal{C}_L$
		then at least one of $a_{\e_j}$, $j=1,2,3,4$, must vanish at $0$. 
		We then consider two cases:
		
		\smallskip
		\noindent
		(1) $a_{\e_1}(0)=0$; 
		
		\smallskip
		\noindent
		(2) $a_{\e_j}(0)=0$ for some $j=2,3$ or $4$.
		
		\smallskip
		For the first case, we use the vanishing property of $a_{\e_1}$, 
		and the smoothing estimate \eqref{eq:generalimhomsmoothing} with the choices
		\begin{align*}
		Q(y,k)=e^{iky},  \quad \beta=0, 
		\quad \phi(k) =a_{\e_1}(k), \quad F(s,y) = s \widehat{\mathcal{F}}^{-1} 
		\big[ e^{isk^2} {\mathcal{T}}_{y\mathrm{b},\e_1,\e_2,\e_3,\e_4}(f_1,f_2,f_3)\big], 
		\end{align*}
		to obtain
		\begin{align*}
		\left\Vert \int_0^t \overline{a_{\e_1}(k)} s \, 
		\mathcal{T}_{\mathrm{yb},\e_1,\e_2,\e_3,\e_4}(f_1,f_2,f_3)\,ds\right\Vert_{L^2_k} 
		& \lesssim 
		\left \Vert \int_0^t e^{-isk^2} \overline{\phi(k)} \, \Big( \int_\R Q(y,k) F(s,y) \, dy \Big) \,ds\right\Vert_{L^2_k}
		\\
		& \lesssim {\big\| F(s,y) \big\|}_{L^1_x L^2_s[0,t]}.
		\end{align*}
		Then, using \eqref{eq:inverseFD}, provided that ${\| \mathcal{F}^{-1}[\mathrm{yb}] \|}_{L^1} < \infty$,
		and using again the decay from \eqref{linestboot},  
		we can bound
		\begin{align*}
		{\big\| F(s,y) \big\|}_{L^1_x L^2_s[0,t]} \lesssim \left\Vert  s u_{1}(s,-\e_2x)\overline{u_{2}(s,-\e_3x)}u_{3}(s,-\e_4x)\mathcal{F}^{-1}\left[\mathrm{yb}\right]
		\right\Vert_{L_{x}^{1}L_{s}^{2}[0,t]}
		\\
		\lesssim {\big\| \langle s \rangle {\| u_1(s) \|}_{L^\infty_x} 
			{\| u_2(s) \|}_{L^\infty_x} {\| u_3(s) \|}_{L^\infty_x} \big\|}_{L^2_s[0,t]}
		\\
		\lesssim \varepsilon^3
		\left\Vert \left\langle s\right\rangle \left\langle s\right\rangle^{-3/2}
		\right\Vert_{L_{s}^{2}[0,t]}
		\lesssim \varepsilon^{3}\sqrt{\log \langle t \rangle}.
		\end{align*}

		For case (2), without loss of generality we may assume $a_{\e_2}(0)=0$. 
		In this case, we first use \eqref{eq:inverseFD} to see that
		\begin{align*}
		{\|\overline{a_{\e_1}(k)}  s \mathcal{T}_{y\mathrm{b},\e_1,\e_2,\e_3,\e_4}(f_1,f_2,f_3)\|}_{L^2} 
		& \lesssim 
		{\|s u_1(s,-\e_2) \, u_2(s,-\e_3)\, u_3(s,-\e_4) \, \mathcal{F}^{-1}[\mathrm{y b}] \|}_{L^2}
		\\ 
		& \lesssim {\| \left\langle x\right\rangle^{-1} s \, u_1(s) \|}_{L^\infty} 
		{\| u_2(s) \|}_{L^\infty} {\| u_3(s) \|}_{L^\infty},
		\end{align*}
		provided  ${\| \langle x \rangle \mathcal{F}^{-1}[y\mathrm{b}] \|}_{L^2} < \infty$.
		Then, we use the local improved decay given by Lemma \ref{lem:lowlocaldecay} with the choices
		\begin{align*}
		Q(x,\ell)=e^{i\ell x}, \quad \beta=-1, \quad \phi(\ell)=a_{\e_2}(\ell), \quad h=f^{\#}(s),
		\end{align*}
		to see that
		\begin{align}\label{loclinestboot}
		{\big\| \left\langle x\right\rangle^{-1} u_1(s) \big\|}_{L^\infty} 
		\lesssim \langle s \rangle^{-1} {\| f^\# \|}_{H^1}.
		\end{align}
		Plugging this estimate in the previous bound, using the a priori bounds and \eqref{linestboot}, we get
		\begin{align*}
		{\|\overline{a_{\e_1}(k)} \mathcal{T}_{y\mathrm{b},\e_1,\e_2,\e_3,\e_4}(f_1,f_2,f_3)\|}_{L^2} 
		& \lesssim \left\langle s\right\rangle^{-1+\alpha} \varepsilon^3,
		\end{align*}
		which, integrating in $s$ gives the desired result and concludes the proof of \eqref{weightmain1}.
\end{proof}

		\begin{rem}\label{remybphi}
			Note that in order to carry out the above argument it sufficed to use
			$|\mathcal{F}^{-1}[\phi_{\epss}] | \lesssim \langle x \rangle^{-(3/2+)}$;
			see \eqref{ybphi}.
			This observation will be helpful in the following section. 
		\end{rem}

	
	\medskip
	\section{Nonlinear estimates II: the regular part}\label{sec:estimateregular}
	In this section we conclude the proof of weighted estimates by proving \eqref{weightmain2}.
	For convenience let us recall here the definition \eqref{muR0}-\eqref{muR2}:
	we have $\mu^\#_R({\bf k}) =  \mu^\#_{R,1}({\bf k}) + \mu^\#_{R,2}({\bf k})$,
	with
	\begin{align}\label{muR1'}
	\begin{split}
	\mu^{\#}_{R,1}(k,\ell,m,n):=\sum_{(A,B,C,D)\in\mathcal{X}_{R}}
	\int\overline{\mathcal{K}^{\#}_{A}(x,k)}\mathcal{K}^{\#}_{B}(x,\ell)
	\overline{\mathcal{K}^{\#}_{C}(x,m)}\mathcal{K}^{\#}_{D}(x,n)\,dx,
	\\
	\mathcal{X}_{R} = := \{ S,R \}^4 \smallsetminus \{(S,S,S,S)\},
	\end{split}
	\end{align}
	and $\mu^{\#}_{R,2}(k,\ell,m,n)$ given by a linear combination of terms
	of the form (we omit the dependence on $\eps_1,\eps_2,\dots$ in the notation)
	\begin{align}
	\label{muR2'}
	\begin{split} 
	\mu_{R,2}^{\#,(0)}(k,m,n,\ell) := \overline{a_{\e_1}(k)} a_{\e_2}(m) \overline{a_{\e_3}(n)} a_{\e_4}(\ell)
	\int\varPsi(x)e^{i(-\e_1 k+\e_2 m-\e_3 n+\e_4\ell)x}\,dx,
	\\
	|\partial_x^\alpha \varPsi(x)| \lesssim \jx^{-\gamma+1}, \quad \alpha =0,1.
	\end{split}
	\end{align}  
	Note that by their definitions the regular parts of the nonlinearity 
	$\mathcal{N}_{R,1}$ and $\mathcal{N}_{R,2}$,
	see the notation \eqref{Nast} 
	have strong localization properties.
	
	We now proceed to prove the two bounds in \eqref{weightmain2} respectively in
	Subsections \ref{ssecNR1} and \ref{ssecNR2}.

	\medskip
	\subsection{Estimate of $\mathcal{N}_{R,1}$}\label{ssecNR1}
	Each element of the sum defining $\mu_{R,1}$ in \eqref{muR1'} 
	has at least one of the indexes $(A,B,C,D)$ equal to $R$.
	It then suffices to consider two situations: $A=R$ and $D=R$; all the others terms can be treated identically. 
	We then define
	\begin{align}
	\label{eq:mur11}
	\mu^{\#,(1)}_{R,1} & := \int\overline{\mathcal{K}}^{\#}_{R}(x,k)\mathcal{K}^{\#}_{M_{1}}(x,\ell)
	\	\overline{\mathcal{K}^{\#}_{M_{2}}(x,n)}\mathcal{K}^{\#}_{M_{3}}(x,m)\,dx,
	\\
	\label{eq:mur12}
	\mu_{R,1}^{\#,(2)} & := \int\overline{\mathcal{K}}^{\#}_{S}(x,k)\mathcal{K}^{\#}_{M_{1}}(x,\ell)
	\overline{\mathcal{K}^{\#}_{M_{2}}(x,n)}\mathcal{K}^{\#}_{R}(x,m)\,dx;
	\qquad M_{i}\in\{S,R\}.
	\end{align}
	In order to estimate the $L_{k}^{2}$ norm of $\partial_{k}\mathcal{N}_{R,1}$ as in \eqref{weightmain2}
	it suffices to prove
	\begin{align}\label{weightmain2A}
	& {\Big\| \int_0^t A_1(s,\cdot) \, ds  \Big\|}_{L^2_k}
	+ {\Big\| \int_0^t A_2(s,\cdot) \, ds  \Big\|}_{L^2_k} \lesssim \jt^{\alpha} \varepsilon^3, 
	\end{align}
	with 
	\begin{align}\label{NR11}
	A_1(k) := & \int_{0}^{t}  iske^{-isk^2}\iiint u^{\#}(\ell)\overline{u^{\#}}(n)u^{\#}(m)  
	\,\mu^{\#,(1)}_{R,1}\left(k,\ell,n,m\right) 
	\,d\ell dmdn\,ds,
	\\
	\label{NR11'}
	A_2(k) := & \int_{0}^{t}  e^{-isk^2} \iiint u^{\#}(\ell)\overline{u^{\#}}(n)u^{\#}(m) 
	\, \partial_{k}\mu^{\#,(1)}_{R,1}\left(k,\ell,n,m\right)
	\,d\ell dmdn\,ds,
	\end{align}
	recall that $u^\#(t,k) = e^{itk^2} f^\#(t,k)$,
	and 
	\begin{align}\label{weightmain2B}
	& {\Big\| \int_0^t B_1(s,\cdot) \, ds  \Big\|}_{L^2_k}
	+ {\Big\| \int_0^t B_2(s,\cdot) \, ds  \Big\|}_{L^2_k} \lesssim \jt^{\alpha} \varepsilon^3, 
	\end{align}
	with
	\begin{align}\label{NR12}
	B_1(k) & := \int_{0}^{t}  iske^{-isk^2}\iiint u^{\#}(\ell)\overline{u^{\#}}(n)u^{\#}(m)  
	\, \mu_{R,1}^{\#,(2)}\left(k,\ell,n,m\right)\,d\ell dmdn\,ds,
	\\
	\label{NR12'}
	B_2(k) & := \int_{0}^{t}  e^{-isk^2} \iiint u^{\#}(\ell)\overline{u^{\#}}(n)u^{\#}(m)  
	\, \partial_{k}\mu_{R,1}^{\#,(2)}\left(k,\ell,n,m\right)
	\,d\ell dmdn\,ds.
	\end{align}

	\medskip
	\subsubsection{Proof of \eqref{weightmain2A}} 
	We first consider \eqref{NR11} which contains the additional growth in $s$. 
	Let $\phi_{1}$ and $\phi_{2}$ be smooth functions such that $\phi_{1}+\phi_{2}=1$,
	$\phi_{1}(k)=1$ for $|k|\leq 1$ and $\phi_{1}(k)=0$ for $|k|\geq2$.
	We will consider the case of low and high frequencies separately by looking at 
	$A_{1,\phi_{1}}$, respectively $A_{1,\phi_2}$, defined as $A_1$ in \eqref{NR11} with the addition
	of the cutoff $\phi_1$, respectively $\phi_2$.
	
	\smallskip
	\noindent
	\textit{Case I: low-frequency part.} 
	Consider the low-frequency part
	\[
	A_{1,\phi_{1}}(t,k) := \int_{0}^{t}\phi_{1}(k) \, iske^{-isk^{2}}
	\iiint u^{\#}(\ell)\overline{u^{\#}}(n)u^{\#}(m)\,\mu_{R,1}^{\#,(1)}\left(k,\ell,n,m\right)\,d\ell dmdn\,ds.
	\]
	Using the definition \eqref{eq:mur11} and the notation \eqref{decompphi} we write
	\begin{align*}
	& A_{1,\phi_{1}}(t,k) = (\sqrt{2\pi})^3\int_{0}^{t}\phi_{1}(k) \, ik \, e^{-isk^{2}} \Big(
	\int \overline{\mathcal{K}}^{\#}_{R}(x,k) \, s \, u_{M_1}(s,x) u_{M_2}(s,x) u_{M_3}(s,x) \, dx \Big) ds.
	\end{align*}
	Then, applying the smoothing estimate \eqref{eq:smoothingQim} 
	with $\mathcal{Q}(x,k)=\mathcal{K}_R^{\#}(x,k)$, $\beta > 1$, and 
	$F(s) = su_{M_{1}}(s)\overline{u_{M_{2}}}(s)u_{M_{3}}(s)$, followed by the linear estimates 
	\eqref{eq:singulardecaysep} and \eqref{regularinfty}, we can bound
	\begin{align*}
	\left\Vert A_{1,\phi_{1}}(k)\right\Vert _{L_{k}^{2}} 
	& \lesssim\left\Vert \left\langle x\right\rangle^{-\beta}
	s \, u_{M_{1}}(s)\overline{u_{M_{2}}(s)}u_{M_{3}}(s)\right\Vert _{L_{x}^{1}L_{s}^{2}[0,t]}
	\\
	& \lesssim\varepsilon^{3}\left\Vert s\left\langle s\right\rangle^{-3/2}\right\Vert_{L_{s}^{2}[0,t]}
	\lesssim\varepsilon^{3}\sqrt{\log \jt}
	\end{align*}
	which is sufficient.

	\smallskip
	\noindent
	\textit{Case II: high-frequency part.} 
	We now consider the high-frequency part:
	\[
	A_{1,\phi_{2}}(k):=\int_{0}^{t}\phi_{2}(k)iske^{-isk^{2}}
	\iiint u^{\#}(\ell)\overline{u^{\#}}(n)u^{\#}(m)\,\mu_{R,1}^{\#,(1)}\left(k,\ell,n,m\right)\,d\ell dmdn
	\,ds.
	\]
	In this case, we need an integration by parts argument, similar to Chen-Pusateri \cite{NLSV},
	to absorb the (possibly) large $k$ factor.
	More precisely, recalling the definition of $\mathcal{K}^{\#}_R$ in \eqref{KR} we write
	\begin{align}\label{KRid}
	ik\mathcal{K}^{\#}_{R}(x,k) = \partial_x\mathcal{K}_{\#,R}'(x,k) - \mathcal{K}_{\#,R}''(x,k),
	\end{align}
	where
	\begin{align}\label{KR1}
	\mathcal{K}_{\#,R}'(x,k):=\begin{cases}
	\chi_{+}(x) \big[ (T(k)-T(0))(m_{+}(x,0)-1) 
	\\
	\qquad + T(k) \left(m_{+}(x,k)-m_{+}(x,0) \right) \big] e^{ikx}
	\\
	\qquad +\chi_{-}(x) \big[\left(m_{-}(x,-k)-m_{-}(x,0)\right)e^{ixk}
	\\
	\qquad -R_{-}(k)\left(m_{-}(x,k)-1\right)e^{-ixk} \big], \qquad \qquad  k\geq0,
	\\
	\\
	-\chi_{-}(x)\Big[ (T(-k)-T(0))(m_{-}(x,0)-1)e^{ikx} 
	\\
	\qquad + T(-k)\left(m_{-}(x,-k)-m_{-}(x,0) \right) \Big]e^{ikx}
	\\
	\qquad -\chi_{+}(x) \big[ \left(m_{+}(x,k)-m_{+}(x,0)\right)e^{ikx}
	\\
	\qquad -R_{+}(-k)\left(m_{+}(x,-k)-1\right)e^{-ixk} \big], \qquad \qquad k<0,
	\end{cases}
	\end{align}
	and
	\begin{align}\label{KR2}
	\mathcal{K}_{\#,R}''(x,k):=\begin{cases}
	\partial_x \big\{\chi_{+}(x) \big[(T(k)-T(0))(m_{+}(x,0)-1) 
	\\ \qquad + T(k) \left(m_{+}(x,k)-m_{+}(x,0) \right) \big]\big\} e^{ikx}
	\\
	\qquad +\partial_x\left\{\chi_{-}(x) \left(m_{-}(x,-k)-m_{-}(x,0)\right)\right\}e^{ixk}
	\\
	\qquad -\partial_x \left\{\chi_{-}(x) R_{-}(k)\left(m_{-}(x,k)-1\right)\right\}e^{-ixk}, \qquad \qquad  k\geq0,
	\\
	\\
	-\partial_x \big\{\chi_{-}(x) \big[ (T(-k)-T(0))(m_{-}(x,0)-1)e^{ikx} 
	\\ 
	\qquad + T(-k) \left(m_{-}(x,-k)-m_{-}(x,0) \right) \big] \big\} e^{ikx}
	\\
	\qquad -\partial_x\left\{\chi_{+}(x)  \left(m_{+}(x,k)-m_{+}(x,0)\right)\right\}e^{ikx}
	\\
	\qquad +\partial_x\left\{\chi_{+}(x) R_{+}(-k)\left(m_{+}(x,-k)-1\right)\right\}e^{-ixk} , \qquad \qquad k<0.
	\end{cases}	
	\end{align}
	Then, with the definition \eqref{eq:mur11}, integrating by parts in $x$ we have
	\begin{align}\label{kmuR11}
	\begin{split}
	- \phi_{2}(k) ik  \, \mu_{R,1}^{(1)} 
	& = \int\overline{ \phi_{2}(k)\mathcal{K}_{\#,R}'(x,k)}(x,k) 
	\partial_x\big[ \mathcal{K}^{\#}_{M_{1}}(x,\ell)\overline{\mathcal{K}^{\#}_{M_{2}}(x,n)}\mathcal{K}^{\#}_{M_{3}}(x,m) \big]\,dx
	\\
	& + \int\overline{ \phi_{2}(k)\mathcal{K}_{\#,R}''(x,k) }(x,k)\mathcal{K}^{\#}_{M_{1}}(x,\ell)
	\overline{\mathcal{K}^{\#}_{M_{2}}(x,n)}\mathcal{K}^{\#}_{M_3}(x,m)\,dx.
	\end{split}
	\end{align}
	Therefore, using again the notation \eqref{decompphi},
	\begin{align}\label{A_1}
	\begin{split}
	- A_{1,\phi_{2}}(k) & = (2\pi)^{3/2} \int_0^t s \, \phi_{2}(k) e^{-isk^2} \int\overline{\mathcal{K}_{\#,R}'(x,k)} 
	\partial_x \big[ u_{M_{1}(s)}\overline{u_{M_{2}}(s)}u_{M_{3}}(s)\big]\,dx\,ds
	\\
	& + (2\pi)^{3/2} \int_0^t s \, \phi_{2}(k) e^{-isk^2} \int\overline{\mathcal{K}_{\#,R}''(x,k)}
	u_{M_{1}}(s)\overline{u_{M_{2}}(s)}u_{M_{3}}(s)\,dx\,ds 
	=: A_{1,1} + A_{1,2}.
	\end{split}
	\end{align}
	Up to an irrelevant constant, we write $A_{1,1}$ as
	\begin{align*}
	\begin{split}
	A_{1,1}(k) & = \int_0^t s \,\phi_{2}(k) e^{-isk^2} \int\overline{\mathcal{K}_{\#,R}'(x,k)} 
	\partial_x \big[ u_{M_{1}(s)}\overline{u_{M_{2}}(s)}u_{M_{3}}(s)\big]\,dx\,ds.
	\end{split}
	\end{align*}
	Focusing on the case where $\partial_x$ hits $u_{M_1}$ (the other cases are identical)
	we use the smoothing estimate 
	\eqref{eq:smoothingQimh} with $Q(x,k)=\mathcal{K}_{\#,R}'(x,k)$, $\beta>3/2$, 
	and $F(s) = s\partial_x u_{M_{1}(s)} \overline{u_{M_{2}}(s)}u_{M_{3}}(s)$ 
	to estimate 
	\begin{align*}
	{\| A_{1,1} \|}_{L^2} & \lesssim 
	{\big\|s \jx^{-\beta} \partial_x u_{M_{1}}(s)\overline{u_{M_{2}}(s)}u_{M_{3}}(s) \big\|}_{L^1_x L_s^{2}}
	\\
	& \lesssim \left\Vert s{\| \jx^{-1}\partial_x u_{M_{1}}(s) \|}_{L^2_x}
	{\| u_{M_{2}}(s) \|}_{L^\infty_x} {\| u_{M_{3}}(s) \|}_{L^\infty_x}\right\Vert_{L^2_s}
	\lesssim \varepsilon^3 \jt^\alpha, 
	\end{align*}
	where, for the last inequality, we have used \eqref{dxuM} and the a priori bounds to estimate 
	${\| \jx^{-1}\partial_x u_{M_{1}}(s) \|}_{L^2_x} \lesssim \varepsilon \js^\alpha$,
	and the linear decay estimates \eqref{eq:singulardecaysep} and \eqref{regularinfty}.

	To estimate the other term $A_{1,2}$ in \eqref{A_1}
	we can again use the smoothing estimate \eqref{eq:smoothingQimh} 
	with $Q(x,k)=\mathcal{K}_{\#,R}''(x,k)$, $\beta>1$ and $F(s,x)=su_{M_{1}}(s)\overline{u_{M_{2}}(s)}u_{M_{3}}(s)$, 
	together with pointwise decay to see that
	\begin{align*}
	{\| A_{1,2} \|}_{L^2} & 
	\lesssim {\big\|s \jx^{-\beta} u_{M_{1}}(s)\overline{u_{M_{2}}(s)}u_{M_{3}}(s) \big\|}_{L^1_x L^{2}_s}
	\lesssim \varepsilon^3 \sqrt{\log \jt}.
	\end{align*}
	
	\smallskip
	To conclude the proof of \eqref{weightmain2A} we need to estimate $A_2$, see \eqref{NR11'}.
	This is easier to deal with than $A_1$ since there is no additional growth in $s$.  
	We simply 
	note that $\partial_k \mathcal{K}^{\#}_R(x,k)$ gives a bounded 
	pseudo-differential operator on $L^2$, see 
	\eqref{eq:weiFR}, and estimate
	\begin{align*}
	\left\Vert A_{2}(t)\right\Vert _{L^{2}} 
	& \lesssim \int_0^t \left\Vert  \int \overline{\partial_k \mathcal{K}^{\#}_R(x,k)}
	\big( u_{M_{1}} \overline{u_{M_{2}}} u_{M_{3}} \big)(s,x) \, dx \right\Vert_{L^2_k}\,ds  
	\\
	& \lesssim \int_0^t \left\Vert u_{M_{1}}(s)\right\Vert_{L^\infty_x}
	\left\Vert u_{M_{2}}(s)\right\Vert _{L^\infty_x} \left\Vert u_{M_{3}}(s)\right\Vert_{L^2_x}\, ds
	\lesssim  \varepsilon^3 \log \jt.
	\end{align*}
	Note that for the above arguments we required $\jx^\gamma V \in L^1$ with $\gamma>\frac{5}{2}$.

	\smallskip
	\subsubsection{Proof of \eqref{weightmain2B}}
	For this we are going to use the localization in $x$ provided by $\mathcal{K}^{\#}_R(x,m)$,
	see the definitions \eqref{NR12}-\eqref{NR12'} and \eqref{eq:mur12},
	similarly to how we used the localization of $\mathcal{K}'_{R,\#}(x,k)$ and $\mathcal{K}_{R,\#}''(x,k)$ above.
	As done for $A_1$ before, we split $B_1$ into low and high frequency parts using an analogous notation.
	
	\smallskip
	\noindent
	\textit{Case I: low-frequency part.} 
	Let us consider the low-frequency part
	\[
	B_{1,\phi_{1}}(k):=\int_{0}^{t}\phi_{1}(k)iske^{-isk^{2}}\iiint u^{\#}(\ell)
	\overline{u^{\#}}(n)u^{\#}(m)\,\mu_{R,1}^{\#,(2)}\left(k,\ell,n,m\right)\,d\ell dmdn\,ds.
	\]
	By the low-frequency smoothing estimate \eqref{eq:smoothingQim} 
	with $Q(x,k)=\mathcal{K}^{\#}_{S}(x,k)$, $\beta=0$ and $F(s,x)=s u_{M_{1}}(s)\overline{u_{M_{2}}(s)}(u)_{R}(s)$,
	together with pointwise decay estimates, see \eqref{eq:singulardecay} and \eqref{regularinfty},
	we have, for $\beta > 1$,
	\begin{align*}
	\left\Vert B_{1,\phi_{1}}(k)\right\Vert _{L_{k}^{2}} & \lesssim {\big\| s u_{M_{1}}(s)
		\overline{u_{M_{2}}(s)}u_{R}(s) \big\|}_{L_{x}^{1}L_{s}^{2}[0,t]}
	\\
	& \lesssim {\big\| s \, {\| u_{M_{1}}(s) \|}_{L^\infty_x} {\| \overline{u_{M_{2}}(s)} \|}_{L^\infty_x} 
		{\| \jx^{\beta} u_{R}(s) \|}_{L^\infty_x} \big\|}_{L_{s}^{2}[0,t]}
	\\
	& \lesssim\varepsilon^{3} \left\Vert s\left\langle s\right\rangle^{-3/2}\right\Vert_{L_{s}^{2}[0,t]}
	\lesssim\varepsilon^{3}\sqrt{\log \jt}
	\end{align*}
	which is sufficient.

	\medskip
	\noindent
	\textit{Case II: high-frequency part.} 
	Next we consider the high-frequency part:
	\begin{equation}\label{eq:B1phi2}
	B_{1,\phi_{2}}(k):=\int_{0}^{t}\phi_{2}(k)iske^{-isk^{2}}\iiint u^{\#}(\ell)\overline{u^{\#}}(n)u^{\#}(m)
	\,\mu_{R,1}^{\#,(2)}\left(k,\ell,n,m\right)\,d\ell dmdn\,ds.
	\end{equation}
	As in the analogous earlier case, we will perform integration by parts to absorb the factor of $k$.
	We start with the following analogue of \eqref{KRid} for $\mathcal{K}^{\#}_S$:
	using the definitions \eqref{KsharpS}-\eqref{K-} we write
	\begin{align}\label{KSid}
	ik\mathcal{K}^{\#}_{S}(x,k) = \partial_x\mathcal{K}_{S,\#}'(x,k) - \mathcal{K}_{S,\#}''(x,k) 
	\end{align}
	with
	\begin{align}\label{KS1}
	\begin{split}
	& \mathcal{K}_{S,\#}'(x,k) := \chi_0(x)\mathcal{K}^{\#}_0(x,k) + \chi_+(x) \mathcal{K}'_{+,\#}(x,k) 
	+ \chi_-(x) \mathcal{K}'_{-,\#}(x,k),
	\\
	& \mathcal{K}'_{+,\#} := (T(k)-T(0)) \mathbf{1}_+(k) e^{ikx} + R_{+}(-k) \mathbf{1}_-(k) e^{-ikx},
	\\
	& \mathcal{K}'_{-,\#} :=- R_{-}(k)  \mathbf{1}_+(k) e^{-ikx} - (T(-k)-T(0) ) \mathbf{1}_-(k) e^{ikx},
	\end{split}
	\end{align}
	and
	\begin{align}\label{KS2}
	\begin{split}
	& \mathcal{K}_{S,\#}''(x,k) := \chi_0(x)T(0) \partial_x m_{+}(x,0)e^{ikx} + \partial_x \chi_+(x) \mathcal{K}'_{+,\#}(x,k) 
	+ \partial_x \chi_-(x) \mathcal{K}'_{-,\#}(x,k).
	\end{split}
	\end{align}
	Then, in analogy with \eqref{kmuR11} and \eqref{A_1}, we have
	\begin{align}\label{kmuR12a}
	- \phi_{2}(k)ik\mu_{R,1}^{(2)} 
	& := \int \overline{\phi_{2}(k)\mathcal{K}_{S,\#}'}(x,k) 
	\partial_x\big[ \mathcal{K}^{\#}_{M_{1}}(x,\ell)\overline{\mathcal{K}^{\#}_{M_{2}}(x,n)}
	\mathcal{K}^{\#}_{R}(x,m) \big]\,dx
	\\
	\label{kmuR12b}
	& + \int \overline{\phi_{2}(k)\mathcal{K}_{S,\#}''}(x,k))\mathcal{K}^{\#}_{M_{1}}(x,\ell)
	\overline{\mathcal{K}^{\#}_{M_{2}}(x,n)} \mathcal{K}^{\#}_{R}(x,m)\,dx,
	\end{align}
	and \eqref{eq:B1phi2} can be written as 
	\begin{align*}
	-B_{1,\phi_{2}}(k) & = (2\pi)^{3/2} \int_0^t s \, e^{-isk^2} \int \overline{\phi_{2}(k)\mathcal{K}_{S,\#}'}(x,k) 
	\partial_x \big[ u_{M_{1}}(s)\overline{u_{M_{2}}(s)}u_{R}(s)\big] \, dx\,ds
	\\
	& + (2\pi)^{3/2} \int_0^t s \, e^{-isk^2} \int \overline{\phi_{2}(k)\mathcal{K}_{S,\#}''}(x,k) 
	u_{M_{1}}(s)\overline{u_{M_{2}}(s)}u_{R}(s) \, dxds =: B_{1,1} + B_{1,2}.
	\end{align*}
	Applying the high-frequency smoothing estimate \eqref{eq:smoothingQimh} 
	with $Q(x,k)=\mathcal{K}_{S,\#}'(x,k)$, $\beta=0$ and 
	$F(s,x)=\partial_x \big[ u_{M_{1}}(s)\overline{u_{M_{2}}(s)}u_{R}(s)\big]$, we can estimate
	\begin{align*}
		{\| B_{1,1} \|}_{L^2_k}  \lesssim 
		{\| s \, \partial_x u_{M_{1}} \, \overline{u_{M_{2}}} \, u_R \|}_{L^1_x L_s^2}
		+ {\| s \, u_{M_{1}} \, \partial_x \overline{u_{M_{2}}} \, u_R \|}_{L^1_x L_s^2} 
		+ {\|s u_{M_{1}}\overline{u_{M_{2}}} \, \partial_x u_R \|}_{L^1_x L_s^2}.
		\end{align*}
		We then use the local $L^\infty$ estimates \eqref{eq:QdecayL2} (with $\beta=0$) for the derivative of $u_{M_1}$, 
		with the usual a priori assumptions,
		the linear decay estimates \eqref{eq:singulardecay} and \eqref{regularinfty} for $u_{M_2}$,
		and \eqref{regularinfty} with $\beta > 3/2$ for $u_R$, to obtain
		\begin{align*}
		{\big\| s \,\partial_x u_{M_1} \, \overline{u_{M_2}} \, u_R \|}_{L^1_x L_s^2([0,t])}
		& \lesssim  {\| s \, {\| \jx^{-1} \partial_x u_{M_1} \|}_{L^2_x} \, {\| \overline{u_{M_2}} \|}_{L^\infty_x} 
			{\| \jx^\beta u_R \|}_{L^\infty_x} \big\|}_{L_s^2([0,t])}
		\\
		& \lesssim {\| s \, \varepsilon \js^{-1/2+\alpha} \cdot 
			\varepsilon \js^{-1/2} \cdot \varepsilon \js^{-1/2} \big\|}_{L_s^2([0,t])}
		\lesssim \varepsilon^3 \jt^\alpha,
		\end{align*}
		with the same estimate interchanging the roles of $M_1$ and $M_2$,
		and, similarly, for any $\beta >1/2$,
		\begin{align*}
		{\|s u_{M_1} \overline{u_{M_2}} \, \partial_x u_R \|}_{L^1_x L_s^2([0,t])}
		& \lesssim {\| s \, {\| u_{M_1} \|}_{L^\infty_x} \, {\| \overline{u_{M_2}} \|}_{L^\infty_x} 
			{\| \jx^\beta \partial_x u_R \|}_{L^2_x} \big\|}_{L_s^2([0,t])}
		\\
		& \lesssim {\| s \, \varepsilon \js^{-1/2} \cdot 
			\varepsilon \js^{-1/2} \cdot \varepsilon \js^{-1/2+\alpha} \big\|}_{L_s^2([0,t])}
		\lesssim \varepsilon^3 \jt^\alpha,
		\end{align*}
		Putting these together gives us ${\| B_{1,1} \|}_{L^2_k} \lesssim \varepsilon^3 \jt^\alpha$.
	Here all the decay estimates can be applied provided $\jx ^{\gamma}V\in L^1$ with $\gamma>\frac{5}{2}$.
	
	Similarly, by the high-frequency smoothing estimate \eqref{eq:smoothingQimh} and pointwise decay,  
	we can estimate, for any $\rho>1$,
	\begin{align*}
	{\| B_{1,2} \|}_{L^2_k} & \lesssim  {\| s  u_{M_{1}}(s)\overline{u_{M_{2}}}(s) u_{R}(s)\|}_{L^1_x L^2_s}
	\\ & \lesssim \left\Vert s {\|  u_{M_{1}} (s)\|}_{L^\infty} {\| u_{M_{2}}(s) \|}_{L^\infty} 
	{\|\jx^{\rho}u_{R}(s)\|}_{L^\infty} \right \Vert_{L^2_s} 
	\lesssim \varepsilon^3 \sqrt{\log \jt}
	\end{align*}
	which again suffices.
	
	\smallskip
	The last term \eqref{NR12'} is easier to treat since there is no extra power of $s$ in front 
	of the nonlinearity. We can use the following PDO bound (see for example \eqref{eq:weiF} and \eqref{eq:weiFR}) 
	\begin{align}\label{mapKS'}
	{\Big\| \int \partial_k \overline{\mathcal{K}^{\#}_{S}}(x,k) F(x) \, dx \Big\|}_{L^2_k} \lesssim {\| \jx F \|}_{L^2},
	\end{align}
	with $F =u_{M_{1}}\overline{u_{M_{2}}}u_{R}$, 
	and then again pointwise decay estimates to bound
	\begin{align*}
	{\| B_2(t) \|}_{L^2_k} 
	& \lesssim  \int_0^t {\| u_{M_{1}}(s)\|}_{L^\infty} {\| u_{M_{2}}(s) \|}_{L^\infty}
	{\| \jx u_{R}(s)\|}_{L^2} \, ds
	\lesssim   \int_0^t \varepsilon^3 \left\langle s\right\rangle^{-3/2} \lesssim 
	\varepsilon^3.
	\end{align*}
	

	\smallskip
	\subsection{Estimate of $\mathcal{N}_{R,2}$}\label{ssecNR2}
	Finally, we estimate $\mathcal{N}_{R,2}$ which is the nonlinear term associated with the measure in \eqref{muR2},
	whose generic terms is of the form \eqref{muR2'}.
	Without loss of generality, it suffices to estimate the contribution to $\mathcal{N}_{R,2}$
	coming from such a generic measure; that is, it suffices to estimate
	the $L^2$ norm of $\partial_k$ of the expression
	\begin{align}\label{muR2est1}
	\begin{split}
	\int_0^t \iiint e^{is(-k^2+\ell^2-m^2+n^2)}
	f^{\#}(s,\ell)\overline{f^{\#}(s,m)} f^{\#}(s,n)\mu_{R,2}^{\#,(0)}(k,\ell,m,n)\,dndmd\ell \,ds
	\\ 
	=  \overline{a_{\e_1}(k)} \int_0^t e^{-isk^2} 
	\int_{\R_x}\varPsi(x)e^{-\e_1 ikx} u_{\e_2}(x) 
	\overline{u_{\e_3}(x)} u_{\e_4}(x)\, dx \, ds
	\end{split}
	\end{align}
	where, for the last identity, we used the notation
	\begin{align}\label{uepsnot} 
	u_{\e_j}(x) := \frac{1}{\sqrt{2\pi}}\int a_{\e_j}(k)e^{\e_j ikx} u^{\#}(k) \, dk. 
	\end{align}
	Note that we are slightly abusing notation here compared to the definition of $u_\pm$ in \eqref{decompphi};
	however these latter and \eqref{uepsnot} are essentially the same object,
	and since we will only use \eqref{uepsnot} in this subsection this should not cause any confusion.
	%
	%
	
	Applying $\partial_k$ to \eqref{muR2est1} will give two contributions, one when the derivative hits the exponential
	and the other one when it hits $\overline{a_{\e_1}(k)}e^{-\e_1 ikx}$.
	We disregard this second one since it is much easier to estimate.
	We are then left with estimating the $L^2_k$ norm of 
	\begin{align}\label{muR2est2}
	\int_0^t is k \, e^{-isk^2} \int \varPsi(x)e^{-\e_1 ikx} u_{\e_2}(x) \overline{u_{\e_3}(x)} u_{\e_4}(x)\, dx \, ds.
	\end{align}
	The analysis is quite similar to the estimates for $\mu^{\#}_{R,1}$ above.  
	We again decompose the term into low and high-frequencies
	letting $\phi_{1}$ and $\phi_{2}$ be smooth functions such that $\phi_{1}+\phi_{2}=1$,
	$\phi_{1}(k)=1$ for $|k|\leq1$ and $\phi_{1}(k)=0$ for $|k|\geq2$.
	
	\smallskip
	\noindent
	\textit{Case I: low-frequency part.} 
	Consider the low-frequency part
	\[
	D_{\phi_{1}}(k):=\int_{0}^{t}\phi_{1}(k) isk \,e^{-isk^{2}}\int 
	\varPsi(x)e^{-\e_1 ikx} u_{\e_2}(x) \overline{u_{\e_3}(x)} u_{\e_4}(x)\, dx \, ds.
	\]
	By the smoothing estimate \eqref{eq:smoothingQim} and pointwise decay, we have
	\begin{align*}
	\left\Vert D_{\phi_{1}}(k)\right\Vert _{L_{k}^{2}} 
	& \lesssim
	{\big\| s\varPsi(x)e^{-\e_1 ikx} u_{\e_2}(x) \overline{u_{\e_3}(x)} u_{\e_4}(x) \big\|}_{L_{x}^{1}L_{s}^{2}([0,t])}
	\\
	& \lesssim \varepsilon^{3} 
	{\big\| s\left\langle s\right\rangle^{-3/2} \big\|}_{L_{s}^{2}([0,t])}
	\lesssim \varepsilon^{3}\sqrt{\log \jt}
	\end{align*}
	which is sufficient. 
	Note that in the last line above we used the decay in \eqref{muR2'} for $\varPsi(x)$.

	\smallskip
	\noindent
	\textit{Case II: high-frequency part.} 
	Consider now the high-frequency part
	\[
	D_{\phi_{2}}(k):=\int_{0}^{t}\phi_{2}(k)isk \, e^{-isk^2} 
	\int \varPsi(x)e^{-\e_1 ikx} u_{\e_2}(x) \overline{u_{\e_3}(x)} u_{\e_4}(x)\, dx\,ds.
	\]
	Similarly to what we did before in \eqref{KRid} and \eqref{KSid} we 
	first convert the factor of $k$ into an $x$ derivative,
	and integrate by parts in $x$ in \eqref{muR2est2} to get, up to a constant,
	\begin{align}\label{muR2est3}
	\int_0^t \phi_{2}(k) s \, e^{-isk^{2}} 
	\int e^{-\eps_1 ikx} 
	\partial_x\big[ \varPsi(x) u_{\e_2}(x) \overline{u_{\e_3}(x)} u_{\e_4}(x) \big]\, dx \, ds;
	\end{align}
	we then estimate the term above in $L^2_k$
	using the high-frequency smoothing estimate \eqref{eq:smoothingQimh},
	and distributing the derivative:
	\begin{align*}
	\begin{split}
	\Big\| \int_0^t s\phi_{2}(k) e^{-isk^{2}} 
	\int e^{\mp ikx}  \partial_x
	\big[ \varPsi(x)  u_{\e_2}(x) \overline{u_{\e_3}(x)} u_{\e_4}(x) \big]\, dx\,ds \Big\|_{L^2_k}
	\\
	\lesssim {\big\| s \partial_x\big[ \varPsi(x) u_{\e_2}(x) \overline{u_{\e_3}(x)} u_{\e_4}(x) \big] 
	\big\|}_{L^1_x L_s^2([0,t])} 
	\lesssim \varepsilon^3 \jt^\alpha.
	\end{split}
	\end{align*}
	where in last line we used  $|\partial_x^j \varPsi|\lesssim \jx^{-\gamma+1}$ for $j=0,1$,
	with $\gamma > 5/2$,
	and the estimate \eqref{eq:QdecayL2} to deduce 
	${\| \jx^{-1} \partial_x u_{\eps} \|}_{L^2} \lesssim \varepsilon \js^{-1/2+\alpha}$.
	
	This concludes the estimate of $\mathcal{N}_{R,2}$
	and of \eqref{weightmain2}.
	The estimates for the regular part of the nonlinearity $\mathcal{N}_R$ are completed and 
	Proposition \ref{pro:weightmain} 
	follows.

\medskip
\section{The Fourier $L^\infty$ bound and asymptotics}\label{secinfty}
In this last section we give the proof of \eqref{fODE} in Proposition \ref{propinfty} and 
therefore also obtain the asymptotics described in Remark \ref{remmodscatt}.
The arguments that follow are similar to previous works, including our paper \cite{NLSV}.
However, since we cannot use the vanishing condition on the distorted transform at zero,
we need some variants of the arguments in \cite{NLSV} that, in particular, exploit cancellations
in some of the nonlinear coefficients; 
see for example the formulas \eqref{N2'}-\eqref{N22} and the paragraph that follows.

\smallskip
\subsection{Preliminaries}
In this section we find it more convenient to set-up the analysis using the notation
from \cite{NLSV} for the generalized eigenfunctions and the various components of the NSD $\mu$. 
In particular, we will begin by using the distorted Fourier transform $\wtF$, and then relate 
it to the $\mathcal{F}^\sharp$ transform.

\begin{rem}[The case of an even resonance]\label{remevenres}
Note that in the case of an even resonance we do not need to introduce (and use) $\mathcal{F}^\sharp$,
and we can directly work with $\wtF$.
All the calculations that follow below still apply almost verbatim, 
without the need to transition to $\mathcal{F}^\sharp$.
The only modification to the argument is in the treatment of \eqref{N2}; see Remark \ref{remN2}.
\end{rem}


\subsubsection{Decomposition of the generalized eigenfunctions}
Let $\mathcal{K}$ be the generalized distorted eigenfunctions from \eqref{matK}.
In \cite{NLSV}, we introduced the following decomposition: 
\begin{equation}\label{eq:decomK}
\sqrt{2\pi}\mathcal{K}(x,k) = \mathcal{K}_{S}(x,k) + \mathcal{K}_{R}(x,k),
\end{equation}
where the singular part is
\begin{align}\label{Ksing}
\mathcal{K}_{S}(x,k) = \chi_+(x)\mathcal{K}_+(x,k) + \chi_-(x)\mathcal{K}_-(x,k)
\end{align}
with
\begin{align}\label{Ksing+}
\mathcal{K}_+(x,k) & = 
\begin{cases}
T(k) e^{ikx} & k\geq0
\\
e^{ikx} + R_+(-k)e^{-ikx}  & k<0,
\end{cases}
\\
\label{Ksing-}
\mathcal{K}_-(x,k) & = 
\begin{cases}
e^{ikx} + R_-(k) e^{-ikx} & k\geq0
\\
T(-k) e^{ikx} & k<0,
\end{cases}
\end{align}
and the regular part is
\begin{align}\label{KRasy}
\mathcal{K}_{R}(x,k) :=
\begin{cases}
\chi_+(x)T(k)\left(m_+(x,k)-1\right)e^{ixk}
\\
+ \chi_-(x)\left[\left(m_-(x,-k)-1\right)e^{ixk}
+ R_-(k) \left(m_-(x,k)-1\right)e^{-ixk}\right], & k\geq0,
\\
\\
\chi_-(x)T(-k)\left(m_-(x,-k)-1\right)e^{ixk}
\\
+\chi_+(x)\left[\left(m_+(x,k)-1\right)e^{ixk}+R_+(-k)\left(m_+(x,-k)-1\right)e^{-ixk}\right],
& k<0.
\end{cases}
\end{align}
It is also useful to rewrite \eqref{Ksing+}-\eqref{Ksing-} as 
a linear combination of exponentials and coefficients\footnote{The $b^\epsilon_\iota$
coefficients below are the same as the $a^\epsilon_\iota$ coefficients in
our previous work \cite{NLSV}, but we renamed them here to avoid confusion with 
the coefficients in \eqref{K+}-\eqref{K-}.}:
\begin{align}\label{Ksing+'}
\begin{split}
& \mathcal{K}_+(x,k) = b_+^+(k) e^{ikx} + b_+^-(k) e^{-ikx},
\\
& \mbox{with} \qquad b^+_+(k) := \mathbf{1}_{+}(k)T(k) + \mathbf{1}_{-}(k), \qquad b_+^-(k) := \mathbf{1}_{-}(k)R_+(-k),
\end{split}
\end{align}
and
\begin{align}\label{Ksing-'}
\begin{split}
& \mathcal{K}_-(x,k) = b_-^+(k) e^{ikx} + b_-^-(k) e^{-ikx},
\\
& \mbox{with} \qquad b_-^+(k) = \mathbf{1}_{+}(k) + \mathbf{1}_{-}(k)T(-k), \qquad b^-_-(k) = \mathbf{1}_{+}(k)R_-(k).
\end{split}
\end{align}
Note that the coefficients $b^+_\iota$ have a jump discontinuity at $k=0$,
while the $b^-_\iota$ coefficients are Lipschitz 
(recall that $T(0)=-1$ and $R_{\pm}(0)=0$, see Lemma \ref{lemngTR}, and Lemma \ref{estiTR}).

\begin{rem}[Regularity of the coefficients]\label{remcoeffLip}
Note that 
\begin{align}\label{afLip}
b^\epsilon_{\iota}(k) \wt{f}(t,k) = \sign(k) b^\epsilon_{\iota}(k) \, f^\sharp(t,k) 
\end{align}
where, according to \eqref{Ksing+'}-\eqref{Ksing-'}, the coefficients
that multiply $f^\sharp$ in the above right-hand side are given by
\begin{align}\label{aLip}
\begin{split}
& \sign(k) b^+_+(k) := \mathbf{1}_{+}(k)T(k) - \mathbf{1}_{-}(k), 
  \qquad \sign(k)b_+^-(k) = -\mathbf{1}_{-}(k)R_+(-k),
\\
& \sign(k) b_-^+(k) = \mathbf{1}_{+}(k) - \mathbf{1}_{-}(k)T(-k), 
  \qquad \sign(k)b^-_-(k) = \mathbf{1}_{+}(k)R_-(k).
\end{split}
\end{align}
In particular, 
all the expressions in \eqref{aLip} are Lipschitz functions.
\end{rem}



\subsubsection{Linear decay}
Using \eqref{afLip}-\eqref{aLip}, the next lemma gives bounds on norms that involve $\wt{f}$,
and decay estimates for projections on $\mathcal{K}_S$ and $\mathcal{K}_R$
under our a priori assumptions on $f^\sharp$.

\begin{lem}
Under the a priori assumptions \eqref{aprweias} 
we have 
\begin{equation}\label{eq:atildef}
    \Vert b^\epsilon_{\iota}(\cdot)\wt{f}(t,\cdot)\Vert_{L^{\infty}} 
    + \Vert b^\epsilon_{\iota}(\cdot)\wt{f}(t,\cdot) \Vert_{L^{2}} + 
    \jt^{-\alpha} \Vert \partial_k \big( b^\epsilon_{\iota} (\cdot)\wt{f}(t,\cdot)\big) \Vert_{L^{2}} \lesssim 
    \varepsilon.
\end{equation}
In particular, the following decay estimates hold true:
\begin{align}
\label{eq:decaysingtilde}
\left\Vert \int  \overline{\mathcal{K}_S}(x,k) e^{itk^2} \wt{f}(t,k)\,dk 
  \right\Vert_{L^{\infty}_x} \lesssim \varepsilon \jt^{-1/2},
\end{align}
and, for $\beta\leq \gamma-1$,
\begin{align}
\label{eq:decayregtilde}
    \left\Vert \jx^{\beta}\int  \overline{\mathcal{K}_R}(x,k) e^{itk^2} \wt{f}(t,k)\,dk \right\Vert_{L^{\infty}_x} 
    \lesssim \varepsilon \jt^{-1/2}.
\end{align}
\end{lem}

\begin{proof}
The proof of \eqref{eq:atildef} follows from the identities \eqref{afLip}-\eqref{aLip},
the a priori assumptions \eqref{aprweias} and Lemmas \ref{estiTR} and  \ref{lemngTR}.

The decay estimates \eqref{eq:decaysingtilde} and \eqref{eq:decayregtilde} can be obtained
using the same proofs of \eqref{eq:singulardecay} and \eqref{regularinfty}.
For the sake of completeness we give the proof for \eqref{eq:decaysingtilde}. 
In view of Sobolev's embedding $H^1 \subset L^\infty$ and \eqref{Sob} we may assume $|t|\geq 1$.
We only consider the contribution to the integral with $k\geq0$ since the case $k<0$ is similar. 
For $\iota \in \{+,-\}$ we define
\begin{equation}\label{Jiota}
J_{\iota}(t,x) := \chi_{\iota}(x)\int_{k\geq0} \overline{\mathcal{K}_{\iota}(x,k)} e^{ik^2t} \wt{f}(t,k)\,dk
  = \chi_{\iota}(x)  \sum_{\eps \in \{+,-\}} \int_{k\geq0} \overline{b_\iota^\eps(k) e^{\eps ikx}}
  e^{ik^{2}t}\wt{f}(t,k)\,dk.
\end{equation}
By \eqref{Ksing}, to show \eqref{eq:decaysingtilde}, it suffices to show decay for $J_\iota$.
It also suffices to look at the term with $\eps=+$ in \eqref{Jiota}, since the case $\eps=-$ is completely analogous.
We then write 
\begin{align*}
\begin{split}
& J_{\iota}(t,x) =  \sum_{\eps \in \{+,-\}} e^{ -i x^2/(4t)} I_{\iota}^\eps (t,X_\eps,x), 
\qquad X_\eps = \eps x/(2 t) 
\\
& I_{\iota}^\eps(t,X,x) := \int_{k \geq 0} e^{it (k-X)^2}  
  \overline{b^\epsilon_{\iota}(k)} \, \wt{f}(t,k)\, dk. 
\end{split}
\end{align*}
Using Lemma \ref{lemstat} with 
$a(x,k) =1$ and $h(k) = \overline{b^\eps_{\iota}(k)} \wt{f}(t,k)$, gives
$$ |J_{\iota}(t,x)| \lesssim \sum_{\eps \in \{+,-\}}
  \frac{1}{\sqrt{|t|}} {\|  b^\eps_{\iota} (\cdot)\wt{f}(t,\cdot) \|}_{L^\infty} 
  + \frac{1}{|t|^{\frac{3}{4}}} {\| \partial_k ( b^\eps_{\iota} (\cdot)\wt{f}(t,\cdot) ) \|}_{L^2}
  \lesssim \varepsilon |t|^{-1/2},
  $$
having used \eqref{eq:atildef} for the last inequality.
\end{proof}


\smallskip
\subsubsection{Decomposition of the NSD}
To obtain asymptotics for $f^{\#}(t,k)$, we first look at the equation for $\wt{f}(t,k)$:
\begin{equation}\label{ODE}
\partial_{t}\widetilde{f}(t,k) = \mp i\iiint e^{it\left(-k^{2}+\ell^{2}-m^{2}+n^{2}\right)}
  \wt{f}(t,\ell) \overline{\wt{f}(t,m)} \wt{f}(t,n)\,\mu(k,\ell,m,n) \, d\ell dmdn,
\end{equation}
where, similarly to \eqref{musharpdef}, the nonlinear spectral measure is given by
\begin{equation}
\mu(k,\ell,m,n):=\int\overline{\mathcal{K}(x,k)}\mathcal{K}(x,\ell)\overline{\mathcal{K}(x,m)}\mathcal{K}(x,n)\,dx.
\end{equation}
Given the decomposition \eqref{eq:decomK} above, we can decompose $\mu$ into two pieces:
\begin{equation}\label{muSR}
 (2\pi)^2\mu (k,\ell,m,n)=\mu_S (k,\ell,m,n)+\mu_R (k,\ell,m,n);
\end{equation}
the singular part $\mu_{S}$ is given by
\begin{align}\label{muSpmasy}
\begin{split}
\mu_S(k,\ell,m,n) & = \mu_+(k,\ell,m,n) + \mu_-(k,\ell,m,n),
\\
\mu_{\ast}(k,\ell,m,n) & := \int \varphi_{\ast}(x)
  \overline{\K_\ast(x,k)} \K_\ast(x,\ell) \overline{\K_\ast(x,m)} \K_\ast(x,n) \, dx, 
  \qquad \varphi_{\ast}=\chi_{\ast}^{4}, \quad
  \ast \in \{+,-\};
\end{split}
\end{align}
the regular part $\mu_{R}$  is given by
\begin{equation}\label{eq:muRasy}
\mu_{R}(k,\ell,m,n) = \mu_{R,1}(k,\ell,m,n) + \mu_{R,2}(k,\ell,m,n)
\end{equation}
where
\begin{equation}\label{muR1asy}
\begin{split}
\mu_{R,1}(k,\ell,m,n):=\sum_{(A,B,C,D)\in\mathcal{X}_{R}}\int\overline{\mathcal{K}_{A}(x,k)}
  \mathcal{K}_{B}(x,\ell)\overline{\mathcal{K}_{C}(x,m)}\mathcal{K}_{D}(x,n)\,dx,
\\ 
\mathcal{X}_{R} := \left\{ (A_{0},A_{1},A_{2},A_{3}) \in \{S,R\}^4 \, : \,\exists \, j=0,1,2,3 \,:\, A_{j}=R\right\}, 
\end{split}
\end{equation}
and
\begin{equation}\label{muR2asy}
\mu_{R,2}(k,\ell,m,n) :=
\int\overline{\mathcal{K}_S(x,k)}\mathcal{K}_S(x,\ell)\overline{\mathcal{K}_S(x,m)}\mathcal{K}_S(x,n)\,dx
  - \mu_{S}(k,\ell,m,n). 
\end{equation}
See \cite[Subsection 4.1]{NLSV} for more details.

According to \eqref{muSR} we decompose the nonlinear terms similarly to \eqref{Nast},
and write Duhamel's formula as
\begin{align}\label{ODE'}
\begin{split}
& i\partial_{t}\widetilde{f}(t,k) = \pm \big[ \mathcal{N}_{S}(t,k) + \mathcal{N}_{R}(t,k) \big],
\\
& (2\pi)^{2} \mathcal{N}_{\ast}(t,k) = \iiint e^{it(-k^{2}+\ell^{2}-m^{2}+n^{2})}
  \widetilde{f}(t,\ell)\overline{\wt{f}(t,m)} \widetilde{f}(t,n)\,\mu_\ast(k,\ell,m,n) \, d\ell dmdn.
\end{split}
\end{align}

\smallskip
\subsection{Asymptotics of the regular part}\label{ssecRasy}
We begin by showing that the regular part only contributes to the lower order reminder terms in \eqref{fODE}.

\begin{prop}\label{prop:asyregular}
Under the a priori assumptions \eqref{aprweias},  one has
\begin{equation}\label{ODEregular}
{\| \mathcal{N}_{R}\left(t,\cdot\right) \|}_{L^{\infty}_k} \lesssim \varepsilon^3 \jt^{-3/2}.
\end{equation}
\end{prop}

\begin{proof}
According to the structure of the regular part of the measure, \eqref{eq:muRasy}, 
it suffices to estimate the nonlinear terms $\mathcal{N}_{R,1}$ and $\mathcal{N}_{R,2}$ 
associated with $\mu_{R,1}$ and $\mu_{R,2}$ respectively, as in \eqref{ODE'}.

We first consider the contribution from $\mathcal{N}_{R,1}$.  
By definition, $\mu_{R,1}$ is a linear combination of the following two types of terms,
up to similar ones obtained by permuting the variables:
\begin{align}
\label{eq:mur11asy}
\mu_{R,1}^{(1)}(k,\ell,m,n) &:=\int\overline{\mathcal{K}}_{R}(x,k)
  \mathcal{K}_{M_{1}}(x,\ell)\overline{\mathcal{K}_{M_{2}}(x,n)}\mathcal{K}_{M_{3}}(x,m)\,dx,
\\
\label{eq:mur12asy}
\mu_{R,1}^{(2)}(k,\ell,m,n) &:=\int\overline{\mathcal{K}}_{S}(x,k)\mathcal{K}_{M_{1}}(x,\ell)
  \overline{\mathcal{K}_{M_{2}}(x,n)}\mathcal{K}_{R}(x,m)\,dx,
  \qquad M_{i}\in\{S,R\}.
\end{align}
Consider the the nonlinear terms associated to the measures above:
\begin{align*}
    \mathcal{N}^{(1)}_{R,1}\left(t,k\right)=\iiint e^{it\left(-k^{2}+\ell^{2}-m^{2}+n^{2}\right)}
  \widetilde{f}(t,\ell)\overline{\wt{f}(t,m)}\widetilde{f}(t,n)\,\mu_{R,1}^{(1)}(k,\ell,m,n) \, d\ell dmdn,
  \\
   \mathcal{N}^{(2)}_{R,1}\left(t,k\right)=\iiint e^{it\left(-k^{2}+\ell^{2}-m^{2}+n^{2}\right)}
  \widetilde{f}(t,\ell)\overline{\wt{f}(t,m)}\widetilde{f}(t,n)\,\mu_{R,1}^{(2)}(k,\ell,m,n) \, d\ell dmdn.
\end{align*}
 
Let us adopt the notation:
\begin{align}\label{eq:notaasy}
u_\ast(t,x) := 
  \int \mathcal{K}_\ast(x,k) e^{itk^2}\wt{f}(t,k) \, dk,
\end{align}
so that, in particular,
\begin{align*}
\mathcal{N}^{(1)}_{R,1}(t,k) & = e^{-itk^2} \int\overline{\mathcal{K}_{R}(x,k)} u_{M_1}(t,x)
  \overline{u_{M_2}}(t,x) u_{M_3}(t,x)\,dx.
\end{align*}
From the definition of $\mathcal{K}_R$ in \eqref{KR} and Lemma \ref{Mestimates1}
we have that $|\jx^{1+} \mathcal{K}_{R}(x,k) | \lesssim 1$, 
and, therefore, we can estimate
\begin{align*}
 {\big\| \mathcal{N}^{(1)}_{R,1}\left(t,\cdot\right) \big\|}_{L^{\infty}_k}
  & = {\Big\|  \int\overline{\mathcal{K}_{R}(x,k)} u_{M_1}(t,\cdot)
  \overline{u_{M_2}}(t,\cdot) u_{M_3}(t,\cdot)\,dx \Big\|}_{L^{\infty}_k}\nonumber
  \\
 &\lesssim 
 \left\Vert  u_{M_1}(t,\cdot) \overline{u_{M_2}}(t,\cdot) u_{M_3}(t,\cdot)\right\Vert_{L^{\infty}_x} 
 \lesssim \varepsilon^3 \jt^{-\frac{3}{2}},
\end{align*}
where in the last line, we used \eqref{eq:decaysingtilde} and \eqref{eq:decayregtilde}.
The other terms are estimated similarly:
\begin{align*}
 \left\Vert  \mathcal{N}^{(2)}_{R,1}\left(t,\cdot\right)\right\Vert _{L^{\infty}}
 &=\left\Vert  \int\overline{\mathcal{K}}_{S}(x,k) u_{M_1}(t,\cdot) \overline{u_{M_2}}(t,\cdot) 
 u_{R}(t,\cdot)\,dx\right\Vert _{L^{\infty}_k}
 \\
 &\lesssim \left\Vert  u_{M_1}(t,\cdot) \overline{u_{M_2}}(t,\cdot){\langle \cdot \rangle}^{1+\rho}
 u_{R}(t,\cdot)\right\Vert _{L^{\infty}} \lesssim \varepsilon^3 t^{-\frac{3}{2}}
\end{align*}
having used again \eqref{eq:decaysingtilde} and \eqref{eq:decayregtilde}.

Finally, we consider the contribution of the nonlinear terms associated with $\mu_{R,2}$ in \eqref{muR2asy}.
We first notice that
\begin{align}\label{muR2asy''}
\begin{split}
\mu_{R,2}(k,\ell,m,n) = \sum 
  \int \overline{\chi_{\e_0}(x)\mathcal{K}_{\e_0}(x,k)}
  \chi_{\e_1}(x)\mathcal{K}_{\e_1}(x,\ell) 
  \overline{\chi_{\e_2}(x)\mathcal{K}_{\e_2}(x,m)}\chi_{\e_3}(x)\mathcal{K}_{\e_3}(x,n)\,dx
\end{split}
\end{align}
where the sum is over
\begin{align*}
\e_0,\e_1,\e_2,\e_3 \in \{+,-\} \quad \mbox{with} \quad (\e_0,\e_1,\e_2,\e_3) \neq (+,+,+,+), \,(-,-,-,-).
\end{align*}
Then we note  the integrand in \eqref{muR2asy''} is compactly supported in $x$
since $\chi_+\chi_-$ is compactly supported.
Without loss of generality, it suffices to estimate the contribution to $\mathcal{N}_{R,2}$ from the measure
\begin{align}\label{muR2est0}
\mu_{R,2}^{(0)} := \int \overline{\chi_+(x)\mathcal{K}_+(x,k)}
  \chi_-(x)\mathcal{K}_-(x,\ell)\overline{\mathcal{K}_S(x,m)}\mathcal{K}_S(x,n)\,dx;
\end{align}
all the other contributions from the sum in \eqref{muR2asy''} can be estimated identically. 
Explicitly, we can write
\begin{align*}
\begin{split}
\mathcal{N}^{(0)}_{R,2}(t,k) := \iiint e^{it(-k^2+\ell^2-m^2+n^2)}
  \wt{f}(t,\ell)\overline{\wt{f}(t,m)}\wt{f}(t,n)\mu_{R,2}^{(0)}(k,\ell,m,n)\,dndmd\ell
  \\ 
  = e^{-itk^2} 
  \int \chi_+(x)\mathcal{K}_+(x,k) \chi_-(x) u_{-}(t,x) \overline{u_S(t,x)} u_S(t,x)\, dx
\end{split}
\end{align*}
where we use the same notation as in \eqref{eq:notaasy} to define $u_{\pm}$. 
Then 
we can estimate
\begin{align*}
{\big\| \mathcal{N}^{(0)}_{R,2}\left(t,\cdot\right) \big\|}_{L^{\infty}_k} &
 = C \left\Vert  \int \chi_+(x)\mathcal{K}_+(x,k) \chi_-(x) u_{-}(t,x) \overline{u_S(t,x)} u_S(t,x)\,
 dx \right\Vert _{L^{\infty}_k}
 \\
 &\lesssim \left\Vert  u_{-}(t,\cdot) \overline{u_{S}}(t,\cdot) u_{S}(t,\cdot)\right\Vert _{L^{\infty}_x} 
 \lesssim \varepsilon^3 \jt^{-\frac{3}{2}}.
\end{align*} 
This completes the proof of \eqref{ODEregular}. 
\end{proof}


\smallskip
\subsection{Asymptotics of singular part}\label{ssecSasy}
It remains to compute asymptotics for $\mathcal{N}_{S}\left(t,k\right)$.
The computations here are similar to \cite{NLSV} and the eventual outcome is the same.
However, since we cannot use the property $\wt{f}(t,0)=0$ 
we need to refine some of the arguments.

\smallskip
\subsubsection{Some useful lemmas}
The first step is the following standard stationary phase-type lemma: 

\begin{lem}
\label{AsLem1}
For $k,t \in \R$, consider the integral expression
\begin{align}\label{I1}
\begin{split}
I[g_1,g_2,g_3](t,k) & = \iiint e^{it \Phi(k,p,m,n)}
	g_1(\epsilon_1(\epsilon_0 k - p + \epsilon_2 m - \epsilon_3 n)) \overline{g_2(m)} 
	g_3(n) \pv \frac{\widehat{\phi}(p)}{p} \, dm dn dp
\\
\Phi(k,p,m,n) & = -k^2 + (\epsilon_0 k - p + \epsilon_2 m - \epsilon_3 n)^2 - m^2 + n^2,
\end{split}
\end{align}
for an even bump function $\phi \in C_0^\infty$ with integral $1$, 
and with $g:=(g_1,g_2,g_3)$ satisfying
\begin{align}
\label{AsLem1as}
{\| g(t) \|}_{L^\infty} + {\| \langle k \rangle g(t) \|}_{L^2} + \langle t \rangle^{-\alpha} {\| g'(t) \|}_{L^2} \leq 1,
\end{align}
for some  $\alpha > 0 $ small enough. 
Then, for any $t \in \R$,
\begin{align}\label{I2}
\begin{split}
I[g_1,g_2,g_3](t,k) = \frac{\pi}{|t|} e^{-itk^2} \int e^{it(-p+\epsilon_0 k)^2} g_1(\epsilon_1(-p+\epsilon_0 k))
	\overline{g_2(\epsilon_2(-p+\epsilon_0 k))}
	\\ \times  g_3(\epsilon_3 (-p+\epsilon_0 k)) \mathrm{p.v.} \frac{\widehat{\phi}(p)}{p} \,dp
	+ \mathcal{O}(|t|^{-1-\rho})
\end{split}
\end{align}
for some $\rho>0$.
\end{lem}

\begin{proof}
See Lemma 5.1 in \cite{GPR}. 
\end{proof}

To obtain the leading order behavior of expressions like those in \eqref{I2} 
we are going to use the following lemma: 

\begin{lem}\label{lemStat}
Consider the principal value integral 
\begin{align}\label{lemStatint}
I(t,K) := e^{-itK^2} a(K) \, \mathrm{p.v.} \int_\R e^{iq^2t} G(q) \frac{\psi(K-q)}{K-q}\,dq
\end{align}
where `$a$' is a bounded Lipschitz coefficient and $\psi \in \mathcal{S}$.
Assume that, for some $\alpha < 1/10$,
\begin{align}
\label{lemStatasG}
{\| G(t) \|}_{L^\infty} 
  + \langle t \rangle^{-\alpha} {\| G (t) \|}_{H^1} \leq 1.
\end{align}
Moreover, assume that one of the following two conditions hold:
\begin{align}\label{lemStatas0}
a(0) = 0 \qquad \mbox {or} \qquad G(0) = 0. 
\end{align}
Then, there exists $\rho \in (0,\alpha)$, such that, as $|t| \rightarrow \infty$,
\begin{align}\label{lemStatconc}
I(t,K) = i \pi \psi(0) \, a(K) G(K) \, \sign(2tK)  + \mathcal{O}(|t|^{-\rho}). 
\end{align}
 
\end{lem}

\begin{proof}
We first look at the case $G(0)=0$ and then discuss how to obtain the statement when $a(0)=0$,
which only requires minor modifications.

\smallskip
{\it Proof of \eqref{lemStatconc} when $G(0)=0$.}
We first prove the statement under the assumption that $G(0)=0$.
Without loss of generality we let $a \equiv 1$.
Let us denote by $\chi_1$ a smooth even cutoff function supported in a neighbourhood of zero, and let $\chi_2 := 1 - \chi_1$.
We decompose \eqref{lemStatint} as
\begin{subequations}\label{StatprI0}
\begin{align}
\nonumber
& I = I_0 + I_1 + I_2 + I_3,
\\
& I_0(t,K) := e^{-itK^2} 
	\mathrm{p.v.} \int_\R e^{iq^2t} G(q) \frac{\psi(K-q)}{K-q} \chi_1\big( (K-q) t^3 \big) \chi_1(q t^{3\alpha}) \, dq,
\\
& I_1(t,K) := e^{-itK^2} 
	\int_\R e^{iq^2t} G(q) \frac{\psi(K-q)}{K-q} \chi_2\big( (K-q) t^3 \big) \chi_1(q t^{3\alpha}) \, dq,
\\
& I_2(t,K) := e^{-itK^2} 
	\int_\R e^{iq^2t} G(q) \frac{\psi(K-q)}{K-q} \chi_2\big( (K-q) t^{1-10\alpha} \big) \chi_2(q t^{3\alpha}) \, dq,
\\
& I_3(t,K) := e^{-itK^2} \mathrm{p.v.} \int_\R e^{iq^2t} G(q) \frac{\psi(K-q)}{K-q} \chi_1\big( (K-q) t^{1-10\alpha} \big) \chi_2(q t^{3\alpha}) \, dq.
\end{align}
\end{subequations}


\smallskip
{\it Estimate of $I_0$}.
On the support of the integral $I_0$ we have $|K-q| \lesssim |t|^{-3}$.
Let $H(q;t) = e^{iq^2t} G(q)  \chi_1(q t^{3\alpha})$ and notice that, 
in view of \eqref{lemStatasG}, we have ${\| \partial_q H \|}_{L^2} \lesssim t$.
Then, using the principal value we have
\begin{align*}
\big| I_0(t,K) \big| = 
	\Big| \, \mathrm{p.v.} \int_\R \big[ H(q;t) - H(K;t) \big] \frac{\psi(K-q)}{K-q} \chi_1\big( (K-q) t^3 \big) \, dq \Big|
\\
\lesssim t \, \int_\R \frac{1}{|K-q|^{1/2}} \chi_1\big( (K-q) t^3 \big) \, dq \lesssim |t|^{-1/2}.
\end{align*}
Therefore, $I_0$ is a remainder term.

\smallskip
{\it Estimate of $I_1$}.
Using the assumption $G(0)=0$ and \eqref{lemStatasG} we have
\begin{align}\label{StatprG0}
|G(q)| \lesssim |q|^{1/2} |t|^{\alpha}.
\end{align}
We can then estimate
\begin{align*}
\big| I_1(t,K) \big| \lesssim |t|^\alpha \int_\R |q|^{1/2} \frac{1}{|K-q|} \chi_2\big( (K-q) t^3 \big) 
  \chi_1(q t^{3\alpha}) \, dq
  \lesssim |t|^{-\alpha/2} \log|t|. 
\end{align*}

\smallskip
{\it Estimate of $I_2$}.
On the support of $I_2$ we have $|q| \gtrsim |t|^{-3\alpha}$ and $|K-q| \gtrsim |t|^{-1+10\alpha}$
so that we can integrate by parts in $q$ using $e^{itq^2}  = (2itq)^{-1} \partial_q e^{itq^2}$:
\begin{align*}
\big| I_2(t,K) \big| & \lesssim |t|^{-1}
	\int_\R \Big| \partial_q \Big( G(q) \frac{\psi(K-q)}{K-q} \frac{1}{q}
	\chi_2\big( (K-q) t^{1-10\alpha} \big) \chi_2(q t^{3\alpha}) \Big) \Big| \, dq
	\lesssim |t|^{-5 \alpha}.
\end{align*}

\smallskip
{\it Asyptotics for $I_3$}.
The last term gives us the leading order asymptotics on the right-hand side of \eqref{lemStatconc}.
First, we notice that we can replace $\psi(K-q)$ by $\psi(0)$.
Then, we can replace $G(q)$ by $G(K)$ using the bound in \eqref{lemStatasG}
which implies $|G(q) - G(K)| \lesssim |q-K|^{1/2} {\|\partial_q G\|}_{L^2} \lesssim |q-K|^{1/2} |t|^\alpha$: 
\begin{align*}
& \Big| \int_\R e^{iq^2t} \big[ G(q) - G(K) \big] 
	\frac{1}{K-q} \chi_1\big( (K-q) t^{1-10\alpha} \big) \chi_2(q t^{3\alpha}) \, dq \Big|
	\\
	& \lesssim |t|^\alpha \int_\R 
	\frac{1}{|K-q|^{1/2}} \, \chi_1\big( (K-q) t^{1-10\alpha} \big) \, dq
	\lesssim |t|^\alpha \cdot (|t|^{-1+10\alpha})^{1/2},
\end{align*}
which can be absorbed in the remainder.
We then have
\begin{align}\label{Statpr5}
\begin{split}
I_3 & = e^{-itK^2} G(K) \psi(0) \, \pv \int_\R e^{iq^2t} \frac{1}{K-q} 
  \chi_1\big( (K-q) t^{1-10\alpha} \big) \chi_2(q t^{3\alpha}) \, dq + \mathcal{O}(|t|^{-1/3})
\\
& = G(K) \psi(0) \, \pv \int_\R e^{2itK q}
  \chi_1\big( q t^{1-10\alpha} \big) \chi_2\big( (q+K) t^{3\alpha} \big) \,  \frac{1}{q} \, dq
  + O(|t|^{-1/3}),
\end{split}
\end{align}
having first changed variables $q \mapsto q+K$ and then replaced $e^{itq^2}$ by $1$
using that $|q| \lesssim |t|^{-1+10\alpha}$ in the last integral above.
Note that the contribution in \eqref{Statpr5} vanishes if $|K| \leq c |t|^{-3\alpha}$ for some small $c>0$.
Letting 
$$F(q) = F(q;t,K):= \chi_1\big( q t^{1-10\alpha} \big) \chi_2\big( (q+K) t^{3\alpha} \big)$$
and using the identities
\begin{align}\label{fourierids}
\what{f g} = \frac{1}{\sqrt{2\pi}} \what{f} \ast \what{g},
\qquad \sign \, x = i\sqrt{\frac{2}{\pi}} \whF \pv \frac{1}{\xi},
\end{align}
from \eqref{Statpr5} we write
\begin{align*}
I_3 & = G(K) \psi(0) \, \sqrt{2\pi} \, \whF \Big(  F(q;t,K) \, \pv \frac{1}{q} \Big) (-2tK)
\\
& = G(K) \psi(0) \, \Big( \what{F} \ast \, (-i) \sqrt{\frac{\pi}{2}} \sign \Big) (-2tK)
\\
& = G(K) \psi(0) i \sqrt{\frac{\pi}{2}} \sign(2tK) \, \big(\int \widehat{F} d\xi \big) 
  + \mathcal{O}\big( |tK|^{-1/2} {\big\| |\xi|\widehat{F} \big\|}_{L^2}\big).
\end{align*}
Since ${\| \partial_qF \|}_{L^2} \lesssim (|t|^{1-10\alpha})^{1/2}$, and $|K| \gtrsim |t|^{-3\alpha}$
the remainder is acceptable, while the leading order term
matches the desired right-hand side of \eqref{lemStatconc}
since $F(0) = \chi_2\big(Kt^{3\alpha} \big)$ and \eqref{StatprG0} holds.

\smallskip
{\it Proof of \eqref{lemStatconc} when $a(0)=0$.}
We can follow the same proof given above starting from the
formulas in \eqref{StatprI0} with the additional coefficient $a(K)$ in front.
The only term that cannot be treated identically is $I_1$, for which we used \eqref{StatprG0}.
In the case $G(0)\neq 0$, but $a(0)=0$, we can still obtain a similar bound to the previous one
by estimating
\begin{align*}
\big| I_1(t,K) \big| \lesssim |a(K)| {\|G\|}_{L^\infty} \int_\R  \frac{1}{|K-q|} \chi_2\big( (K-q) t^3 \big) 
  \chi_1(q t^{3\alpha}) \, dq
  \lesssim |a(K)| \log |t|. 
\end{align*}
This gives a sufficient bound if $|K| \lesssim |t|^{-\rho-}$. 
When instead $|K| \gtrsim |t|^{-\rho-} \gg |t|^{-3\alpha}$ we also have $|K-q| \gtrsim |t|^{-\rho-}$
on the support of the integral and can estimate it by
\begin{align*}
\big| I_1(t,K) \big| \lesssim |a(K)| {\|G\|}_{L^\infty} \int_\R  |t|^{\rho+}
  \chi_1(q t^{3\alpha}) \, dq
  \lesssim |t|^{-3\alpha + \rho+}, 
\end{align*}
which is sufficient for \eqref{lemStatconc}. 

\end{proof}

\subsubsection{Proof of \eqref{fODE}}
Recall that we can write
$\mathcal{N}_S = \mathcal{N}_+ + \mathcal{N}_-$ with
\begin{align}\label{NSasy}
& \mathcal{N}_\iota(t,k) = \frac{1}{(2\pi)^{2}} \iiint e^{it(-k^2+\ell^2-m^2+n^2)} 
	\wt{f}(t,\ell) \overline{\wt{f}}(t,m) \wt{f}(t,n) \mu_\iota(k,\ell,m,n)\,dndmd\ell,
\\
\nonumber
& \mu_\iota(k,\ell,m,n) 
= \sum_{\e_0,\e_1,\e_2,\e_3 \in \{+,-\}}\overline{b_\iota^{\e_0}(k)} \,b_\iota^{\e_1}(\ell) \,
  \overline{b_\iota^{\e_2}(m)} \,b_\iota^{\e_3}(n) \, \sqrt{2\pi} \,
  \widehat{\varphi_\iota}(\e_0k-\e_1\ell+\e_2m-\e_3n),
\end{align}
where $\varphi_{\iota}=\chi_{\iota}^{4}, \iota \in \{+,-\}$.
From the formulas \eqref{chi+-} we have
\begin{align}\label{whatvarphi}
& \widehat{\varphi_\iota}(k) = \sqrt{\frac{\pi}{2}}\delta_{0}(k) 
  + \iota \pv \frac{\what{\zeta}(k)}{ik} + \widehat{\varpi_\iota}(k),
\end{align}
for some even $\zeta \in C_c^\infty$.
In what follows we disregard $\varpi_\iota$ which gives faster decaying remainder terms like $\mathcal{N}_R$
in Subsection \ref{ssecRasy}.

We denote the nonlinear terms corresponding to the $\delta_0$ and $\pv$ in \eqref{whatvarphi}
by $\mathcal{N}_\iota^\delta$ and $\mathcal{N}_\iota^{\mathrm{p.v.}}$ respectively, so that 
\begin{align}\label{Nsplit}
\mathcal{N}_S = \mathcal{N}_+^\delta + \mathcal{N}_-^\delta 
  + \mathcal{N}_+^{\mathrm{p.v.}} + \mathcal{N}_-^{\mathrm{p.v.}}
\end{align}
where, after changing variables $\ell \mapsto p = \e_0k-\e_1\ell+\e_2m-\e_3n$ in \eqref{NSasy}, 
we can write
\begin{align}\label{Nddef}
\begin{split}
& \mathcal{N}_\iota^{\delta}(t,k) := \frac{1}{(2\pi)^{3/2}} \sum_{\e_0,\e_1,\e_2,\e_3 \in \{+,-\}}
	\iiint e^{it(-k^2+(\e_0k-p+\e_2m-\e_3n)^2-m^2+n^2)} 
	\\ & \times \overline{b_\iota^{\e_0}(k)} 
	\, b_\iota^{\e_1}(\e_1(\e_0k-p+\e_2m-\e_3n)) \,
	\overline{b_\iota^{\e_2}(m)} \, b_\iota^{\e_3}(n) \, 
	\\
	& \times
	\wt{f}(t,\e_1(\e_0k-p+\e_2m-\e_3n)) \overline{\wt{f}}(t,m) \wt{f}(t,n) \, 
	\sqrt{\frac{\pi}{2}}\delta_{0}(p)\,dn dm dp,
\end{split}
\end{align}
and
\begin{align}\label{Npvdef}
\begin{split}
& \mathcal{N}_\iota^{\mathrm{p.v.}}(t,k) = \frac{\iota}{(2\pi)^{3/2}} \sum_{\e_0,\e_1,\e_2,\e_3 \in \{+,-\}}
	\iiint e^{it(-k^2+(\e_0k-p+\e_2m-\e_3n)^2-m^2+n^2)} 
	\\ & \times \overline{b_\iota^{\e_0}(k)} 
	\, b_\iota^{\e_1}(\e_1(\e_0k-p+\e_2m-\e_3n)) \,
  \overline{b_\iota^{\e_2}(m)} \,b_\iota^{\e_3}(n) \,
\\
& \times
	\wt{f}(t,\e_1(\e_0k-p+\e_2m-\e_3n)) \overline{\wt{f}}(t,m) \wt{f}(t,n) \, 
	\mathrm{p.v.}\,\frac{\hat{\zeta}(p)}{ip} \,dn dm dp.
\end{split}
\end{align}

\smallskip
{\it Asymptotics for the $\pv$ terms}.
We apply Lemma \ref{AsLem1} to the integral \eqref{Npvdef}
using \eqref{eq:atildef} to verify the assumption \eqref{AsLem1as}:
from the conclusion \eqref{I2}, and changing variables $p \mapsto \e_0k-p$, we obtain
\begin{align}\label{Npveps}
\begin{split}
\mathcal{N}_\iota^{\mathrm{p.v.}}(t,k) & = \frac{1}{(2\pi)^{3/2}}  \frac{\pi}{|t|} e^{-itk^2} 
	\sum_{\e_0,\e_1,\e_2,\e_3 \in \{+,-\}} \mathcal{N}_{\iota,\e_0,\e_1,\e_2,\e_3}^{\mathrm{p.v.}}(t,k),
\\
\mathcal{N}_{\iota,\e_0,\e_1,\e_2,\e_3}^{\mathrm{p.v.}}(t,k) & :=
	 \iota\, \overline{b_\iota^{\e_0}(k)} \int e^{itp^2} \, b_\iota^{\e_1}(\epsilon_1p) \,
  \overline{b_\iota^{\e_2}(\epsilon_2p)} \,b_\iota^{\e_3}(\epsilon_3 p) \,
\\
& \times
	\wt{f}(t,\epsilon_1p) \overline{\wt{f}}(t,\epsilon_2p) \wt{f}(t,\epsilon_3 p) \, 
	\mathrm{p.v.}\, \frac{\what{\zeta}(\eps_0k - p)}{i(\eps_0k - p)} \, dp 
	 +  \mathcal{O}(\varepsilon^3 \jt^{-1-\rho}).
\end{split}
\end{align}
In the remaining of the proof below we will write $\approx$ to indicate equality up to acceptable 
$\mathcal{O}(\varepsilon^3 \jt^{-1-\rho})$ terms.

Notice that when $\epsilon = -$, the coefficients $b_\iota^\epsilon$ vanish at $0$, see \eqref{Ksing+'}-\eqref{Ksing-'}.
We can then apply Lemma \ref{lemStat} to \eqref{Npveps} under the first, resp. second, 
assumption in \eqref{lemStatas0} when $\eps_0=-$, resp. $\eps_1$ or $\eps_2$ or $\eps_3=-$,
and using the identity \eqref{afLip} and \eqref{eq:atildef} to verify \eqref{lemStatasG}.
The conclusion \eqref{lemStatconc}, $\what{\zeta}(0)= (2\pi)^{-1/2}$,
and the formulas \eqref{Npveps}, then give us
\begin{align}
\label{Npv1}
& \mathcal{N}_+^{\mathrm{p.v.}}(t,k) + \mathcal{N}_-^{\mathrm{p.v.}}(t,k)
  = \mathcal{N}^{(1)}(t,k) + \mathcal{N}^{(2)}(t,k)
  + \mathcal{O}(\varepsilon^3 \jt^{-1-\rho}),
\\
\label{N1}
\begin{split}
\mathcal{N}^{(1)}(t,k) := \frac{1}{4|t|}  e^{-itk^2} 
  \sum_{\substack{\iota \in \{+,-\} \\ (\e_0,\e_1,\e_2,\e_3) \neq (+,+,+,+)} } 
  \iota \, \overline{b_\iota^{\e_0}(k)} b_\iota^{\e_1}(\epsilon_1k) \,
  \overline{b_\iota^{\e_2}(\epsilon_2k)} b_\iota^{\e_3}(\epsilon_3k) \,
  \\
  \times \wt{f}(t,\epsilon_1k) \overline{\wt{f}}(t,\epsilon_2k) \wt{f}(t,\epsilon_3k) \,\sign(2tk),
\end{split}
\\
\label{N2}
\begin{split}
\mathcal{N}^{(2)}(t,k) := 
  \frac{1}{(2\pi)^{3/2}} \frac{\pi}{|t|}  e^{-itk^2} \sum_{\iota \in \{+,-\}} \iota \, \overline{b_\iota^+(k)}
  \int e^{itp^2} |b_\iota^+(p)|^2 \, 
  b_\iota^+(p) \,
  \\
  \times |\wt{f}(t,p)|^2 \wt{f}(t,p) \, \mathrm{p.v.} \frac{\what{\zeta}(k - p)}{i(k - p)} \, dp.
\end{split}
\end{align}

To simplify further \eqref{N2} we combine the $\iota=+$ and $-$ contributions,
and to write them as (recall also \eqref{afLip})
\begin{align}\label{N2'}
& \mathcal{N}^{(2)}(t,k) = \mathcal{N}^{(2)}_1(t,k) + \mathcal{N}^{(2)}_2(t,k),
\\
\begin{split}\label{N21}
\mathcal{N}^{(2)}_1(t,k) := \frac{1}{(2\pi)^{3/2}} \frac{\pi}{|t|} e^{-itk^2} 
  \big[ \overline{b_+^+(k)} + \overline{b_-^+(k)} \big] 
  \int e^{itp^2} |b_+^+(p)|^2b_+^+(p) \sign(p)
  \\
  \times |f^\sharp (t,p)|^2 f^\sharp(t,p) \, \mathrm{p.v.} \frac{\what{\zeta}(k - p)}{i(k - p)} \, dp,
\end{split}
\\
\begin{split}\label{N22}
\mathcal{N}^{(2)}_2(t,k) := -\frac{1}{(2\pi)^{3/2}} \frac{\pi}{|t|} e^{-itk^2} 
  \overline{b_-^+(k)} \int e^{itp^2} \Big[ |b_+^+(p)|^2b_+^+(p) \sign(p)
  \\ + |b_-^+(p)|^2 b_-^+(p) \sign(p) \Big] \,
  |f^\sharp (t,p)|^2 f^\sharp(t,p) \, \mathrm{p.v.} \frac{\what{\zeta}(k - p)}{i(k - p)} \, dp.
\end{split}
\end{align}
The main observation is that the expressions \eqref{N21} and \eqref{N22} above
satisfy the assumptions of Lemma \ref{lemStat}.
More precisely, 
we see from \eqref{Ksing+'}-\eqref{Ksing-'} that
\begin{align}
b^+_+(k) + b_-^+(k) = \mathbf{1}_{+}(k)\big( T(k) + 1 \big) + \mathbf{1}_{-}(k) \big(1 +T(-k) \big),
\end{align}
which is Lipschitz and vanishes at zero;
therefore \eqref{N21} is of the form \eqref{lemStatint} and satisfies the first assumption in \eqref{lemStatas0},
and the assumption \eqref{lemStatasG} (see \eqref{eq:atildef}).
Similarly, we have
\begin{align}
\begin{split}
& |b_+^+(p)|^2 b_+^+(p) \sign(p) + |b_-^+(p)|^2 b_-^+(p) \sign(p)
\\
& = \big(\mathbf{1}_{+}(p)|T(p)|^2 + \mathbf{1}_{-}(p)\big) (\mathbf{1}_{+}(p)T(p) - \mathbf{1}_{-}(p))
\\
& + \big(\mathbf{1}_{+}(p) + \mathbf{1}_{-}(p)|T(p)|^2 \big) (\mathbf{1}_{+}(p) - \mathbf{1}_{-}(p)T(-p))
\end{split}
\end{align}
which is also Lipschitz and vanishes at $p=0$ since $T(0)=-1$;
therefore, \eqref{N22} is of the form \eqref{lemStatint}, satisfies the second assumption in \eqref{lemStatas0},
as well as \eqref{lemStatasG} in view of the a priori bounds on $f^\sharp$ (or \eqref{eq:atildef}).
It follows that 
\begin{align}
\begin{split}\label{N22asy}
\mathcal{N}^{(2)}(t,k) \approx 
  \frac{1}{4|t|} \sum_{\iota \in \{+,-\}} \iota |\overline{b_\iota^+(k)}|^4
  |\wt{f}(t,k)|^2 \wt{f}(t,k) \,\sign(2tk).
\end{split}
\end{align}

\begin{rem}[The case of an even resonance]\label{remN2}
In the case of an even resonance, we can arrive at \eqref{Npv1}-\eqref{N2} as done above. 
The cancellations in the expression for $\mathcal{N}^{(2)}(t,k)$, however, need to be seen
slightly differently from \eqref{N2'}-\eqref{N22} since in the case of an even resonance one has $T(0)=1$;
also we do not need to use $\mathcal{F}^\sharp$.
More precisely, one can write (note the change of one sign in \eqref{N21e} compared to \eqref{N21},
and of two signs in \eqref{N22e} compared to \eqref{N22})
\begin{align}
\nonumber
& \mathcal{N}^{(2)}(t,k) = \mathcal{N}^{(2)}_1(t,k) + \mathcal{N}^{(2)}_2(t,k),
\\
\begin{split}\label{N21e}
\mathcal{N}^{(2)}_1(t,k) := \frac{1}{(2\pi)^{3/2}} \frac{\pi}{|t|} e^{-itk^2} 
  \big[ \overline{b_+^+(k)} - \overline{b_-^+(k)} \big] 
  \int e^{itp^2} |b_+^+(p)|^2b_+^+(p) 
  \\
  \times |\wt{f} (t,p)|^2 \wt{f}(t,p) \, \mathrm{p.v.} \frac{\what{\zeta}(k - p)}{i(k - p)} \, dp,
\end{split}
\\
\begin{split}\label{N22e}
\mathcal{N}^{(2)}_2(t,k) := \frac{1}{(2\pi)^{3/2}} \frac{\pi}{|t|} e^{-itk^2} 
  \overline{b_-^+(k)} \int e^{itp^2} \Big[ |b_+^+(p)|^2b_+^+(p) 
  \\ - |b_-^+(p)|^2 b_-^+(p) 
  \Big] \,
  |\wt{f} (t,p)|^2 \wt{f}(t,p) \, \mathrm{p.v.} \frac{\what{\zeta}(k - p)}{i(k - p)} \, dp.
\end{split}
\end{align}
Then, from \eqref{Ksing+'}-\eqref{Ksing-'} we have
$b_+^+(k) - b_-^+(k) = \mathbf{1}_+( T(k) - 1) + \mathbf{1}_-(1-T(-k))$,
which is Lipschitz and vanishes at zero. Therefore, we can then proceed as above to obtain
\eqref{N22asy}.
\end{rem}

Eventually, for $t > 0$, we combine this with \eqref{N1} and an algebraic calculation to arrive at
\begin{align}\label{Npvasy5}
\begin{split}
\mathcal{N}_+^{\mathrm{p.v.}}(t,k) & + \mathcal{N}_-^{\mathrm{p.v.}}(t,k) 
\\
\approx \frac{1}{4|t|} \Big[ & - {T(-k)} \mathcal{N}^+[f](k) \mathbf{1}_+(k)
  + |\wt{f}(k)|^2 \wt{f}(k)  \mathbf{1}_-(k) -  R_+(k)  \mathcal{N}^+[f](-k) \mathbf{1}_-(k)
\\
& + |\wt{f}(k)|^2 \wt{f}(k)  \mathbf{1}_+(k) 
  - {R_-(-k)} \mathcal{N}^-[f](k) \mathbf{1}_+(k)
  - T(k) \mathcal{N}^-[f](-k) \mathbf{1}_-(k) \Big]
\end{split}
\end{align}
where
\begin{align}\label{N+-}
\begin{split}
\mathcal{N}^+[f](k) & := \big|T(k)\wt{f}(k) + R_+(k)\wt{f}(-k)\big|^2 
  \big( T(k)\wt{f}(k) + R_+(k)\wt{f}(-k) \big),
\\
\mathcal{N}^-[f](k) & := \big|T(k)\wt{f}(-k) + R_-(k)\wt{f}(k)\big|^2 
  \big( T(k)\wt{f}(-k) + R_-(k)\wt{f}(k) \big).
\end{split}
\end{align}


\smallskip
{\it Asymptotics for the $\delta_0$ terms}.
We now look at 
the contributions from \eqref{Nddef}.
The analysis here is standard 
and, for $\iota\in\{+,-\}$, we have, up to $\mathcal{O}( \varepsilon^3 |t|^{-1-\rho})$ remainders,
\begin{align}
\label{Nd+}
\begin{split}
\mathcal{N}_+^{\delta} & \approx \frac{1}{4|t|} 
  \Big[ T(-k) \mathcal{N}^+[f](k) \mathbf{1}_+(k)
  + \Big(|\wt{f}(k)|^2 \wt{f}(k) +  R_+(k) \mathcal{N}^+[f](-k) \Big) \mathbf{1}_-(k) \Big]
\\
\mathcal{N}_-^{\delta} & \approx \frac{1}{4|t|}  
  \Big[\Big(|\wt{f}(k)|^2 \wt{f}(k) + {R_-(-k)} \mathcal{N}^-[f](k) \Big) \mathbf{1}_+(k)
  + {T(k)} \mathcal{N}^-[f](-k) \mathbf{1}_-(k) \Big].
\end{split}
\end{align}

\smallskip
{\it Conclusion}.
From \eqref{Nsplit}, putting together \eqref{Npvasy5} and \eqref{Nd+} 
and canceling terms, we obtain
\begin{align}\label{Nasyconc}
\begin{split}
\mathcal{N}_S 
  & = \mathcal{N}_+^{\delta}(t,k)+\mathcal{N}_+^{\mathrm{p.v.}}(t,k)
  +\mathcal{N}_-^{\delta}(t,k) + \mathcal{N}_-^{\mathrm{p.v.}}(t,k)
\\
& = \frac{1}{4|t|} 
  \Big(2|\wt{f}(k)|^2 \wt{f}(k)\Big)
  +  \mathcal{O}(\varepsilon^3 |t|^{-1-\rho})
\\
& = \frac{1}{2t}|\wt{f}(k)|^2 \wt{f}(k) 
  +  \mathcal{O}(\varepsilon^3 |t|^{-1-\rho}).
\end{split}
\end{align}
In view of \eqref{ODE'} and \eqref{ODEregular}, and $\wt{f}(k) = \sign(k) f^\sharp(k)$, 
we have arrived at \eqref{fODE}.




	\bigskip
	
\end{document}